\documentclass[12pt]{article}
\usepackage{times}
\usepackage{amsfonts}
\usepackage{amssymb}
\usepackage{amsmath}
\usepackage{amsthm}
\usepackage{algpseudocode}
\usepackage{algorithm}
\usepackage{color}
\usepackage{enumitem}
\usepackage{graphicx}
\usepackage{cite}
\usepackage{enumitem}
\usepackage{authblk}
\usepackage{mathtools}

\newtheorem{theorem}{Theorem}
\newtheorem{lemma}[theorem]{Lemma}
\newtheorem{proposition}[theorem]{Proposition}

\newtheorem{definition}[theorem]{Definition}

\newcommand{\real}{{\mathbb R}}

\newcommand{\realinf}{\real \cup \{+\infty\}}

\newcommand{\nat}{\mathbb{N}}

\newcommand{\inner}[2]{\left\langle#1,#2\right\rangle}
\newcommand{\biginner}[2]{\big\langle{#1},{#2}\big\rangle}

\newcommand{\toset}{\rightrightarrows}

\DeclareMathOperator{\dom}{dom}
\DeclareMathOperator{\ri}{ri}

\DeclareMathOperator{\diag}{diag}

\DeclareMathOperator{\conv}{conv}

\DeclareMathOperator*{\argmin}{arg\,min}

\DeclareMathOperator{\suchthat}{ST}

\newcommand{\set}[2]{\left\{#1 \; \left|\;\; #2 \right.\right\}}

\newcommand{\norm}[1]{\left\|#1\right\|}
\newcommand{\smallnorm}[1]{\lVert{#1}\rVert}

\DeclareMathOperator{\proj}{proj}
\newcommand{\orthog}{^{\perp}}
\newcommand{\conjspace}[1]{^{\perp_{#1}}}
\newcommand{\half}{\tfrac{1}{2}}

\DeclareMathOperator{\dist}{dist}
\DeclareMathOperator{\prox}{prox}

\DeclareMathOperator{\graph}{graph}

\newcommand{\transpose}{^{\scriptscriptstyle \top}}
\newcommand{\invtrans}{^{-\scriptscriptstyle\top}}

\newcommand{\dotbar}[1]{\dot{\bar{#1}}}

\newcommand{\spcdot}{\,\cdot\,}

\DeclareMathOperator{\im}{im}
\DeclareMathOperator{\identity}{Id}

\DeclareMathOperator{\hyphenprox}{-prox}
\newcommand{\sprox}[1]{{#1}\!\hyphenprox}

\DeclareMathOperator{\hyphenproj}{-proj}
\newcommand{\sproj}[1]{{#1}\!\hyphenproj}

\DeclareMathOperator{\hyphendist}{-dist}
\newcommand{\sdist}[1]{{#1}\!\hyphendist}

\mathchardef\mhyphen="2D

\newcommand{\approxmin}[2]{{#1}\mhyphen\!\argmin_{#2}\,}

\newcommand{\Rho}{\mathrm{P}}

\DeclareMathOperator{\last}{last}

\makeatletter
\newcommand{\svdots}{%
  \vbox{\fontsize{\sf@size}{\sf@size pt}\linespread{0.3}\selectfont
    \kern0.2\baselineskip
    \hbox{.}\hbox{.}\hbox{.}%
    \kern0.1\baselineskip
  }%
}
\makeatother

\newcommand{\LongState}[1]{\State \parbox[t]{\dimexpr\linewidth-\algorithmicindent}{#1\strut}}

\begin{document}

\title{Using Subproblem Objective Gaps in Inexact Augmented
       Lagrangian and ADMM Algorithms, with Applications to
       Stochastic Mixed Integer Programming\thanks{This work was funded in part by the
       U.S.~Office of Naval Research grant N00014-24-1-2403.}}
\author[1]{Jonathan Eckstein}
\affil[1]{\small Department of Management Science and Information Systems,
          Rutgers Business School Newark and New Brunswick,
          Rutgers University}
\date{\today}

\maketitle

\vspace{-3ex}

\begin{abstract}
Through a ``partial strong convexity'' lemma, this paper shows how bounds on
subproblem objective value suboptimality can be used in inexact augmented
Lagrangian methods and ADMM algorithms.  The ADMM result uses a small but
important refinement on a long-standing criterion for approximately solving
subproblems. The results enable two new approaches to computing Lagrangian
bounds on the optimal values of stochastic mixed-integer programming problems,
with simpler convergence analysis than the prior state of the art. In each
case, the subproblems are solved by variants of the classical
Frank-Wolfe algorithm.  However, as compared to prior methods of the same
type, there is much more freedom in the choice of Frank-Wolfe variant.
\end{abstract}

\section{Introduction}
\label{sec:introduction}
This paper establishes that, for both the augmented Lagrangian method (ALM)
and the alternating direction method of multipliers (ADMM), it is possible to
use an objective-value tolerance criterion to accept inexact subproblem
solutions.  While such a criterion cannot always provide a bound on the
distance to the exact subproblem solution as specified in long-standing
results such as~\cite{EckBer92}, it makes it possible, through the ``partial
strong convexity'' lemma in Section~\ref{sec:partialStrong} below, to bound
the error in the multiplier update, which is sufficient to prove standard
convergence results if the subproblem objective errors are appropriately
controlled. Sections~\ref{sec:alms} and \ref{sec:admms} respectively prove
convergence of versions of the classical ALM and ADMM that use objective
tolerances for their subproblems, in the ADMM case using a small but critical
sharpening of the subproblem approximation criterion first proposed
in~\cite{EckBer92}.  To obtain convergence of the new ALM and ADMM variants,
the square roots of the subproblem objective errors should form a summable
sequence.

The ALM results may be viewed as significantly generalizing the ``$(A'')$''
inexact subproblem criterion for inequality-constrained problems analyzed by
Rockafellar in the seminal paper~\cite{Roc76b}.  For the ADMM, the results
appear to be new.

A practical application of the objective-gap-based approximation results
consists of ALM and ADMM methods that
solve their subproblems with variants of the Frank-Wolfe (FW) method, also
known as the conditional gradient method, since FW algorithms commonly provide
bounds on the difference between the objective value of the current iterate
and the optimal solution (and rarely any other measures of their progress
toward optimality).  The principal current example of such a combination of
algorithms is for computing Lagrangian bounds on mixed-integer stochastic
programming problems, as proposed in~\cite{BCDELL17}.  That algorithm uses the
fully corrective FW method~\cite{Hol74,vHoh77,LJJ15}, also known as simplicial
decomposition, to solve subproblems within the progressive hedging (PH)
algorithm~\cite{RW91}. The objective-tolerance inexact ALM and ADMM methods
developed here permit development of similar algorithms for the same
application, but with a simpler convergence theory that does not require the
recourse assumption imposed in~\cite{BCDELL17}.  Unlike the algorithm
in~\cite{BCDELL17}, these methods can use any variant of FW that provides an
objective gap, rather than being restricted to the fully corrective version.
The updated methods proposed here also allow the MILP subproblems to be solved
to gradually tightening objective gaps, rather than exactly.  The convergence
analysis here also allows, through linear changes of variables, for
variable-by-variable proximal/penalty parameter variations that are important
for the practical performance of PH-class methods.  When an FW method is used
on the subproblems, Section~\ref{sec:stochProgFWALM} also shows that an
ALM-based method may be as practical to implement as one based on PH.

\section{Foundations}
This section reviews prior results needed for the later analysis and assumes
basic familiarity with monotone operators on $\real^n$; see for
example~\cite{Roc70book,Roc76a,BauComBook}, with the original citation
being~\cite{Min62}.  For any maximal monotone operator
$T:\real^n\toset\real^n$ and scalar $c > 0$, we define the resolvent or
proximal mapping $\prox_{cT} : \real^n \to \real^n$ by
$\prox_{cT}(r)$ being the unique vector $x\in\real^n$ such
that there exists $y\in T(x)$ with $x + cy = r$.  This mapping may also be
written $\prox_{cT} = (\identity + cT)^{-1}$.

\subsection{Foundational algorithms}

The following potentially inexact and overrelaxed versions of the proximal
point algorithm~\cite{Roc76a} and Douglas-Rachford (DR) splitting method for
monotone operators~\cite{LioMer79} developed in~\cite{EckBer92} form the
basis for the later analysis in this work.

\begin{theorem}[a generalized proximal point algorithm]
\label{thm:genPPA}
Let $T:\real^n\toset\real^n$ be a maximal monotone operator, and suppose that 
\begin{enumerate}[label={(\roman*)}, nosep]
\item \label{item:thmGenPPA:stepsizes} $\{c_k\}_{k=0}^\infty \subset
\real_{++}$ a sequence of positive scalars with $\inf_{k\geq 0} \{ c_k \} > 0$
\item \label{item:thmGenPPA:relax} $\{\nu_k\}_{k=0}^\infty$ is a sequence of
real numbers with $\inf_{k \geq 0} \{\nu_k\} > 0$ and $\sup_{k \geq 0}
\{\nu_k\} < 2$
\item \label{item:thmGenPPA:errors} $\{\epsilon_k\}_{k=0}^\infty$ be a sequence of
nonnegative real numbers with $\sum_{k=0}^\infty \epsilon_k < \infty$.
\end{enumerate}
Starting from an arbitrary $x^0\in \real^n$, suppose that
$\{x^k\}_{k=0}^\infty,\{e^k\}_{k=0}^\infty \subset \real^n$ conform for
all $k \geq 0$ to the recursion
\begin{align} \label{genppa}
x^{k+1} &= \nu_k \prox_{c_k T}(x^k) + (1-\nu_k) x^k + e^k, && \text{where} &
\norm{e^k} &\leq \epsilon_k.
\end{align}
Then if $T^{-1}(0) \neq \emptyset$, the sequence $\{x^k\}$ converges to some
$x^*\in\real^n$ such that $0\in T(x^*)$.  
If $T^{-1}(0) = \emptyset$, then
$\{x^k\}$ must be an unbounded sequence.
\end{theorem}

\begin{theorem}[a generalized DR splitting algorithm]
\label{thm:genDR}
Let $A,B:\real^n\toset\real^n$ be two maximal monotone operators, and, for
arbitrary $p^0,z^0\in\real^n$, let $\{p^k\}_{k=0}^\infty,
\{z^k\}_{k=0}^\infty, \{q^k\}_{k=1}^\infty, \{r^k\}_{k=1}^\infty \subset
\real^n$ be sequences evolving according to the recursions, for all $k \geq
0$,
\begin{align}
q^{k+1} &= \prox_{cA}(p^k - c z^k) + e_q^{k} \label{drcycleA} \\
r^{k+1} &= \nu_k q^{k+1} + (1-\nu_k) p^k + c z^k \label{drcycleRelax} \\
p^{k+1} &= \prox_{cB}(r^{k+1}) + e_p^{k} \label{drcycleB} \\
z^{k+1} &= \frac{1}{c}(r^{k+1} - p^{k+1}) \label{drcycleBdual},
\end{align}
where
\begin{itemize}[nosep]
\item $c > 0$ is a fixed scalar,
\item $\{\nu_k\}_{k=0}^\infty \subset \real$ is a sequence such that $\inf_k
\nu_k > 0$ and $\sup_k \nu_k < 2$, and
\item $\{e_q^k\}_{k=0}^\infty, \{e_p^k\}_{k=0}^\infty \subset \real^n$ 
are sequences such that $\sum_{k=0}^\infty \smallnorm{e_q^k} < \infty$ and
$\sum_{k=0}^\infty \smallnorm{e_p^k} < \infty$.
\end{itemize}
Then, if a solution to the inclusion $0 \in A(p) + B(p)$ exists, then
$\{p^k\}$ and $\{q^k\}$ converge to some $p^*$ such that $0 \in
A(p^*)+B(p^*)$, while $\{z^k\}$ converges to some $z^* \in B(p^*)$ such that
$-z^* \in A(p^*)$.  

On the other hand, if no solution to $0 \in A(p) + B(p)$ exists, then at least
one of the sequences $\{p^k\}$ or $\{z^k\}$ must be unbounded.
\end{theorem}

\noindent Both of these convergence results are proved in~\cite{EckBer92}, with some
notation differences.  They respectively generalize foundational results
in~\cite{Roc76a} and~\cite{Gab83}, which lacked the relaxation factors $\nu_k$
and the possibility of inexact computation of the iterates.

\subsection{Monotone operators under linear changes of variables}

For the algorithms to follow, linear changes of variables for monotone
operators will also prove useful: letting $T : \real^n \toset \real^n$ be a
maximal monotone operator and $Q$ be any invertible $n \times n$ matrix, the
point-to-set operator $Q\transpose T Q : \real^n \toset \real^n$ is defined by
\begin{align}
       (\forall\, \bar x \in \real^n) \quad 
          Q\transpose T Q(\bar x) = \set{Q\transpose y}{y \in T(Q \bar x)},
\end{align}
from which it is readily seen that 
\begin{align}
\graph(Q\transpose T Q) 
    &= \set{(Q^{-1} x, Q\transpose y)}{(x,y)\in\graph T}. \label{changecoordgraph}
\end{align}
This operator corresponds to $T$ represented in a different coordinate system
in which each input vector $x$ to $T$ has its representation changed to $\bar
x \doteq Q^{-1} x$ and each output vector has its representation changed to
$\bar y \doteq Q\transpose y$ (for the remainder of this section, bar accents
indicate vectors in the altered coordinate system).

\begin{proposition}[linear changes of variables for a monotone operator]
\label{prop:monotoneChangeCoords}
Suppose $T: \real^m\toset\real^m$ is a maximal monotone operator and that $Q$
is any $n\times n$ invertible matrix.  Then the operator $Q\transpose T Q$ is
maximal monotone.  The roots of $Q\transpose T Q$ take the form $Q^{-1} x$,
where $x$ is a root of $T$.
\end{proposition}
\begin{proof}
Consider any $(\bar x,\bar y),(\bar x',\bar y')\in\graph(Q\transpose T Q)$.
Then $\bar y=Q\transpose y$ for some $y\in T(Q\bar x)$ and $\bar
y'=Q\transpose y'$ for some $y'\in T(Q\bar x')$.  We then have
\begin{equation*}
\inner{\bar x-\bar x'}{\bar y-\bar y'} 
  = \inner{\bar x-\bar x'}{Q\transpose y - Q\transpose y'}
                   = \inner{Q\bar x-Q\bar x'}{y-y'}
                   \geq 0,
\end{equation*}
where the inequality follows because $T$ is monotone, $y\in T(Q\bar x)$, and
$y'\in T(Q\bar x')$.  By the arbitrary choice of $(\bar x,\bar y),(\bar
x',\bar y')\in\graph(Q\transpose T Q)$, one may conclude that $Q\transpose T
Q$ is monotone.

Now let $u,v\in\real^n$ be any vectors such that $\inner{\bar x-u}{\bar y-v}
\geq 0$ for all $(\bar x,\bar y) \in \graph Q\transpose T Q$. Using the
expression for $\graph(Q\transpose T Q)$ in~\eqref{changecoordgraph}, an
equivalent condition is that
\begin{align*}
\big(\forall\,(x,y)\in\graph T\big) 
&& 0 \leq \inner{Q^{-1}x - u}{Q\transpose y - v} 
&= (Q^{-1}x - u)\transpose Q\transpose  Q\invtrans (Q\transpose y - v) &&&&&& \\
&&&= \big( Q (Q^{-1}x - u) \big)\transpose 
     \big( Q\invtrans(Q\transpose y - v)\big) \\
&&&= (x - Qu)\transpose ( y - Q\invtrans v),
\end{align*}
where $Q\invtrans$ denotes the inverse transpose, $Q\invtrans =
(Q\transpose)^{-1}$. Since $T$ is maximal, the last version of the condition
above implies $\big(Qu,Q\invtrans v\big) \in \graph T$.
Using~\eqref{changecoordgraph} once again, having $(Qu,Q\invtrans v) \in
\graph T$ in turn implies that
\begin{equation*}
\graph(Q\transpose T Q) \ni \big(Q^{-1}Qu,Q\transpose Q\invtrans v\big) = (u,v).
\end{equation*}
Thus, $Q\transpose T Q$ is maximal.

Finally, suppose $0 \in Q\transpose T Q(\bar x)$.  Multiplying on the left by
$Q\invtrans$, it follows that $0 \in T(Q\bar x)$.  Defining $x = Q\bar x$, one
then has that $\bar x = Q^{-1} x$, where $0 \in T(x)$, proving the last
claim.
\end{proof}

\subsection{Resolvents and foundational algorithms combined with
linear changes of variables}

Consider the calculation of resolvent maps of operators of the form
$Q\transpose T Q$ and discussed immediately above. For any $\bar v\in\real^n$
and scalar $c>0$, calculating $\prox_{c ( Q\transpose T Q )}(\bar v)$ involves
finding the unique $(\bar x,\bar y) \in \graph Q\transpose T Q$ such that
$\bar x + c \bar y = \bar v$ and then returning $\bar x$.
Using~\eqref{changecoordgraph}, doing so is equivalent to finding the
necessarily unique $(x,y) \in \graph T$ such that $Q^{-1}x + c Q\transpose y =
\bar v$, and returning $Q^{-1} x$.  Multiplying this last equation by
$Q\invtrans$, an equivalent condition is finding the unique $(x,y)
\in \graph T$ such that $Q\invtrans Q^{-1}x + c y = Q\invtrans \bar v$.
Letting $v \doteq Q\bar v$, so that $\bar v = Q^{-1} v$, one has $Q\invtrans
Q^{-1}x + c y = Q\invtrans Q^{-1} v$.  Setting $S \doteq Q\invtrans Q^{-1}$,
an equivalent condition is $Sx + c y = S v$.  Combining this equation with
$(x,y) \in \graph T$ and rearranging yields the inclusion $S(x-v) + cT(x) \ni
0$.  Since the immediately preceding analysis shows that the solution to this
inclusion is $x = Q^{-1} \prox_{c (Q\transpose T Q)}(Q v)$, it exists and is
unique for any choice of $v\in\real^n$.  The following definition formalizes
this ``preconditioned resolvent'' operation:

\begin{definition}[preconditioned proximal maps]
For any maximal monotone operator $T:\real^n \toset \real^n$, vector $v\in
\real^n$, scalar $c > 0$, and $n\times n$ symmetric positive definite matrix
$S$, let $\sprox{S}_{cT}(v)$ be the unique solution to the inclusion $S(x-v) +
cT(x) \ni 0$, or equivalently $Sx + cT(x) = S v$.
\end{definition}

\noindent In operator notation, one may also write $\sprox{S}_{cT} = (S +
cT)^{-1} S$.  In the special case $T= \partial f$ for a closed proper convex
function $f:\real^n \to \realinf$, the condition that $\dot x$ is the solution
to $S(x-v) + cT(x) \ni 0$ specializes to
\allowdisplaybreaks
\begin{align*}
&&&&                  &&& S(\dot x - v) + c \partial f(\dot x) \ni 0 \\
&&&& \Leftrightarrow  &&& \partial f(\dot x) + \frac{1}{c} S(\dot x - v) \ni 0 \\
&&&& \Leftrightarrow  &&& \dot x = \argmin_{x\in\real^n} 
                        \left\{ f(x) + 
                         \frac{1}{2c} (x - v)\transpose S (x - v) \right\} &&&&&& \\
&&&& \Leftrightarrow  &&& \dot x = \argmin_{x\in\real^n} 
                        \left\{ f(x) + 
                         \frac{1}{2c} \norm{x - v}^2_S \right\},
\end{align*}
under the standard definition of the matrix-induced norm $\norm{u}_S =
\sqrt{u\transpose S u}$.  Thus, one may view the $\sprox{S}_{cT}$ operation as
a version of  standard proximal minimization in which the norm changes the one
induced by the matrix $S = Q\invtrans Q^{-1}$.

The following change-of-variables lemma formalizes the above discussion and
will be used several times below to pass between coordinate systems when
applying proximal mappings:

\begin{lemma}[changes of variable for proximal maps] \label{lem:proxCoV}
Suppose that
\begin{itemize}[nosep]
\item $T:\real^n \toset \real^n$ is maximal monotone
\item $c > 0$ is a scalar
\item $Q$ is an $n\times n$ invertible real matrix
\item $\dot x, v \in \real^n$
\end{itemize}
and definef
\begin{align} \label{covdefs}
S &\doteq Q\invtrans Q^{-1} &
\dotbar x &\doteq Q^{-1}\dot x &
\bar v &\doteq Q^{-1} v.
\end{align}
Then, $\dot x = \sprox{S}_{cT}(v)$ if and only if $\dotbar{x} = \prox_{c
(Q\transpose T Q)}(\bar v)$.
\end{lemma}
\begin{proof}
The hypothesis $\dot x = \sprox{S}_{cT}(v)$ is by definition equivalent to
\begin{align*}
&&
S(\dot x - v) + cT(\dot x) &\ni 0 \\
\Leftrightarrow &&
Q\invtrans Q^{-1} (\dot x - v) + cT(\dot x) &\ni 0 &&
  [\text{by the definition of~} S] \\
\Leftrightarrow &&
Q\invtrans (Q^{-1} \dot x - Q^{-1} v) + cT(\dot x) &\ni 0 \\
\Leftrightarrow &&
Q^{-1} \dot x - Q^{-1} v + c Q\transpose T(\dot x) &\ni 0 &&
  [\text{multiplying by~} Q\transpose] \\
\Leftrightarrow &&
\dotbar x - \bar v + c Q\transpose T(Q \dotbar x) &\ni 0 &&
  [\text{substituting from~\eqref{covdefs}}] \\
\Leftrightarrow &&  
\dotbar x + c (Q\transpose TQ)(\dotbar x) &\ni \bar v
\end{align*}
The last inclusion is equivalent to the claimed result $\dotbar{x} = \prox_{c
(Q\transpose T Q)}(\bar v)$.
\end{proof}

The two propositions below respectively present preconditioned version of the
generalized proximal point and DR splitting methods.  Only the
symmetric preconditioner matrix $S$ will be specified, and the propositions'
convergence proofs employ a matrix $Q$ such that $Q\invtrans Q^{-1} = S$.  A
simple way to construct such a $Q$ is to make it the standard symmetric square
root of $S^{-1}$: let $\lambda_{\max} = \lambda_1 \geq \cdots \geq
\lambda_n = \lambda_{\min} > 0$ be the eigenvalues of $S$ and let $U$ be an $n
\times n$ orthogonal matrix whose columns are corresponding eigenvectors, so
that $S = U \diag(\lambda_1, \ldots, \lambda_n) U\transpose$.  Then let $Q
\doteq S^{-1/2} \doteq U \diag(\lambda_1^{-1/2}, \ldots, \lambda_n^{-1/2})
U\transpose$, so that $Q^{-1} = Q\invtrans = U \diag(\lambda_1^{1/2}, \ldots,
\lambda_n^{1/2}) U\transpose$.  Then
\allowdisplaybreaks[0]
\begin{align*}
Q\invtrans Q^{-1} 
&= U \diag(\lambda_1^{1/2}, \ldots, \lambda_n^{1/2}) U\transpose
  U \diag(\lambda_1^{1/2}, \ldots, \lambda_n^{1/2}) U\transpose \\
&= U \diag(\lambda_1^{1/2}, \ldots, \lambda_n^{1/2}) 
       \diag(\lambda_1^{1/2}, \ldots, \lambda_n^{1/2}) U\transpose \\
&= U \diag(\lambda_1, \ldots, \lambda_n) U\transpose \\
&= S,
\end{align*}
as claimed.  Furthermore, $\smallnorm{Q} = \lambda_{\min}^{-1/2}$ and
$\smallnorm{Q^{-1}} = \lambda_{\max}^{1/2} = \smallnorm{S}^{1/2}$.

The following result establishes that the generalized proximal point algorithm
retains its convergence properties when $\sprox{S}$ operators are
substituted for the usual resolvents.

\begin{proposition} 
\label{prop:genPPArs}
Let $T:\real^n\toset\real^n$ be a maximal monotone operator, and $S$ be an $n \times n$ symmetric positive definite matrix.   Further suppose that 
\begin{enumerate}[label={(\roman*)}, nosep]
\item \label{item:thmGenPPArs:stepsizes} $\{c_k\}_{k=0}^\infty \subset
\real_{++}$ a sequence of positive scalars with $\inf_{k\geq 0} \{ c_k \} > 0$
\item \label{item:thmGenPPArs:relax} $\{\nu_k\}_{k=0}^\infty$ is a sequence of
real numbers with $\inf_{k \geq 0} \{\nu_k\} > 0$ and $\sup_{k \geq 0}
\{\nu_k\} < 2$
\item \label{item:thmGenPPArs:errors} $\{\epsilon_k\}_{k=0}^\infty$ be a sequence of
nonnegative real numbers with $\sum_{k=0}^\infty \epsilon_k < \infty$.
\end{enumerate}
Starting from an arbitrary $x^0\in \real^n$, suppose that
$\{x^k\}_{k=0}^\infty,\{e^k\}_{k=0}^\infty \subset \real^n$ conform for
all $k \geq 0$ to the conditions
\begin{align} \label{genppars}
x^{k+1} &= \nu_k \cdot \sprox{S}_{c_k T}(x^k) + (1-\nu_k) x^k + e^k &
\norm{e^k} &\leq \epsilon_k.
\end{align}
Then if $T^{-1}(0) \neq \emptyset$, the sequence $\{x^k\}$ converges to some
$x^*\in\real^n$ such that $0\in T(x^*)$.  If $T^{-1}(0) = \emptyset$, then
$\{x^k\}$ must be an unbounded sequence.
\end{proposition}
\begin{proof}
Let $Q \doteq S^{-1/2}$ as described above, so that $Q\invtrans Q^{-1} =
Q^{-2} = S$ and $\smallnorm{Q^{-1}} = \smallnorm{S}^{1/2}$. The proof revolves
around relating the path of $\{x^k\}$ to that of the original algorithm of
Theorem~\ref{thm:genPPA} to the operator $\overline T \doteq Q\transpose T Q =
Q T Q$ using the same $\{c_k\}$ and $\{\nu_k\}$, along with an appropriately
defined error sequence.  Define
\begin{align} \label{bardefs1}
(\forall\,k \geq 0) &&
\bar x^k &\doteq Q^{-1} x^k & 
\dot x^k &\doteq \sprox{S}_{c_k T}(x^k) &
\dotbar{x}^k &\doteq Q^{-1} \dot x^k &
\bar e^k &\doteq Q^{-1} e^k.
\end{align}
Here, $\dot x^k$ is the exact result of the $\sprox{S}_{c_k T}$ operation at
iteration $T$, so that~\eqref{genppars} may be written as
\begin{equation*}
(\forall\,k\geq 0) \quad
  x^{k+1} = \nu_k \dot x^k + (1-\nu_k) x^k + e^k.
\end{equation*}
Multiplying this equation by $Q^{-1}$ by yields
\begin{align}
(\forall\,k\geq 0) && &&
  Q^{-1} x^{k+1} &= \nu_k Q^{-1} \dot x^k + (1-\nu_k) Q^{-1} x^k + Q^{-1} e^k &&&& 
  \nonumber \\
&& \Leftrightarrow &&
  \bar x^{k+1} &= \nu_k \dotbar x^k + (1-\nu_k) \bar x^k + \bar e^k.
  \label{almostbarredgenppa}
\end{align}

For each $k \geq 0$, using~\eqref{bardefs1} and applying
Lemma~\ref{lem:proxCoV} under the substitutions $\dot x \leftarrow \dot x^k$
and $v \leftarrow x^k$ yields $\dotbar x^k = \prox_{c_k (Q\transpose T
Q)}(\bar x^k) = \prox_{c_k \overline{T}}(\bar x^k)$, which may be substituted
into~\eqref{almostbarredgenppa} to produce
\begin{equation*}
(\forall\,k\geq 0) \quad
\bar x^{k+1} = \nu_k \prox_{c_k \overline{T}}(\bar x^k) 
               + (1-\nu_k) \bar x^k + \bar e^k.
\end{equation*}
This recursion is identical to~\eqref{genppa} under the substitutions $T
\leftarrow \overline{T}$, $x^k \leftarrow \bar x^k$, and $e^k \leftarrow \bar
e^k$. Furthermore,
\begin{align*}
(\forall\, k \geq 0) &&
\norm{\bar e^k} &= \norm{Q^{-1} e^k} \leq \norm{Q^{-1}} \norm{e^k} 
            = \norm{S}^{1/2} \norm{e^k}
            \leq \norm{S}^{1/2} \epsilon_k.
\end{align*}
Since $\{\epsilon_k\}$ is summable, it follows that $\big\{\smallnorm{\bar
e^k}\big\}$ is summable.  Therefore all the assumptions of
Theorem~\ref{thm:genPPA} are met, and the theorem asserts that $\{\bar x^k\}$
converges to a root $\bar x^*$ of $\overline{T}$ if one exists, and is
otherwise unbounded.

By Proposition~\ref{prop:monotoneChangeCoords}, roots of $\overline{T} = QTQ =
Q\transpose T Q$ exist if and only if roots of $T$ exist.  If these roots
exist, Theorem~\ref{thm:genPPA} asserts that $\{\bar x^k\}$ converges to a
root $\overline{T}$. Again using Proposition~\ref{prop:monotoneChangeCoords},
$\bar x^* = Q^{-1} x^*$, where $x^*$ is some root of $T$.
From~\eqref{bardefs1}, $x^k = Q \bar x^k$ for all $k\geq 0$, so by continuity
of the linear map $Q$,
\begin{equation*}
\lim_{k\to \infty} x^k 
= \lim_{k\to \infty} Q \bar x^k
= Q\left(\lim_{k\to \infty} \bar x^k \right)
= Q \bar x^*
= Q Q^{-1} x^* 
= x^*,
\end{equation*}
meaning that $\{x^k\}$ converges to a root of $T$.  

On the other hand, if $T$ has no roots, then $\overline{T} = Q\transpose T Q$
also has no roots, and Theorem~\ref{thm:genPPA} asserts that $\{\bar x^k\}$ is
unbounded. By the nonsingularity of $Q$, if then follows that $\{x^k\} = \{Q
\bar x^k \}$ is also unbounded.
\end{proof}

\begin{proposition}[a generalized DR splitting algorithm]
\label{prop:genDRrs}
Let $A,B:\real^n\toset\real^n$ be two maximal monotone operators, and, for
arbitrary $p^0,z^0\in\real^n$, let $\{p^k\}_{k=0}^\infty,
\{z^k\}_{k=0}^\infty, \{q^k\}_{k=1}^\infty, \{r^k\}_{k=1}^\infty \subset
\real^n$ be sequences evolving according to the recursions, for all $k \geq
0$,
\begin{align}
q^{k+1} &= \sprox{S}_{cA}(p^k - c S^{-1} z^k) + e_q^{k} \label{drcycleArs} \\
r^{k+1} &= \nu_k q^{k+1} + (1-\nu_k) p^k + c S^{-1} z^k \label{drcycleRelaxrs} \\
p^{k+1} &= \sprox{S}_{cB}(r^{k+1}) + e_p^{k} \label{drcycleBrs} \\
z^{k+1} &= \frac{1}{c}S(r^{k+1} - p^{k+1}) \label{drcycleBdualrs},
\end{align}
where
\begin{itemize}[nosep]
\item $c > 0$ is a fixed scalar,
\item $S$ is a positive definite $n \times n$ symmetric matrix
\item $\{\nu_k\}_{k=0}^\infty \subset \real$ is a sequence such that $\inf_k
\nu_k > 0$ and $\sup_k \nu_k < 2$, and
\item $\{e_q^k\}_{k=0}^\infty, \{e_p^k\}_{k=0}^\infty \subset \real^n$ 
are sequences such that $\sum_{k=0}^\infty \smallnorm{e_q^k} < \infty$ and
$\sum_{k=0}^\infty \smallnorm{e_p^k} < \infty$.
\end{itemize}
Then, if a solution to the inclusion $0 \in A(p) + B(p)$ exists, then
$\{p^k\}$ and $\{q^k\}$ converge to some $p^*$ such that $0 \in
A(p^*)+B(p^*)$, while $\{z^k\}$ converges to some $z^* \in
B(p^*)$ such that $-z^* \in A(p^*)$.  

If no solution to $0 \in A(p) + B(p )$ exists, then at least
one of the sequences $\{p^k\}$ or $\{z^k\}$ must be unbounded.
\end{proposition}
\begin{proof}
As in the previous proof, start by letting $Q \doteq S^{-1/2}$ as described
above, so that $Q\invtrans Q^{-1} = Q^{-2} = S$ and $\smallnorm{Q^{-1}} =
\smallnorm{S}^{1/2}$.  Define
$\overline{A} \doteq Q\transpose A Q = QAQ$ and $\overline B \doteq Q\transpose B
Q = QBQ$, which are both maximal monotone by
Proposition~\ref{prop:monotoneChangeCoords}, and observe that
\begin{align*}
&&&&&& 0 &\in \overline A(\bar x) + \overline B(\bar x) &
\Leftrightarrow &&
0 &\in Q A(Q\bar x) + Q B(Q\bar x) &&&&&&&& \\
&&&&&&&&\Leftrightarrow &&
0 &\in A(Q\bar x) + B(Q\bar x) \\
&&&&&&&&\Leftrightarrow &&
0 &\in (A + B)(Q \bar x),
\end{align*}
so that the roots of $\overline A + \overline B$ are of the form $Q^{-1} x$,
where $x$ is a root of $A+B$.  

The proof now proceeds much like a more complicated version of the previous
one, relating the sequences produced
by~\eqref{drcycleArs}-\eqref{drcycleBdualrs} to those evolved
by~\eqref{drcycleA}-\eqref{drcycleBdual} applied to the operators $\overline
A$ and $\overline B$.  Define
\begin{align*}
(\forall\,k\geq 0) && \dot q^{k+1} &\doteq \sprox{S}_{cA}(p^k - c z^k) &
                      \dot p^{k+1} &\doteq \sprox{S}_{cB}(r^{k+1}),
\end{align*}
so that
\begin{align*}
(\forall\,k\geq 1) && q^k &= \dot q^k + e_q^k &
                      p^k &= \dot p^k + e_p^k.
\end{align*}
Further define
\begin{align*}
(\forall\,k\geq 1) &&
\bar q^k &\doteq Q^{-1} q^k & 
\dotbar q^k &\doteq Q^{-1} \dot q^k &
\bar r^k &\doteq Q^{-1} r^k &
\dotbar p^k &\doteq Q^{-1} \dot p^k \\
(\forall\,k\geq 0) &&
\bar p^k &\doteq Q^{-1} p^k &
\bar e_q^k &\doteq Q^{-1} e_q^k &
\bar e_p^k &\doteq Q^{-1} e_p^k &
\bar z^k &\doteq Q z^k &
\end{align*}
Applying Lemma~\ref{lem:proxCoV} to $q^{k+1} \doteq \sprox{S}_{cA}(p^k - c
z^k)$ with $T \leftarrow A$, $\dot x \leftarrow \dot q^{k+1}$, and $v
\leftarrow p^k - c S^{-1} z^k$ and therefore 
\begin{equation*}
\bar v \leftarrow Q^{-1}(p^k - c S^{-1} z^k)
  = Q^{-1}(p^k - c Q^2 z^k) = Q^{-1}p^k - cQz^k = \bar p^k - c \bar z^k
\end{equation*}
yields
\begin{equation*} 
(\forall\,k\geq 0) \quad
\dotbar q^k = \prox_{c(Q\transpose A Q)}(\bar p^k - c \bar z^k)
            = \prox_{c \overline A}(\bar p^k - c \bar z^k).
\end{equation*}
Next, starting with the definition of $\bar q^k$ and finally using the immediately
preceding expression for~$\dotbar q^k$,
\begin{multline}
(\forall\,k\geq 0) \quad
\bar q^{k+1} =  Q^{-1} q^{k+1} = Q^{-1} (\dot q^{k+1} + e_q^{k} ) 
  \\ = Q^{-1} \dot q^{k+1} + Q^{-1} e_q^{k} 
     = \dotbar q^{k+1} + \bar e_q^{k} 
     = \prox_{c \overline A}(\bar p^k - c \bar z^k) + \bar e_q^{k}.
       \label{drrsA}
\end{multline}
Subsequently, multiplying~\eqref{drcycleRelaxrs} by $Q^{-1}$ and using that
$Q^{-1}S^{-1} z^k = Q^{-1}Q^2 z^k = Q z^k = \bar z^k$ leads to
\begin{align}
(\forall\,k\geq 0) &&&&
Q^{-1} r^{k+1} &= \nu_k Q^{-1} q^{k+1} + (1-\nu_k) Q^{-1} p^k + c Q^{-1} S^{-1} z^k \nonumber \\
&& \Leftrightarrow &&
\bar r^{k+1} &= \nu_k \bar q^{k+1} + (1-\nu_k) \bar p^k + c \bar z^k. &&&&&&
   \label{drrsrelax}
\end{align}
Next, for each $k \geq 0$, one applies Lemma~\ref{lem:proxCoV} with $T
\leftarrow B$, $\dot x \leftarrow \dot p^{k+1}$, and $v \leftarrow r^{k+1}$,
and therefore
\begin{align*}
\dotbar x &= Q^{-1} \dot x = Q^{-1} \dot p^{k+1} = \dotbar p^{k+1} &
\bar v &= Q^{-1} v = Q^{-1} r^{k+1} = \bar r^{k+1}.
\end{align*}
Since $\dot p^{k+1} = \sprox{S}_{cB}(r^{k+1})$, the conclusion of the lemma is
then that
\begin{equation} \label{drBstepequiv}
(\forall\,k \geq 0) \quad 
\dotbar p^{k+1} = \prox_{c(Q\transpose B Q)}(\bar r^{k+1}) 
                = \prox_{c\overline B}(\bar r^{k+1}).
\end{equation}
Starting with the definition of $\{\bar p^k\}$ one then has 
\begin{align}
(\forall\,k\geq 0) \quad
\bar p^{k+1} &= Q^{-1} p^{k+1} \nonumber \\
   &= Q^{-1} \big( \sprox{S}_{cB}(r^{k+1}) + e_p^{k+1} \big) 
         && [\text{from~\eqref{drcycleBrs}}] \nonumber \\
   &= Q^{-1} \dot p^{k+1} + Q^{-1} e_p^{k+1} \nonumber \\
   &= \dotbar p^{k+1} + \bar e_p^{k+1} \nonumber \\
   &= \prox_{c\overline B}(\bar r^{k+1}) + \bar e_p^{k+1}
        && [\text{from~\eqref{drBstepequiv}}]. \label{drrsB}
\end{align}
Finally, substituting the last recursion step~\eqref{drcycleBdualrs} into the definition of $\bar z^k$ yields
\begin{multline}
(\forall\,k\geq 0) \qquad
\bar z^{k+1} = Q z^{k+1} = \frac{1}{c} Q S(r^{k+1} - p^{k+1})
  = \frac{1}{c}Q Q^{-2} (r^{k+1} - p^{k+1}) \\
  = \frac{1}{c}(Q^{-1}r^{k+1} - Q^{-1}p^{k+1})
  = \frac{1}{c}(\bar r^{k+1} - \bar p^{k+1}).
 \label{drrsduals}
\end{multline}
Collecting~\eqref{drrsA}, \eqref{drrsrelax}, \eqref{drrsB},
and~\eqref{drrsduals}, one has for all $k \geq 0$ that 
\begin{align*}
\bar q^{k+1} &= \prox_{c \overline A}(\bar p^k - c \bar z^k) + \bar e_q^{k} \\
\bar r^{k+1} &= \nu_k \bar q^{k+1} + (1-\nu_k) \bar p^k + c \bar z^k  \\
\bar p^{k+1} &= \prox_{c \overline B}(\bar r^{k+1}) + \bar e_p^{k}  \\
\bar z^{k+1} &= \frac{1}{c}(\bar r^{k+1} - \bar p^{k+1}).
\end{align*}
This set of recursions is identical to those of Theorem~\ref{thm:genDR} except
for the overbars on every vector and the substitutions $A \leftarrow \overline
A$ and $B \leftarrow \overline B$.  Furthermore,
\begin{align*}
(\forall\,k\geq 1) &&
\norm{\bar e_q^k} &= \norm{Q^{-1} e_q^k} \leq \norm{Q^{-1}} \norm{e_q^k}
                   = \norm{S}^{1/2} \norm{e_q^k} &&&&&& \\ &&
\norm{\bar e_p^k} &= \norm{Q^{-1} e_p^k} \leq \norm{Q^{-1}} \norm{e_p^k}
                   = \norm{S}^{1/2} \norm{e_p^k},              
\end{align*}
so the summability of $\big\{\smallnorm{e_q^k}\big\}$ and
$\big\{\smallnorm{e_p^k}\big\}$ implies that $\big\{\smallnorm{\bar
e_q^k}\big\}$ and $\big\{\smallnorm{\bar e_p^k}\big\}$ are summable.
Therefore, Theorem~\ref{thm:genDR} applies.

Suppose now that $A + B$ has roots.  Then $\overline A + \overline B$ has
roots, and Theorem~\ref{thm:genDR} asserts that $\{\bar p^k\}$ and $\{\bar
q^k\}$ converge to some $\bar p^*$ such that $0 \in \overline A(\bar
p^*)+\overline B(\bar p^*)$, while $\{\bar z^k\}$ converges to some $\bar z^*
\in B(\bar p^*)$ such that $- \bar z^* \in A(\bar p^*)$.  Since $\bar p^*$ is
root of $\overline A + \overline B$, it is of the form $Q^{-1}p^*$, where
$p^*$ is a root of $A+B$, as argued at the beginning of the proof.  By the
continuity of the linear operator $Q$, one then has that
\begin{align*}
\lim_{k\to\infty} p^k &= \lim_{k\to\infty} Q\bar p^k = Q \bar p^* = Q Q^{-1} p^* = p^* &
\lim_{k\to\infty} q^k &= \lim_{k\to\infty} Q\bar q^k = Q \bar q^* = Q Q^{-1} p^* = p^*,
\end{align*}
establishing the claimed convergence of $\{p^k\}$ and $\{q^k\}$.
Theorem~\ref{thm:genDR} also asserts that $\{\bar z^k\}$ converges to some
$\bar z^* \in \overline{B}(\bar p^*)$ such that $-\bar z^* \in
\overline{A}(\bar p^*)$.  From the definitions of $\overline A$ and $\overline
B$, defining $z^* \doteq Q^{-1}\bar z^*$, and rewriting $\bar p^* = Q^{-1}
p^*$ as $Q \bar p^* = p^*$, one then has
\begin{align}
&&&&&&&&
         \bar z^* &\in Q B(Q\bar p^*) && \wedge &
       - \bar z^* &\in Q A(Q\bar p^*) &&&&&&&& \nonumber \\
&&&&&&\Leftrightarrow &&
         Q^{-1} \bar z^* &\in B(Q\bar p^*) && \wedge &
       - Q^{-1} \bar z^* &\in A(Q\bar p^*) \nonumber \\
&&&&&&\Leftrightarrow &&
          z^* &\in B(p^*) && \wedge &
        - z^* &\in A(p^*). \label{drzEquiv}
\end{align}
For all $k \geq 0$, one has $z^k = Q^{-1} \bar z^k$ because $\bar z^k$ was
defined equal to $Qz^k$.  Therefore, using the continuity of the linear map
$Q^{-1}$,
\begin{equation*}
\lim_{k\to\infty} z^k
 = \lim _{k\to\infty} Q^{-1} \bar z^k
 = Q^{-1}\left(\lim_{k\to\infty} z^k\right)
 = Q^{-1}\bar z^*
 = z^*
\end{equation*}
which in combination with~\eqref{drzEquiv} proves the assertion about the
convergence of $\{z^k\}$.

It remains to consider the case that $A + B$ has no roots.  In this situation,
$\overline A + \overline B$ also has no roots, and so Theorem~\ref{thm:genDR}
asserts that $\{\bar p^k\} = \{Q^{-1} p^k\}$ or $\{\bar z^k\} =
\{Q z^k\}$ is unbounded. Since $Q^{-1}$ and $Q$ are both nonsingular, it
follows that $\{p^k\}$ or $\{z^k\}$ must be unbounded.
\end{proof}

Observing that the inclusion $S(x-r) + cT(x) \ni 0$ in the definition of the
$\sprox{S}$ operation is equivalent to $(1/c)S(x-r) + T(x) \ni 0$, it follows
that $\sprox{S}_{cT} = \sprox{\big((1/c)S\big)}_T$ for any monotone operator $T$.
Noting also how $c$ and $S$ alway occur together
in~\eqref{drcycleArs}-\eqref{drcycleBdualrs}, it is possible to fix $c = 1$ in
the algorithm, since the effect of $c \neq 1$ may be equivalently obtained by
scaling $S$ by $1/c$.  Thus, $c$ is included primarily for historical reasons.

The algorithm above may be streamlined somewhat by replacing the vectors $z^k$
with the vectors $\tilde z^k \doteq S^{-1} z^k$ for all $k \geq 0$.  One then
obtains the equivalent set of recursions
\begin{align}
q^{k+1} &= \sprox{S}_{cA}(p^k - c \tilde z^k) + e_q^{k} \label{drcycleA1} \\
r^{k+1} &= \nu_k q^{k+1} + (1-\nu_k) p^k + c \tilde z^k \label{drcycleRelax1} \\
p^{k+1} &= \sprox{S}_{cB}(r^{k+1}) + e_p^{k} \label{drcycleB1} \\
\tilde z^{k+1} &= \frac{1}{c}(r^{k+1} - p^{k+1}) \label{drcycleBdual1},
\end{align}
While the resulting calculations are simpler, they carry the minor
inconvenience that $\{\tilde z^k\}$ converges to $S^{-1} z^*$, where
$z^* \in B(p^*) \cap -A(p^*)$, rather than to such a $z^*$ itself.

Another way to create algorithms essentially identical to those of this
subsection is to use versions of Theorems~\ref{thm:genPPA} and~\ref{thm:genDR}
formulated for abstract Hilbert spaces and then apply them to an
$n$-dimensional space whose inner product is $\inner{x}{y}_S \doteq
x\transpose S y$ and whose norm is consequently $\smallnorm{x}_S \doteq
\sqrt{x\transpose S x}$ (monotone operators over $\real^n$ can also be shown
to be monotone over this space).  Here, we instead use the change-of-variables
approach in $\real^n$ since it requires less abstraction. The requisite proofs
above are somewhat long and mechanical, but much shorter than re-proving
convergence from first principles, while needing less background than the
general-Hilbert-space approach. The change-of-variables techniques is also
used, for example, in~\cite{Eck94c}.

\section{Objective gaps and partial strong convexity}
\label{sec:partialStrong}
The following simple but apparently new results are key elements in the ensuing
analysis:

\begin{lemma}[Partial strong convexity] \label{lem:partialStrongConvexity}
Let $h:\real^n\to\real\cup\{+\infty\}$ be a closed proper convex function, $A$
be an $m \times n$ matrix, $\sigma > 0$ be any positive scalar, and define the
function $\phi : \real^n\to\real\cup\{+\infty\}$ by $\phi(x) \doteq h(x) +
\frac{\sigma}{2} \norm{Ax}^2$. Then,
\index{Partial strong convexity}
\index{Strong convexity!Partial|see{Partial strong convexity}}
\index{Strong convexity|see{Strongly convex function}}
\begin{equation} \label{partialstrongsubgrad}
\big(\forall\,x,x'\in\real^n, w\in\partial \phi(x)\big) \qquad\quad
\phi(x') \geq \phi(x) + \inner{w}{x'-x} + \frac{\sigma}{2}\norm{Ax' - Ax}^2.
\end{equation}
In particular, if $x^*\in\real^n$ minimizes $g$, then
\begin{equation} \label{partialstrongopt}
(\forall\,x'\in\real^n) \qquad\quad
\phi(x') \geq \phi(x^*) + \frac{\sigma}{2} \norm{Ax' - Ax^*}^2.
\end{equation}
\end{lemma}
\begin{proof}
Since $q(x) \doteq \frac{\sigma}{2} \norm{Ax}^2$ is differentiable and defined
everywhere, $\partial \phi(x) = \partial(h + q)(x) = \partial h(x) + \nabla
q(x) = \partial h(x) + \sigma A\transpose Ax$ for all $x\in\real^n$ by
standard results in convex analysis, for example~\cite[Theorems 23.8 and
25.1]{Roc70book}.  To establish~\eqref{partialstrongsubgrad}, choose any
$x,x'\in\real^n$ and $w\in\partial \phi(x)$.  Since $w \in \partial \phi(x) =
\partial h(x) + \nabla q(x) = \partial h(x) + \sigma A\transpose Ax$, it
follows that $y\doteq w - \sigma A\transpose Ax \in \partial h(x)$.  By the
definition of subgradients, one has
\begin{align}
h(x') &\geq h(x) + \inner{y}{x'-x} \nonumber \\
&= h(x) + \inner{w - \sigma A\transpose Ax}{x-x'} \nonumber \\
&= h(x) + \inner{w}{x-x'} - \inner{\sigma A\transpose Ax}{x-x'} \nonumber \\
&= h(x) + \inner{w}{x-x'} - \sigma \inner{Ax}{Ax-Ax'}. \label{subgradwithquad}
\end{align}
On the other hand, simple quadratic expansion yields
\begin{align}
\norm{Ax'}^2 
&= \norm{Ax + (Ax - Ax')}^2 \nonumber \\
&= \norm{Ax}^2 + 2 \inner{Ax}{Ax - Ax'} + \norm{Ax - Ax'}^2 \label{quadexpand}
\end{align}
Adding $\sigma/2$ times~\eqref{quadexpand} to~\eqref{subgradwithquad} yields
\begin{align*}
h(x') + \frac{\sigma}{2}\norm{Ax'}^2 
&\geq h(x) + \inner{w}{x-x'} - \sigma \inner{Ax}{Ax-Ax'} \\
& \qquad
+ \frac{\sigma}{2} \norm{Ax}^2 + \sigma \inner{Ax}{Ax - Ax'} 
+ \frac{\sigma}{2}\norm{Ax - Ax'}^2 \\
&= h(x) + \frac{\sigma}{2} \norm{Ax}^2 + \inner{w}{x-x'} 
           + \frac{\sigma}{2}\norm{Ax - Ax'}^2,
\end{align*}
which from the definition of $\phi$ is equivalent to
\begin{equation*}
\phi(x') \geq \phi(x) + \inner{w}{x-x'} + \frac{\sigma}{2}\norm{Ax - Ax'}^2.
\end{equation*}
Since the choices of $x,x'\in\real^n$ and $w\in\partial \phi(x)$ were arbitrary,
\eqref{partialstrongsubgrad} has been established.

Now suppose that $x^*$ minimizes $g$, meaning that $0\in\partial \phi(x^*)$.  One
may therefore take $x=x^*$ and $w=0$ in~\eqref{partialstrongsubgrad},
resulting in~\eqref{partialstrongopt}.
\end{proof}

\begin{lemma} \label{lem:partialStrongDistance}
Let $h:\real^n\to\real\cup\{+\infty\}$ be a closed proper convex function, $A$
be an $m \times n$ matrix, $\sigma > 0$ be any positive scalar, and define the
function $\phi:\real^n\to\real\cup\{+\infty\}$ by $\phi(x) \doteq h(x) +
\frac{\sigma}{2} \norm{Ax}^2$. Then, if $x^*\in\real^n$ is a minimizer of $\phi$,
\begin{equation} \label{partialstrongdistance}
(\forall\,x'\in\real^n) \qquad 
\norm{Ax' - Ax^*} \leq \sqrt{\frac{2}{\sigma}\big(\phi(x') - \phi(x^*)\big)}.
\end{equation}
\end{lemma}
\begin{proof}
Fix any $x'\in\real^n$. The previous lemma asserts
that~\eqref{partialstrongopt} holds, and it may be rearranged into
\begin{align*}
&&
\phi(x') - \phi(x^*) &\geq \frac{\sigma}{2} \norm{Ax' - Ax^*}^2 &&&& \\
\Leftrightarrow && \frac{2}{\sigma}\big(\phi(x') - \phi(x^*)\big) &\geq \norm{Ax' - Ax^*}^2 \\
\Leftrightarrow && \sqrt{\frac{2}{\sigma}\big(\phi(x') - \phi(x^*)\big)} &\geq \norm{Ax' - Ax^*},
\end{align*}
which, taking into account that $x'\in\real^n$ was arbitrary, is equivalent
to~\eqref{partialstrongdistance}.
\end{proof}

Lemma~\ref{lem:partialStrongDistance} will be used in the analysis of the
approximate ALM and ADMM algorithms proposed in the next two sections.  The
following notation will simplify the descriptions of those methods:

\begin{definition} \label{def:approxmin}
\index{Approximate minimizer} For any $\delta > 0$, set $S$, and function $h :
S \to \real \cup \{\pm\infty\}$, define
\begin{equation}
\approxmin{\delta}{x\in S} \! \big\{ h(x) \big\}
   \doteq \set{x\in S}
          {h(x) \leq \left(\inf_{x\in S} h(x)\right) + \delta},
\end{equation}
that is, all $x\in S$ that come within $\delta$ of minimizing $h$ over $S$, as
measured by objective value.
\end{definition}

\section{Objective-gap inexact augmented Lagrangian methods}
\label{sec:alms}
\subsection{Parametric duality framework}
Consider a generic convex optimization problem formulated according to
Rockafellar's parametric duality framework as found in~\cite{Roc74}
or~\cite[Section 19.2]{BauComBook}: let $F:\real^{n+m} \to \real \cup
\{+\infty\}$ be a closed proper convex function, and let the primal problem be
to minimize $F(x,0)$ over $x\in\real^n$.  The corresponding dual problem is to
maximize $-F^*(0,p)$ over $p\in\real^m$, where the $*$ denotes the convex
conjugate operation. Defining the parametric value function $\varphi(u) \doteq
\inf_{x\in\real^n} \big\{ F(x,u) \big\}$ and using the definition of the
convex conjugate, another way of expressing the dual problem is to minimize
the following convex function of $p\in\real^m$:
\begin{multline} 
F^*(0,p) \doteq \sup_{\substack{x\in \real^n \\ u \in \real^m}}
                   \big\{ \inner{u}{p} - F(x,u) \big\} 
     = \sup_{u\in \real^m}
           \Big\{ \sup_{x\in\real^n} \big\{ \inner{u}{p} - F(x,u) \big\} \Big\}
          \\
     = \sup_{u\in \real^m}
           \Big\{ \inner{u}{p} - \inf_{x\in\real^n} \big\{ F(x,u) \big\} \Big\}    
     = \sup_{u\in \real^m}
           \big\{ \inner{u}{p} - \varphi(u) \big\}
     = \varphi^*(p). \label{varphiconj}
\end{multline}
As originally shown in~\cite{Roc74}, careful choice of $F$ allows the
superficially simple problem formulation $\min_{x\in\real^n}\set{F(x,u)}{u=0}$
to model essentially any convex optimization problem and its dual.
Frequently, although it is not required, $F$ is chosen so that minimizing
$F(x,u)$ over $u$ is straightforward for any fixed choice of $x$. For example,
to model the simple equality-constrained problem $\min_{x\in\real^n}
\set{f(x)}{Mx=b}$, where $f:\real^n\to\real\cup\{+\infty\}$ is closed proper
convex, $M$ is and $m \times n$ real matrix, and $b\in\real^m$, it is
customary to choose
\begin{equation} \label{Axeqbparam}
F(x,u) = 
\begin{cases}
f(x), & \text{if~} Mx + u = b \\
+\infty, & \text{otherwise}.
\end{cases}
\end{equation}
Once $x$ is determined, the unique minimizing value of $u$ is $u = b - Mx$,
since any other choice results in $F(x,u) = +\infty$.

\subsection{Algorithm analysis}
Augmented Lagrangian algorithms, as first established in~\cite{Roc76a}, are
obtained by applying the proximal point algorithm to the subgradient map
$\partial \varphi^*$ of the dual function $\varphi^*$.  The analysis here 
applies the preconditioned proximal point algorithm of
Proposition~\ref{prop:genDRrs}, meaning that it uses the recursion
\begin{equation} \label{dualrecursion1}
(\forall\,k \geq 0) \quad
p^{k+1} = \nu_k \big( \sprox{S}_{c_k \partial \varphi^*}(p^k)\big) + (1-\nu_k) p^k + e^k,
\end{equation}
which is the recursion~\eqref{genppars} of Proposition~\ref{prop:genDRrs} with
$T=\partial \varphi^*$ and the iterate sequence being $\{p^k\} \subset
\real^m$ instead of $\{x^k\} \subset \real^n$.  The following result shows how
to exactly calculate $\sprox{S}_{c_k \partial \varphi^*}(p^k)$:

\begin{lemma}\label{lem:dualProx}
Let $F:\real^{n+m}\to\real\cup\{+\infty\}$ be closed proper convex, let $S$ be
a symmetric positive definite $m\times m$ matrix, and let $\varphi(u) \doteq
\inf_{x\in\real^n} \big\{ F(x,u) \big\}$.  Then, for any
$r\in\real^m$ and scalar $c > 0$, the vector $\dot p \doteq \sprox{S}_{c
\partial \varphi^*}(r)$ may be computed by
\begin{align}
(\dot x, \dot u) &\in 
\argmin_{\substack{x\in\real^n \\ u\in\real^m}} 
   \left\{F(x,u) - \inner{r}{u} + \frac{c}{2}\norm{u}^2_{S^{-1}} \right\} 
   \label{absauglag1} \\
\dot p &= r - c S^{-1} \dot u \label{absupdate}
\end{align}
if the minimum in~\eqref{absauglag1} is attained.
\end{lemma}
\begin{proof}
%
If the minimum in~\eqref{absauglag1} is attained,
\begin{align}
\min_{\substack{x\in\real^n \\ u\in\real^m}} 
   \left\{F(x,u) - \inner{r}{u} + \frac{c}{2}\norm{u}^2_{S^{-1}} \right\}
&=
\min_{u\in\real^m} \left\{
   \min_{x\in\real^n} \left\{F(x,u) - \inner{r}{u} + \frac{c}{2}\norm{u}^2_{S^{-1}} \right\}
\right\} \nonumber \\
&=
\min_{u\in\real^m} \left\{
   \inf_{x\in\real^n} \big\{F(x,u)\big\}- \inner{r}{u} + \frac{c}{2}\norm{u}^2_{S^{-1}} 
\right\} \nonumber \\
&=
\min_{u\in\real^m} \left\{
   \varphi(u) - \inner{r}{u} + \frac{c}{2}\norm{u}^2_{S^{-1}}
\right\}. \label{minabsauglag}
\end{align}
Define
\begin{align*}
a(u) &\doteq - \inner{r}{u} + \frac{c}{2}\norm{u}^2_{S^{-1}} &
b(u) &\doteq \varphi(u) + a(u),
\end{align*}
so that $b(u)$ is the function of $u$ being minimized
in~\eqref{minabsauglag}. A necessary and sufficient condition for $\dot u$ to
be attain the minimum in~\eqref{minabsauglag} is to have $0 \in
\partial b(\dot u)$.  Since the convex function $a$ is differentiable and
defined everywhere, one has $\partial b(u) = \partial\varphi(u) + \nabla a(u)$
for all $u\in \real^m$ (for example, using~\cite[Proposition 4.2.2]{BerConvex}
to obtain $\partial a(u) = \{ \nabla a(u) \}$ and~\cite[Theorem
23.8]{Roc70book} to obtain $\partial b(u)=\partial\varphi(u) + \nabla a(u)$
for all $u\in\real^m$).  Therefore, at the minimizer $\dot u$
of~\eqref{minabsauglag}, one has
\begin{equation*}
0 \in \partial b(\dot u) = \partial\varphi(\dot u) + \nabla a(\dot u)
= \partial\varphi(\dot u) - r + c S^{-1} \dot u,
\end{equation*}
and hence $\dot p = r - c S^{-1} \dot u \in \partial\varphi(\dot u)$.  By
the standard duality relations for convex conjugates --- see for
example~\cite[Theorem 23.5]{Roc70book} --- it follows that $\dot u \in
\partial\varphi^*(\dot p)$.  Then, 
\begin{align*}
&&&&
\dot p + c S^{-1} \dot u &= r - c S^{-1} \dot u + c S^{-1}\dot u = r 
&& \Leftrightarrow &
S \dot p + c \dot u &= Sr 
&& \Leftrightarrow &
S (\dot p - r) + c \dot u = 0, &&
\end{align*}
which, along with $\dot u \in \partial \varphi^*(\dot p)$, means that $\dot p =
\sprox{S}_{c\partial \varphi^*}(r)$.
\end{proof}

To avoid a detour into further technicalities, it will be assumed throughout
that minimizers of expressions like~\eqref{absauglag1} exist.

The following proposition formulates and proves convergence of a general
inexact augmented Lagrangian method, drawing on
Proposition~\ref{prop:genPPArs} in its analysis.  To make the algorithm
statement marginally more readable, it replaces $S^{-1}$ in the above lemma
with an arbitrary symmetric positive definite matrix $W$.

\begin{proposition}
\label{prop:absALMapproxObj}
Suppose $F: \real^{n+m} \to \real \cup \{+\infty\}$ is closed proper convex, 
let $p^0 \in \real^m$ be arbitrary, and suppose that
\begin{enumerate}[label={(\roman*)}, nosep]
\item $\{c_k\}_{k=0}^\infty \subset \real_{++}$ a sequence of positive scalars
with $\inf_{k\geq 0} \{ c_k \} > 0$
\item $\{\nu_k\}_{k=0}^\infty$ is a sequence of real numbers with $\inf_{k
\geq 0} \{\nu_k\} > 0$ and $\sup_{k \geq 0}
\{\nu_k\} < 2$
\item \label{item:absALMapprox:errorAssump} $\{\delta_k\}_{k=0}^\infty \subset
\real_+$ is a sequence of nonnegative numbers such that $\sum_{k=0}^\infty
\sqrt{c_k \delta_k} < \infty$
\item $W$ is any positive definite symmetric $m \times m$ matrix.
\end{enumerate}
Further suppose that the sequences $\{p^k\}_{k=0}^\infty, \{u^k\}_{k=1}^\infty
\subset \real^m$ and $\{ x^k \}_{k=1}^\infty \subset \real^n$ evolve according
to the following recursions for all $k \geq 0$~:
\begin{align}
(x^{k+1},u^{k+1}) &\in
   \approxmin{\delta_k}{\substack{x\in\real^n \\ u\in\real^m}}
      \left\{ F(x,u) - \inner{p^k}{u} + \frac{c_k}{2}\norm{u}^2_W \right\}
   \label{absalmobjapprox} \\
p^{k+1} &= p^k - \nu_k c_k W u^{k+1}, \label{absalsoamupdate}
\end{align}
where the exact minimum in~\eqref{absalmobjapprox} is assumed to be attainable
for all $k$. 

If there exists any optimal solution to the dual problem of minimizing the
dual function $\varphi^*$ identified in~\eqref{varphiconj}, then $\{p^k\}$
converges to some such minimizer, and furthermore
\begin{align}
\lim_{k\to\infty} c_k u^{k+1} &= 0 \label{absalmasympfeasck2} \\
\lim_{k\to\infty} u^k &= 0 \label{absalmasympfeas2} \\     
\limsup_{k\to\infty} F(x^k,u^k) &\leq \inf\set{F(x,0)}{x\in\real^n}. \label{absalmasympopt2}
\end{align}
Furthermore, every limit point of the sequence $\{x^k\}$ is an optimal
solution of the primal problem $\min_{x\in\real^n} \big\{ F(x,0) \big\}$.

If no minimizers of the dual function $\varphi^*$ exist, then $\{p^k\}$ is an
unbounded sequence.
\end{proposition}
\begin{proof}
The proof hinges on showing that 
\begin{equation} \label{dualrecursion2}
(\forall\,k \geq 0) \quad
p^{k+1} = \nu_k \cdot \sprox{(W^{-1})}_{c_k \partial \varphi^*}(p^k) 
                             + (1-\nu_k) p^k + e^k,
\end{equation}
where $\{e^k\} \subset \real^m$ is such that $\big\{\smallnorm{e^k}\big\}$
forms a summable sequence.  This is a form of the generalized proximal
point recursion~\eqref{genppars} of Proposition~\ref{prop:genPPArs}, with $S$
replaced by $W^{-1}$.  For the remainder of this proof, define $S \doteq
W^{-1}$.

For all $k\geq 0$, let $(\dot x^k, \dot u^k)$ denote some exact minimizer of 
\begin{equation*}
F(x,u) - \inner{p^k}{u} +
\frac{c_k}{2}\norm{u}^2_W = F(x,u) - \inner{p^k}{u} +
\frac{c_k}{2}\norm{u}^2_{S^{-1}},
\end{equation*}
as assumed to exist in~\eqref{absalmobjapprox}.  Let $F_k$ denote the function
$\real^{n+m} \to \real \cup \{+\infty\}$ given by $(x,u) \mapsto F(x,u) -
\inner{p^k}{u} + \frac{c_k}{2}\norm{u}^2_W$.  Then the approximate
minimization of the augmented Lagrangian stipulated in the
condition~\eqref{absalmobjapprox} means that
\begin{equation} \label{absalmcompactapprox}
(\forall\,k \geq 0) \quad
F_k(x^{k+1},u^{k+1}) - F_k(\dot x^k, \dot u^k) \leq \delta_k.
\end{equation}
Let $A = \big[ 0 \;\, W]$, the $(n+m)\times m$ matrix such that
$A(x,u) = Wu$.  Then, defining $G_k : \real^{n+m} \to \real \cup \{+\infty\}$
by $G_k(x,u) = F(x,u) - \inner{p^k}{u}$, one has that $G_k$ is convex and
$F_k(x,u) = G_k(x,u) + \frac{c_k}{2}\norm{A(x,u)}^2$ for all
$(x,u)\in\real^{n+m}$. Then, using Lemma~\ref{lem:partialStrongDistance} with
$\sigma \leftarrow c_k$, $x' \leftarrow (x^{k+1},u^{k+1})$. $x^*
\leftarrow (\dot x^k, \dot u^k)$, and $h \leftarrow F_k$, one has
\begin{align}
(\forall\,k \geq 0) \quad
\norm{Wu^{k+1} - W \dot u^k} 
&= \norm{A(x^{k+1},u^{k+1}) - A(\dot x^k, \dot u^k)} 
  && [\text{by the definition of~}A] \nonumber \\
&\leq \sqrt{\frac{2}{c_k}\big(F_k(x^{k+1},u^{k+1}) - F_k(\dot x^k, \dot u^k)\big)}
  && [\text{using Lemma~\ref{lem:partialStrongDistance}}] \nonumber \\
&\leq \sqrt{{2\delta_k}/{c_k}}
  && [\text{by~\eqref{absalmobjapprox}}].  \label{uerrorsqrt}
\end{align}
Lemma~\ref{lem:dualProx} with $r \leftarrow p^k$, $c \leftarrow c_k$, and $F
\leftarrow F_k$ asserts that $\sprox{S}_{c_k \partial d}(p^k) = p^k - c_k
S^{-1} \dot u^k$, that is, $\sprox{(W^{-1})}_{c_k \partial d}(p^k) = p^k - c_k
W \dot u^k$.  Then, define 
\begin{equation*}
(\forall\,k\geq 0) \quad e^k \doteq \nu_k c_k (W \dot u^k - Wu^{k+1}),
\end{equation*}
so that 
\begin{multline*}
(\forall\,k \geq 0) \quad
\nu_k \cdot \sprox{(W^{-1})}_{c_k \partial d}(p^k) + (1-\nu_k) p^k + e^k \\
  = \nu_k (p^k - c_k W \dot u^k) + (1-\nu_k) p^k + \nu_k c_k (W \dot u^k - Wu^{k+1}) \\
  = \nu_k (p^k - c_k W u^{k+1}) + (1-\nu_k) p^k 
  = p^k - c_k \nu_k W u^{k+1} = p^{k+1}.
\end{multline*}
Therefore, \eqref{dualrecursion2} holds.  Using~\eqref{uerrorsqrt}, one has
\begin{multline*}
(\forall\,k \geq 0) \quad
\norm{e^k} 
= \norm{\nu_k c_k (W u^{k+1} - W \dot u^k)} \\
= \nu_k c_k \norm{W u^{k+1} - W \dot u^k}
\leq \nu_k c_k \sqrt{{2\delta_k}/{c_k}} 
= \nu_k \sqrt{2 c_k \delta_k}
< 2 \sqrt{2 c_k \delta_k},
\end{multline*}
where the final inequality is due to $\nu_k$ being bounded away from $2$.
Since $\sum_{k=0}^\infty {}\sqrt{c_k
\delta_k} < \infty$ by assumption, it follows that
$\big\{\smallnorm{e^k}\big\}$ forms a summable sequence, which together
with~\eqref{dualrecursion2} means that Proposition~\ref{prop:genPPArs} then
asserts that $\{p^k\}$ converges to a minimizer of $\varphi^*$ if one exists,
and is otherwise unbounded.

For the remainder of the proof, consider only the convergent case (from this
point, the reasoning is of a standard nature for augmented Lagrangian methods).
Since $\{p^k\}$ converges,
\begin{align*}
0 &= \lim_{k\to\infty} \{ p^k -  p^{k+1} \}
  = \lim_{k\to\infty} \big \{ p^k - (p^k - c_k W u^{k+1}) \big \}
  = \lim_{k\to\infty} \{ c_k W u^{k+1} \},
\end{align*}
so $c_k W u^{k+1} \to 0$ 
and hence $c_k u^{k+1} \to 0$, since $W$ is nonsingular.  Since $\{c_k\}$ is
bounded away from zero, it follows that $u^{k+1} \to 0$, and
both~\eqref{absalmasympfeasck2} and~\eqref{absalmasympfeas2} hold.
To prove~\eqref{absalmasympopt2}, let 
\begin{equation*}
(\forall\,k\geq 0) \quad
\zeta_k \doteq \min_{\substack{x\in\real^n\\u\in\real^m}}
          \left\{F(x,u) - \inner{p^k}{u} + \frac{c_k}{2}\norm{u}^2_W\right\}.
\end{equation*}
One then has 
\begin{equation*}
(\forall\,k\geq 0) \;\; (\forall\,x'\in\real^n) \qquad  
F(x',0) = F(x',0) - \inner{p^k}{0} + \frac{c_k}{2}\norm{0}^2_W \geq \zeta_k
\end{equation*}
because any $(x',0)$ is a possible choice of $(x,u)$ in the minimand in the
definition of $\zeta_k$.  Since $x'$ can take any value in $\real^n$ in the above
inequality, it follows that $\inf_{x\in\real^n} F(x,0) \geq \zeta_k$ for all $k$.
The definition of $\zeta_k$ and~\eqref{absalmobjapprox} then yield
\begin{equation} \label{almobjsqueeze}
(\forall\, k \geq 0) \quad
F(x^{k+1},u^{k+1}) - \inner{p^k}{u^k} + \frac{c_k}{2}\norm{u^{k+1}}^2_W - \delta_k 
\leq \zeta_k \leq \inf_{x\in\real^n} F(x,0).
\end{equation}
Of the terms on the left of this inequality,
\begin{itemize}
\item $\inner{p^k}{u^k} \to 0$ since $\{p^k\}$ is convergent and $u^k \to 0$.
\item $(c_k/2) \smallnorm{u^{k+1}}^2_W = \half (u^{k+1})\transpose (c_k W
u^{k+1}) \to 0$ since $u^k \to 0$ and $c_k W u^{k+1} \to 0$.
\item $\delta_k \to 0$ since it was assumed $\big\{\sqrt{c_k\delta_k}\big\}$ is summable
and $\{c_k\}$ is bounded away from 0.
\end{itemize}
Taking the limit in~\eqref{almobjsqueeze} then establishes that
\begin{equation*}
\limsup_{k\to\infty} F(x^{k+1},u^{k+1}) \leq \inf_{x\in\real^n} F(x,0),
\end{equation*}
proving~\eqref{absalmasympopt2}.

It remains only to prove the assertion about the limit points of $\{x^k\}$.
Suppose that $x^\infty$ is a limit point of $\{x^k\}$, implying existence
of an infinite set of indices $\mathcal{K}$ such that
$\lim_{k\to\infty,k\in\mathcal{K}} \{ x^k \} = x^\infty$. One then has
\begin{equation*}
F(x^\infty,0) 
\leq \liminf_{\substack{k\to\infty\\k\in\mathcal{K}}} F(x^k,u^k)
\leq \limsup_{\substack{k\to\infty\\k\in\mathcal{K}}} F(x^k,u^k)
\leq \limsup_{k\to\infty} F(x^k,u^k) 
\leq \inf_{x\in\real^n} F(x,0),
\end{equation*}
where the first inequality uses that $F$ is closed and thus lower
semicontinuous, with $u^k \to 0$, and the last inequality follows
from~\eqref{absalmasympopt2}. Therefore, $x^{\infty}$ is an optimal solution
of the primal problem.
\end{proof}

Note that the proof makes no claim that $\{x^k\}$ possesses any limit points,
only that they are optimal if they exist.  Whether limit points of $\{x^k\}$
exist can depend on whether the set of optimal solutions is bounded and the
details of the procedure used to solve the subproblems.  In practice,
subproblem solvers are typically ``warm started'' from last iterate of the
previous subproblem and $\{x^k\}$ tends to converge whenever $\{p^k\}$ does.

\subsection{Applications with subspace constraints}
\label{sec:subspaceCon}
Now consider the class of problems of the form $\min_{x\in\real^n}
\set{f(x)}{Mx \in V}$, where $f:\real^n\to\real\cup\{+\infty\}$ is closed
proper convex, $M$ is an $m \times n$ matrix, and $V$ is a linear subspace of
$\real^m$.  To formulate this problem in the parametric duality framework, set
\begin{equation} \label{subspaceparam}
F(x,u) = 
\begin{cases}
f(x), & \text{if~} Mx + u \in V \\
+\infty, & \text{otherwise}.
\end{cases}
\end{equation}
At any given iteration $k\geq 0$, the
minimization~\eqref{absalmobjapprox} in the augmented Lagrangian
algorithm of Proposition~\ref{prop:absALMapproxObj} takes the following form
for this choice of $F$:
\begin{equation*}
\min_{\substack{x\in\real^n \\ u\in\real^m}}
      \left\{ F(x,u) - \inner{p^k}{u} + \frac{c_k}{2}\norm{u}^2_W \right\}
 = \min_{x\in \real^n}
      \left\{
         f(x) + \min_{u: Mx + u \in V} 
         \left\{ 
            - \inner{p^k}{u} + \frac{c_k}{2}\norm{u}^2_W 
         \right\}
      \right\}.
\end{equation*}
Changing the variable in the inner minimand to $v=Mx+u$, hence $u=v-Mx$, one
may express the same minimum as
\begin{multline*}
\min_{x\in \real^n}
      \left\{
         f(x) + \min_{v \in V} 
         \left\{ 
            - \inner{p^k}{v-Mx} + \frac{c_k}{2}\norm{v-Mx}^2_W 
         \right\}
      \right\}
\\ =
\min_{x\in \real^n}
      \left\{
         f(x) + \inner{p^k}{Mx} + \frac{c_k}{2} \min_{v \in V} 
         \Big\{ 
             - \inner{p^k}{v} + \norm{v-Mx}^2_W 
         \Big\}
      \right\}.
\end{multline*}
Now, assume that $p^k \in V\orthog$ (this property may be shown to be required
for $\varphi^*(p^k)$ to be finite, but for brevity it is simply assumed here).
In this case, the inner product $\inner{p^k}{v}$ in the inner minimand is
always zero, simplifying the overall minimization to
\begin{equation} \label{xvmin}
\min_{x\in \real^n}
      \left\{
         f(x) + \inner{p^k}{Mx} + \frac{c_k}{2} \min_{v \in V} 
         \Big\{ 
             \norm{v-Mx}^2_W 
         \Big\}
      \right\}.
\end{equation}
For any closed convex set $C\in\real^m$ and $z\in\real^m$, define
\begin{align*}
\sdist{W\!}_{C}(z) &\doteq \min_{v\in C} \big\{ \smallnorm{v-z}_W \big\} &
\sproj{W\!}_{C}(z) &\doteq \argmin_{v\in C} \big\{ \smallnorm{v-z}_W \big\},
\end{align*}
which are respectively the distance from $z$ to $C$ and the projection of $z$
onto $C$ in the norm induced by $W$.  The minimization
in~\eqref{xvmin} may then be written
\begin{equation} \label{spalmsubprob}
\min_{x\in \real^n}
      \left\{
         f(x) + \inner{p^k}{Mx} + \frac{c_k}{2} 
             \Big( \sdist{W\!}_V(Mx) \Big)^2
      \right\},
\end{equation}
and, given any choice of $x \in \real^n$, the unique optimal value of $v$
in~\eqref{xvmin} is then $\sproj{W\!}_{V}(Mx)$.  The following proposition
collects some standard facts about $W\!$-projectors onto subspaces:

\begin{proposition} \label{prop:wproj}
Given a linear subspace $V \subseteq \real^m$ and a symmetric positive definite
$m \times m$ matrix~$W$,
\begin{enumerate}[label={(\roman*)}, nosep]
\item \label{item:wproj:linear} $\sproj{W\!}_V$ is a linear map whose matrix
form is $Z \doteq B(B\transpose W B)^{-1}B\transpose W$, where $B$ is any
matrix whose columns form a basis for $V$ (or $Z=0$ if $V = \{0\}$).
\item \label{item:wproj:conj} $\identity - \sproj{W\!}_V = \identity - Z =
\sproj{W\!}_{V\conjspace{W}}$, where $V\conjspace{W}$ denotes the space
$W\!$-conjugate to $V$, that is
\begin{equation} \label{defconjspace}
V\conjspace{W} \doteq \set{w\in\real^n}{(\forall\, v\in V) \;\; w\transpose W v = 0}.
\end{equation}
\item \label{item:wproj:idem}  $Z$ and $\identity-Z$ are idempotent, that is, $Z^2 =
Z$ and $(\identity - Z)^2 = \identity - Z$.
\item \label{item:wproj:selfAdj}  $Z$ and $\identity-Z$ are $W\!$-self-adjoint
matrices, meaning that $Z\transpose W = WZ$ and $(\identity-Z)\transpose W = W(\identity -
Z)$.
\end{enumerate}
\end{proposition}

\noindent The full proof is omitted, but~\ref{item:wproj:linear} is easily
established by computing the vector $s$ minimizing $\smallnorm{Bs-z}^2_W$,
after which the projection is $Bs$. The remaining results are straightforward
to confirm from~\ref{item:wproj:linear}.

Returning to~\eqref{xvmin}, the value of $u$ corresponding to any given
choice of $v$ is
\begin{equation*}
u = v - Mx = \sproj{W\!}_{V}(Mx) - Mx = -\big( \sproj{W\!}_{V\conjspace{W}}(Mx) \big).
\end{equation*}
Assuming that $\sproj{W\!}_V(z)$ may be readily computed for any choice of
$z\in\real^m$, one way to find a pair $(x^{k+1}, u^{k+1})$ satisfying the
condition~\eqref{absalmobjapprox} in the inexact augmented Lagrangian method
is therefore to
\begin{align*}
\text{find}&&
x^{k+1} &\in \approxmin{\delta_k}{x\in \real^n}
      \left\{
         f(x) + \inner{p^k}{Mx} + \frac{c_k}{2} 
             \Big( \sdist{W\!}_V(Mx) \Big)^2
      \right\}, && \\
\text{then set} &&
u^{k+1} &= -\big( \sproj{W\!}_{V\conjspace{W}}(Mx^{k+1}) \big). 
\end{align*}
Substituting this form of $u^{k+1}$ into~\eqref{absalsoamupdate} yields the
complete algorithm recursions
\begin{align}
x^{k+1} &\in \approxmin{\delta_k}{x\in \real^n}
      \left\{
         f(x) + \inner{p^k}{Mx} + \frac{c_k}{2} 
             \Big( \sdist{W\!}_V(Mx) \Big)^2
      \right\} \label{lessabsapproxmin} \\
p^{k+1} &= p^k + \nu_k c_k W \big( \sproj{W\!}_{V\conjspace{W}}(Mx^{k+1}) \big). 
     \label{lessabsupdate}
\end{align}
Now, $u^{k+1} \in V\conjspace{W}$, meaning that $v W u^{k+1} = 0$ for all $v\in
V$, so $W u^{k+1} \in V\orthog$.  The vector added to $p^k$ to obtain
$p^{k+1}$ is $\nu_k c_k W u^{k+1}$, which is just $W u^{k+1}$ multiplied by a
scalar, and thus also in $V\orthog$.  Inductively, it is then clear that if
$p^0\in V\orthog$, one will have $p^k \in V\orthog$ for all $k$, and thus the
assumption above that $p^k \in V\orthog$ is justified.

An equivalent form of the algorithm may be derived by using $w^k \doteq W^{-1}
p^k$ to represent each Lagrange multiplier estimate vector $p^k \in V\orthog$.
It is then readily seen that $w^k \in V\conjspace{W}$ for all $k\geq 0$.
Multiplying~\eqref{lessabsupdate} by $W^{-1}$, one then obtains the
equivalent multiplier update recursions
\begin{align*}
&&
W^{-1} p^{k+1} 
  &= W^{-1} p^k + \nu_k c_k W^{-1} W \big( \sproj{W\!}_{V\conjspace{W}}(Mx^{k+1}) \big) \\
\Leftrightarrow &&
w^{k+1} 
  &= w^k + \nu_k c_k \big( \sproj{W\!}_{V\conjspace{W}}(Mx^{k+1}) \big).
\end{align*}
Since $p^k = W w^k$, the inner-product term in~\eqref{lessabsapproxmin} now
becomes $\inner{W w^k}{Mx} = (w^k)\transpose W M x$, for which one may use the
``$W\!$-inner-product'' notation $\inner{w^k}{Mx}_W$, and arrive at the
alternative algorithm formulation
\begin{align}
x^{k+1} &\in \approxmin{\delta_k}{x\in \real^n}
      \left\{
         f(x) + \inner{w^k}{Mx}_W + \frac{c_k}{2} 
             \Big( \sdist{W\!}_V(Mx) \Big)^2
      \right\} \label{lessabswapproxmin} \\
w^{k+1} &= w^k + \nu_k c_k \big( \sproj{W\!}_{V\conjspace{W}}(Mx^{k+1}) \big). 
     \label{lessabswupdate}
\end{align}
This is the form of the algorithm that would be directly obtained from a
derivation in an abstract Hilbert space using the inner product
$\inner{\spcdot}{\spcdot}_W$ and corresponding induced norm
$\smallnorm{\spcdot}_W$.

The algorithm developed below in Section~\ref{sec:stochProgFWALM} requires
calculation of the gradient of (a special case of) the quadratic term
in~\eqref{lessabsapproxmin} and~\eqref{lessabswapproxmin}.  The following
lemma provides the necessary general formula:

\begin{lemma} \label{lem:quadGradient}
Let $V$ be a linear subspace of $\real^m$, $M$ be any $m \times n$ matrix, and
$W$ be an $m \times m$ positive definite symmetric matrix. 
\begin{align*}
(\forall\,x\in \real^m) \qquad
\nabla \! \left[ \frac{1}{2}\Big( \sdist{W\!}_V(Mx) \Big)^2 \right]
  &= M\transpose W \big( \sproj{W\!}_{V\conjspace{W}}(Mx) \big) \\
  &= M\transpose W \big((\identity -\, \sproj{W\!}_V)(Mx) \big).
\end{align*}
\end{lemma}
\begin{proof}
Defining the matrix $Z$ as in Proposition~\ref{prop:wproj}, one
has for any $x\in\real^n$ that
\begin{align}
\frac{1}{2}\Big( \sdist{W\!}_V(Mx) \Big)^2
  &= \frac{1}{2} \norm{(\identity-Z)Mx}_W^2 \nonumber \\
  &= \frac{1}{2} \big((\identity - Z) Mx \big)\transpose W (\identity - Z) Mx \nonumber \\
  &= \frac{1}{2} x\transpose M\transpose (\identity - Z)\transpose W (\identity - Z) Mx. \label{grad3}
\end{align}
Since Proposition~\ref{prop:wproj}\ref{item:wproj:selfAdj} asserts that
$(\identity-Z)\transpose W = W (\identity - Z)$ and Proposition~\ref{prop:wproj}\ref{item:wproj:idem} states that $(\identity - Z)^2 = (\identity - Z)$,
\begin{align*}
M\transpose (\identity - Z)\transpose W (\identity - Z) M 
&= M\transpose W (\identity - Z)^2 M = M \transpose W (\identity - Z) M.
\end{align*}
Substituting the resulting (symmetric) matrix for the one in~\eqref{grad3} and
differentiating,
\begin{align*}
\nabla \! \left[ \frac{1}{2}\Big( \sdist{W\!}_V(Mx) \Big)^2 \right]
&= \nabla \! \left[ 
                 \frac{1}{2} x\transpose 
                 \big( M \transpose W (\identity - Z) M\big) x \right]
= M \transpose W (\identity - Z) M x,
\end{align*}
which, once again referring to Proposition~\ref{prop:wproj}, is equivalent to
both claimed expressions for the gradient.
\end{proof}

\section{Objective-gap inexact ADMMs}
\label{sec:admms}
This section considers variants of the ADMM (alternating direction method of
multipliers) algorithm.  Here, it is most convenient to use
Fenchel-Rockafellar duality~\cite{FenchNotes,RockThesis,RocFenchDual}.  
In this setting, one is given two functions
$f:\real^n \to \real \cup \{+\infty\}$ and $g:\real^m \to \real \cup
\{+\infty\}$, along with an $m\times n$ real matrix $M$.  The corresponding 
primal optimization problem is 
\begin{equation} \label{fenchbase}
\min_{x\in\real^n} \big\{ f(x) + g(Mx) \big\},
\end{equation}
and the corresponding dual problem is
\begin{equation} \label{fencheldual}
\min_{p\in\real^m}\big\{ f^*(-M\transpose p) + g^*(p) \big\},
\end{equation}
where ``$*$'' denotes the convex conjugacy operation as in the previous
section.  Defining the respective primal and dual objective functions
\begin{align}
P(x) &\doteq f(x) + g(Mx) & Q(p) &\doteq f^*(-M\transpose p) + g^*(p)
\end{align}
one has the following standard duality result; a proof may be found, for example,
in~\cite[Chapter 15]{BauComBook}.

\begin{proposition}[Fenchel-Rockafellar strong duality]
\label{prop:fenchobj}
For $x^* \in \real^n$ and $p^* \in \real^m$, the following conditions are equivalent:
\begin{enumerate}[label={(\roman*)}, nosep]
\item \label{prop:fenchobj:nogap} $P(x^*) + Q(p^*) = 0$
\item \label{prop:fenchobj:primalsuff} $-M\transpose p^* \in \partial f(x^*)$ 
and $p^* \in \partial g(Mx^*)$
\item \label{prop:fenchobj:dualsuff} $x^* \in \partial f^*(-M\transpose p^*)$ 
and $Mx^* \in \partial g^*(p^*)$.
\end{enumerate}
Whenever any of these equivalent conditions hold, $x^*$ must be optimal for
the primal problem~\eqref{fenchbase} and $p^*$ must be optimal for
the dual problem~\eqref{fencheldual}.
\end{proposition}

As has been known since the work of Gabay~\cite{Gab83}, the ADMM class of
algorithms is an application of ``Douglas-Rachford'' (DR) splitting methods
for maximal monotone operators~\cite{LioMer79}
to~\eqref{fencheldual}.\footnote{Since Douglas and Rachford~\cite{DouRac56}
only proposed an extremely narrow special case of this algorithm, a more
accurate name might be ``Lions-Mercier splitting.''}  The following
proposition translates this insight to the preconditioned generalized DR
splitting method of Proposition~\ref{prop:genDRrs}.  It 
resembles~\cite[Theorem 8]{EckBer92}, except for the presence of the
preconditioning matrix $W$ and a small but crucial generalization in the
approximation criterion for the $f$ minimization step.

\begin{proposition} \label{prop:convergeADMM}
Consider problem~\eqref{fenchbase}, let $W$ be any
symmetric positive definite $m \times m$ matrix, let
$\{\delta_x^k\}_{k=1}^{\infty} \subset \real_+$ be summable and
$\{d_z^k\}_{k=0}^{\infty}
\subset \real^m$ be such that $\sum_{k=0}^{\infty} \smallnorm{d_z^k} \leq
\infty$, and suppose that $\{\nu_k\}_{k=0}^\infty
\subset \real$ is such that $\inf_k \nu_k > 0$ and $\sup_k \nu_k < 2$.
For any constant scalar $c>0$ and arbitrary given initial $z^0, p^0 \in
\real^m$, suppose that the sequences $\{z^k\}_{k=0}^\infty,
\{p^k\}_{k=0}^\infty \subset \real^m$ and $\{x^k\}_{k=1}^\infty, \{\dot
x^k\}_{k=1}^\infty \subset \real^m$ conform to the following conditions for all
$k\geq 0$:
\begin{align}
&\norm{Mx^{k+1} - M\dot x^{k+1}} \leq \delta_x^k, \text{~where~}
\dot x^{k+1} \in 
   \argmin_{x\in\real^n} 
         \left\{f(x) + \inner{p^k}{Mx}  + \frac{c}{2}\norm{Mx - z^k}^2_W \right\}
         \label{admmx} \\
&z^{k+1} = 
   \argmin_{z\in\real^m} \left\{ 
             g(z) - \inner{p^k}{z} 
               + \frac{c}{2}\norm{\nu_k Mx^{k+1} + (1-\nu_k)z^k-z}^2_W
           \right\} + d_z^{k}
         \label{admmz} \\
&p^{k+1} = p^k + c W \big(\nu_k Mx^{k+1} + (1-\nu_k)z^k - z^{k+1}\big). 
         \label{admmupdate}
\end{align}
If strong duality holds for $f$, $g$, and $M$, then $\{p^k\}$ converges to a
solution $p^*$ of the dual problem $\min_{p\in\real^m}\big\{ f^*(-M\transpose
p) + g^*(p) \big\}$, while $\{z^k\}$ and $\{Mx^k\}$ converge to some
$z^* \in \partial g^*(p^*)$ such that $-z^* \in
\partial\big(f^* \circ (-M\transpose)\big)(p^*)$.  If the regularity condition $\ri
\im \partial f \cap \im M\transpose \neq \emptyset$ also holds, then $z^* =
Mx^*$, where $x^*$ is some solution of the primal problem $\min_{x\in\real^n}
\big\{ f(x) + g(Mx) \big\}$.

If the dual problem has no solution but there exists some $\bar p \in \real^m$ at
which the dual objective is finite, at least one of the sequences $\{p^k\}$ or
$\{z^k\}$ must be unbounded.
\end{proposition}
\noindent Before commencing the proof, a few remarks are in order:
\begin{enumerate}
\item It is not necessary in practice to actually compute the exact subproblem
solution $\dot x^{k+1}$ in~\eqref{admmx}.  It is sufficient only to verify
that $x^{k+1}$ is with a distance $\delta_x^k$ of some exact solution, after
applying the linear operator $M$ to both vectors.  Another way of expressing
the same condition is
\begin{equation}
\dist\!\left(Mx^{k+1}, M \argmin_{x\in\real^n} 
         \left\{f(x) + \inner{p^k}{Mx}  + \frac{c}{2}\norm{Mx - z^k}^2 \right\}\right)
\leq \delta_x^k.
\end{equation}
If $\ker M \neq \{0\}$, this condition is weaker than requiring that 
\begin{equation}
\left\{
\dist\!\left(x^{k+1}, \argmin_{x\in\real^n} 
         \left\{f(x) + \inner{p^k}{Mx}  + \frac{c}{2}\norm{Mx - z^k}^2 \right\}\right)
\right\}
\end{equation}
be summable, which is effectively the condition required in~\cite[Theorem
8]{EckBer92}.  The difference between these two conditions, although it may
appear minor, is critical to the analysis of objective-gap-based inexact ADMM
algorithms in Proposition~\ref{prop:convergeADMMog} to follow.

\item Much as observed following Proposition~\ref{prop:genDRrs}, the constant $c$ is
not strictly necessary in \eqref{admmx}-\eqref{admmupdate}.  By replacing $W
\leftarrow cW$ and then $c \leftarrow 1$, the algorithm with $c=1$ can be made
to produce exactly the same iterates as in any $c \neq 1$ case.  Therefore, $c$
is included mainly for historical reasons.
\end{enumerate}

\begin{proof} 
Define the matrix $S \doteq W^{-1}$ and the set-valued operators $A \doteq
\partial\big(f^* \circ (-M\transpose)\big)$ and $B \doteq \partial g^*$.
Since $g$ is proper, $g^*$ is closed proper convex, so $B$ is maximal
monotone.  Similarly, $f^*$ is closed proper convex.  It follows that $f \circ
(-M\transpose)$ is closed and convex.  With regard to $f \circ (-M\transpose)$
being proper, there are two primary situations considered in the hypothesis:
\begin{enumerate}[nosep]
\item Strong duality holds. In this case $Q(p^*)$ must be finite at some optimal
solution $p^*$ of the dual problem, in which case $f^*(-M\transpose p^*)$ must
be finite, and therefore $f^* \circ (-M\transpose)$ is proper.
\item The dual problem has no solution but there exists some $\bar p$ for
which the dual objective $f^*(-M\transpose \bar p) + g^*(\bar p) < \infty$.
Then $f^*(-M\transpose \bar p) < \infty$ and so $f^* \circ (-M\transpose)$ is
proper.
\end{enumerate}
In either case, $f^* \circ (-M\transpose)$ is proper.  Thus, it is closed proper
convex, and so $A = \partial\big(f^* \circ (-M\transpose)\big)$ is maximal
monotone.  In conclusion, $A$ and $B$ are both maximal monotone operators
$\real^n \toset \real^n$ in all situations allowed in the hypothesis.

Continuing, define
\begin{align} \label{admmproofdefqr}
(\forall\,k\geq 0) &&
q^{k+1} &\doteq p^k + c W (Mx^{k+1} - z^k) &
r^{k+1} &\doteq p^k + c W \big(\nu_k Mx^{k+1} + (1-\nu_k)z^k \big).
\end{align}
The core of the proof is to establish that the sequences $\{p^k\}$,
$\{q^k\}$, $\{r^k\}$, and $\{z^k\}$ evolve according to the generalized DR
splitting procedure~\eqref{drcycleArs}-\eqref{drcycleBdualrs},
with choices of $\{e_x^k\}$ and $\{e_z^k\}$ that meet the assumptions of
convergence result in Proposition~\ref{prop:genDRrs}. To begin this process,
define
\begin{equation*}
(\forall\,k\ \geq 0) \quad \dot q^{k+1} \doteq p^k + c W (M\dot x^{k+1} - z^k).
\end{equation*}
The next step is to verify the following claim:
\begin{equation} \label{admmclaim}
(\forall\,k \geq 0) \quad 
   \dot q^{k+1} = \sprox{S}_{cA}(p^k - c W z^k)
                = \sprox{S}_{cA}(p^k - c S^{-1} z^k),
\end{equation}
that is, that $\dot q^{k+1}$ is the exact result of the $\prox$ operation
in~\eqref{drcycleArs}.

To establish the claim, start by fixing any $k \geq 0$. Since $f$ is closed
proper convex and $a_k : x \mapsto \inner{p^k}{Mx} + \frac{c}{2}\norm{Mx -
z^k}_W^2$ is convex and both finite and differentiable everywhere, one has,
much as in the proof of Lemma~\ref{lem:dualProx}, that since $\dot x^{k+1}$
minimizes $f + a_k$, 
\begin{align}
0 &\in \partial f(\dot x^{k+1}) + \nabla a_k(\dot x^{k+1}) \nonumber \\
&=  \partial f(\dot x^{k+1}) + M\transpose p^k 
         + cM\transpose W (M\dot x^{k+1} - z^k)  \nonumber \\
&= \partial f(\dot x^{k+1}) + M\transpose\big(p^k + c W ( M \dot x^{k+1} - z^k)\big) 
         \nonumber \\
&= \partial f(\dot x^{k+1}) + M\transpose \dot q^{k+1}. \label{dualfhappy}
\end{align}
Rearranging, $-M\transpose \dot q^{k+1} \in \partial f(\dot x^{k+1})$, which
by the properties of conjugate functions --- again, see for
example~\cite[Theorem 23.5]{Roc70book} --- means that $\dot x^{k+1} \in
\partial f^*(-M\transpose \dot q^{k+1})$.
Multiplying this inclusion by $-M$ yields
\begin{equation} \label{trickyf}
-M \dot x^{k+1} \in -M \partial f^*(-M\transpose \dot q^{k+1}) 
= (-M\transpose)\transpose \partial f^*(-M\transpose \dot q^{k+1})
\subseteq \partial \big(f^* \circ (-M\transpose)\big)(\dot q^{k+1}),
\end{equation}
where the ``$\subseteq$'' relationship follows from the linear chain rule for
subgradients; see for example~\cite[Theorem 23.9]{Roc70book}.  Furthermore, 
\begin{align*}
S \dot q^{k+1} + c(-M \dot x^{k+1})
   &= S\big((p^k + c W (M\dot x^{k+1} - z^k)\big) + c(-M \dot x^{k+1}) \\
   &= Sp^k + c S W (M\dot x^{k+1} - z^k) - c M \dot x^{k+1} \\
   &= Sp^k + c M\dot x^{k+1} - c z^k - c M \dot x^{k+1} 
        && [\text{since~} S = W^{-1}] \\
   &= Sp^k  - c z^k \\
   &= S(p^k - c S^{-1} z^k).
\end{align*}
This equation, together with $-M \dot x^{k+1} \in \partial \big(f^* \circ
(-M\transpose)\big)(\dot q^{k+1}) = A(\dot q^{k+1})$, means that $q^{k+1} =
\sprox{S}_{cA}(p^k - c S^{-1} z^k)$, matching~\eqref{admmclaim}.  Since the
choice of $k \geq 0$ was arbitrary, \eqref{admmclaim} is verified.

Next, define
\begin{align*}
e_x^{k} &\doteq q^{k+1} -\sprox{S}_{cA}(p^k - c S^{-1} z^k) = q^{k+1} - \dot q^{k+1},
\end{align*}
which by simple algebraic rearrangement implies that~\eqref{drcycleArs} holds.  By
the definitions $q^{k+1}$ and $\dot q^{k+1}$,
\begin{align*}
(\forall\, k\geq 0) \quad
e_x^{k} &= \big(p^k + c W (Mx^{k+1} - z^k)\big) 
                 - \big(p^k + c W (M \dot x^{k+1} - z^k)\big) \\
                 &= c W (Mx^{k+1} - M \dot x^{k+1}),
\end{align*}
and consequently, using~\eqref{admmx},
\begin{equation*}
(\forall\, k\geq 0) \quad 
\norm{e_x^{k}} = \norm{c W (Mx^{k+1} - M \dot x^{k+1})}
   \leq c \norm{W} \norm{Mx^{k+1} - M \dot x^{k+1}}
   \leq c \norm{W} \delta_x^k.
\end{equation*}
Since $\{\delta_x^k\}$ was assumed summable in the hypothesis,
$\big\{\smallnorm{e_x^{k}}\big\}$ is summable.

Next, starting with the definition of $\{r^k\}$, one has
\begin{align*}
(\forall\, k\geq 0) \quad
r^{k+1} &= p^k + c W \big(\nu_k Mx^{k+1} + (1-\nu_k)z^k \big) \\
        &= \nu_k\big(p^k + cW(Mx^{k+1} - z^k)\big) + (1-\nu_k) p^k + cWz^k \\
        &= \nu_k q^{k+1} + (1-\nu_k) p^k + cWz^k \\
        &= \nu_k q^{k+1} + (1-\nu_k) p^k + c S^{-1} z^k,
\end{align*}
which matches~\eqref{drcycleRelaxrs}.

The next step in the core of the proof is to show that~\eqref{drcycleBrs}
holds for some choice of $\{e_z^k\}_{k=1}^\infty$ whose norms form a summable
sequence.  To this end, fix any $k \geq 0$ and define
\begin{align} 
\dot z^{k+1} &\doteq 
   \argmin_{z\in\real^m} \left\{ 
             g(z) - \inner{p^k}{z} 
               + \frac{c}{2}\norm{\nu_k Mx^{k+1} + (1-\nu_k)z^k-z}_W^2
           \right\}
         \label{admmproofexactz} \\
\dot p^{k+1} &\doteq p^k + c W \big(\nu_k Mx^{k+1} + (1-\nu_k) z^k - \dot z^{k+1}\big), 
         \label{admmproofexactp}
\end{align}
the respective values of $z^{k+1}$ and $p^{k+1}$ that would be
computed if $e_z^k = 0$.  Substituting the definition of $\dot z^{k+1}$
into~\eqref{admmz} immediately yields $z^{k+1} = \dot z^{k+1} +
d_z^{k}$ and therefore $d_z^{k} = z^{k+1} - \dot z^{k+1}$.
From the optimality of $\dot z^{k+1}$ in~\eqref{admmproofexactz} and the
definition of $\dot p^{k+1}$ in~\eqref{admmproofexactp},
\begin{equation*}
    0 \in \partial g(\dot z^{k+1}) - p^k + c W \big(\dot z^{k+1} 
                                                       - \nu_k Mx^{k+1} 
                                                       - (1-\nu_k) \dot z^k \big)
            = \partial g(\dot z^{k+1}) - \dot p^{k+1}
\end{equation*}
so $\dot p^{k+1} \in \partial g(\dot z^{k+1})$, and therefore $\dot z^{k+1}
\in \partial g^*(\dot p^{k+1}) = B(\dot p^{k+1})$.  Furthermore, using the definitions of
$\dot z^{k+1}$ and $r^{k+1}$ and that $S$ and $W$ are inverses,
\begin{align*}
S \dot p^{k+1} + c \dot z^{k+1}
  &= S \Big( p^k + c W \big(\nu_k Mx^{k+1} + (1-\nu_k) z^k - \dot z^{k+1}\big)\Big)
        + c \dot z^{k+1} \\
  &= S p^k + c S W \big(\nu_k Mx^{k+1} + (1-\nu_k) z^k\big) - c S W \dot z^{k+1} 
        + c \dot z^{k+1} \\
  &= S p^k + c\big(\nu_k Mx^{k+1} + (1-\nu_k) z^k\big)- c \dot z^{k+1} + c \dot z^{k+1} \\
  &= S p^k + c\big(\nu_k Mx^{k+1} + (1-\nu_k) z^k\big) \\
  &= S \Big(p^k + c W \big(\nu_k Mx^{k+1} + (1-\nu_k) z^k\big)\Big) \\
  &= S r^{k+1},
\end{align*}
which means in conjunction with $\dot z^{k+1} \in B(\dot p^{k+1})$ that $\dot
p^{k+1} = \sprox{S}_{cB}(r^{k+1})$.  Define
\begin{equation*}
e_z^{k} \doteq p^{k+1} - \dot p^{k+1} = p^{k+1} - \sprox{S}_{cB}(r^{k+1}),
\end{equation*}
whereby~\eqref{drcycleBrs} immediately holds.  Then,
substituting for $p^{k+1}$ with~\eqref{admmupdate} and for $\dot p^{k+1}$
with~\eqref{admmproofexactp},
\begin{align*}
\norm{e_z^{k}} &= \norm{p^{k+1} - \dot p^{k+1}} \\
 &= \bigg\|p^k + c W \big(\nu_k Mx^{k+1} + (1-\nu_k)z^k - z^{k+1}\big) \\
 & \qquad \qquad \qquad
    - \Big( p^k + c W\big(\nu_k Mx^{k+1} 
               + (1-\nu_k) z^k - \dot z^{k+1}\big) \Big)\bigg\| \\
&= \norm{c W z^{k+1} - c W \dot z^{k+1})} 
\leq c \norm{W} \norm{z^{k+1} - \dot z^{k+1}} = c \norm{W} \norm{d_z^k}.
\end{align*}
Since this holds for all $k$ and $\big\{ \smallnorm{d_z^k} \big\}$ is summable
the sequence $\big\{\smallnorm{e_z^k}\big\}_{k=1}^\infty$ is also summable.  

It has already been shown that~\eqref{drcycleBrs} holds for arbitrary $k \geq
0$, so the final step in the core of the proof is to
establish~\eqref{drcycleBdualrs}.
From the definitions of $r^{k+1}$ and $p^{k+1}$,
\begin{align*}
(\forall\, k\geq 0) \quad
\frac{1}{c}S(r^{k+1} - p^{k+1})
&=
\frac{1}{c} S \bigg(p^k + c W \big(\nu_k Mx^{k+1} + (1-\nu_k)z^k \big) \\
&\qquad\qquad\qquad
              - \Big( p^k + c W\big(\nu_k Mx^{k+1} + (1-\nu_k)z^k - z^{k+1} \Big)\bigg) \\
&= \frac{1}{c} \cdot c SW z^{k+1} = z^{k+1},
\end{align*}
so~\eqref{drcycleBdualrs} is verified and the core of the proof is complete.

With the definitions in~\eqref{admmproofdefqr}, the analysis up to this point
has established that $\{p^k\}$, $\{q^k\}$, $\{r^k\}$, and $\{z^k\}$ evolve
according to the generalized preconditioned DR splitting
procedure~\eqref{drcycleArs}-\eqref{drcycleBdualrs} for $A =
\partial\big(f^* \circ (-M\transpose)\big)$ and $B = \partial g^*$ both
maximal monotone, with the error sequences $\{e_x^k\}$ and $\{e_z^k\}$ being
norm summable. Proposition~\ref{prop:genDRrs} therefore applies.

If $f$, $g$, and $M$ satisify strong duality, then, consulting
Proposition~\ref{prop:fenchobj}, there exist vectors $p^*
\in \real^m$ and $x^* \in \real^n$ such that
\begin{align*}
-M\transpose p^* &\in \partial f(x^*) &
p^* &\in \partial g(Mx^*),
\end{align*}
or equivalently
\begin{align*}
x^* &\in \partial f^*(-M\transpose p^*) &
Mx^* &\in \partial g^*(p^*).
\end{align*}
In particular, since $\partial\big(f^* \circ (-M\transpose)\big)(p^*)
\supseteq -M \partial f^*(-M\transpose p^*)$ \cite[Theorem 23.9]{Roc70book},
one then has
\begin{align*}
A(p^*) + B(p^*)
&= \partial\big(f^* \circ (-M\transpose)\big)(p^*) + \partial g(p^*) \\
&\supseteq -M \partial f^*(-M\transpose p^*) + \partial g(p^*) \\
&\ni -Mx^* + Mx^* \\
&= 0. 
\end{align*}
Thus, the operator $A + B$ possesses at least one root.  Therefore, the
convergent case of Proposition~\ref{prop:genDRrs} applies, so 
\begin{itemize}[nosep]
\item $\{p^k\}$ and $\{q^k\}$ converge to some $p^* \in \real^m$ with $0 \in
A(p^*) + B(p^*)$.
\item $\{z^k\}$ converges to some $z^* \in B(p^*)$ such that $-z^* \in A(p^*)$.
\end{itemize}
Such a $p^*$ must be a solution to the dual problem.  The assertions about the
convergence of $\{z^k\}$ now follow from the definitions of $A$ and $B$.
Further, rearranging~\eqref{admmupdate} leads to
\begin{equation*}
(\forall\,k \geq 0) \quad p^{k+1} - p^k + c(z^{k+1} - z^k) = c\nu_k(Mx^{k+1} - z^k).
\end{equation*}
Since $\{p^k\}$ and $\{z^k\}$ both converge, it follows that $c\nu_k(Mx^{k+1}
- z^k) \to 0$.  Since $\{\nu_k\}$ is bounded away from zero, it then follows $Mx^{k+1}
- z^k \to 0$, which in conjunction with $z^k \to z^*$ means that $Mx^k \to
  z^*$ as well.  The first set of assertions in the proposition have thus been
  established.

Now assume that both strong duality and the regularity condition $\ri \im \partial f
\cap \im M\transpose \neq \emptyset$ hold.  Since
\begin{equation}
\im \partial f = \dom \partial f^* \supseteq \ri \dom f^*, \label{intermedri}
\end{equation}
where the equation follows for example from~\cite[Corollary 23.5.1]{Roc70book}
and the ``$\supseteq$'' condition from~\cite[Theorem 23.4]{Roc70book}, it
follows that
\begin{align*}
    \ri \dom f^* \cap \im (-M\transpose)
    &= \ri \dom f^* \cap \im M\transpose \\
    &= \ri \ri \dom f^* \cap \im M\transpose \\
    &\supseteq \ri \im \partial f^*  \cap \im M\transpose
       &&[\text{by~\eqref{intermedri}}] \\
    &\neq \emptyset.
\end{align*}
Therefore, \cite[Theorem 23.9]{Roc70book} asserts that
\begin{equation} \label{nicesubgrad}
\partial\big(f^* \circ (-M\transpose)\big)
= (-M\transpose)\transpose \circ \partial f^* \circ (-M\transpose) 
= -M \circ \partial f^* \circ (-M\transpose),
\end{equation}
so it must be possible to express $-z^*
\in A(p^*) = \partial\big(f^* \circ (-M\transpose)\big)(p^*)$ in the form
$-Mx^*$, for some $x^* \in \partial f^*(-M\transpose p^*)$.  From $-z^* =
-Mx^*$, one immediately has $z^* = Mx^*$, and so
\begin{align*}
x^* &\in \partial f^*(-M\transpose p^*) & z^* &= Mx^* \in \partial g^*(p^*).
\end{align*}
The two inclusions above are exactly the conditions in
Proposition~\ref{prop:fenchobj}\ref{prop:fenchobj:dualsuff}, so that
proposition asserts that $x^*$ is a solution to the primal problem.  

Any root of $A + B$ clearly solves the dual problem, so if the dual problem
has no solution, no such roots can exist.  In this case,
Proposition~\ref{prop:genDRrs} asserts that at least one of $\{p^k\}$ or
$\{z^k\}$ is unbounded.
\end{proof}

It is now relatively simple matter to combine Proposition~\ref{prop:genDRrs} with
Lemma~\ref{lem:partialStrongDistance} to produce a version of the generalized
inexact ADMM that uses objective gaps for its subproblem approximation
criteria:

\begin{proposition} \label{prop:convergeADMMog}
Consider problem~\eqref{fenchbase}, let $W$ be any symmetric positive definite
$m \times m$ matrix, let $\{\tau_k^x\}_{k=1}^{\infty},
\{\tau_k^z\}_{k=1}^{\infty} \subset \real_+$ be sequences such that
\begin{align} \label{sqrtsum}
\sum_{i=1}^\infty \sqrt{\tau_k^x} &< \infty &
\sum_{i=1}^\infty \sqrt{\tau_k^z} &< \infty.
\end{align}
Also suppose that $\{\nu_k\}_{k=0}^\infty
\subset \real$ is such that $\inf_k \nu_k > 0$ and $\sup_k \nu_k < 2$.
For any constant scalar $c>0$ and arbitrary given initial $z^0, p^0 \in
\real^m$, suppose that the sequences $\{z^k\}_{k=0}^\infty,
\{p^k\}_{k=0}^\infty \subset \real^m$ and $\{x^k\}_{k=1}^\infty \subset
\real^m$ conform to the following recursions for all $k\geq 0$:
\begin{align}
x^{k+1} &\in 
   \approxmin{\tau_k^x}{x\in\real^n} 
         \left\{f(x) + \inner{p^k}{Mx}  + \frac{c}{2}\norm{Mx - z^k}^2_W \right\}
         \label{admmxog} \\
z^{k+1} &\in 
   \approxmin{\tau_k^z}{z\in\real^m} \left\{ 
             g(z) - \inner{p^k}{z} 
               + \frac{c}{2}\norm{\nu_k Mx^{k+1} + (1-\nu_k)z^k-z}^2_W
           \right\}
         \label{admmzog} \\
p^{k+1} &= p^k + c W \big(\nu_k Mx^{k+1} + (1-\nu_k)z^k - z^{k+1}\big). 
         \label{admmupdateog}
\end{align}
If strong duality holds for $f$, $g$, and $M$, then $\{p^k\}$ converges to a
solution $p^*$ of the dual problem $\min_{p\in\real^m}\big\{ f^*(-M\transpose
p) + g^*(p) \big\}$, while $\{z^k\}$ and $\{Mx^k\}$ converge to some
$z^*\in\real^m$ such that $z^* \in \partial g^*(p^*)$ and $-z^* \in
\partial\big(f^* \circ (-M\transpose)\big)(z^*)$.  If the regularity condition $\ri
\im \partial f \cap \im M\transpose \neq \emptyset$ is also true, then $z^* =
Mx^*$, where $x^*$ is some solution of the primal problem $\min_{x\in\real^n}
\big\{ f(x) + g(Mx) \big\}$, and all limit points of $\{x^k\}$ are
optimal primal solutions.

If the dual problem has no solution but there exists some $\bar p \in \real^m$ at
which the dual objective is finite, at least one of the sequences $\{p^k\}$ or
$\{z^k\}$ must be unbounded.
\end{proposition}

\begin{proof}
The hypotheses of the proposition are identical to the previous one, except
that~\eqref{admmxog}-\eqref{admmzog} take the place
of~\eqref{admmx}-\eqref{admmz}. Other than the claim regarding the limit
points of $\{x^k\}$, the conclusions are also identical, and will all follow
from the previous result if there exist sequences $\{\dot x^k\}_{k=1}^{\infty}
\subset \real^n$, $\{\delta_x^k\}_{k=1}^{\infty} \subset
\real_+$, and $\{d_z^k\}_{k=1}^{\infty} \subset \real^m$ such
that~\eqref{admmx}-\eqref{admmz} are satisfied for all $k\geq 0$,
$\{\delta_x^k\}$ is summable, and $\big\{\smallnorm{d_z^k}\big\}$ is summable.

To begin demonstrating this, let $R$ be the symmetric matrix square root of
$W$, following a similar construction to that of $Q$ immediately prior to
Proposition~\ref{prop:genPPArs}, so that $R^2 = R\transpose R = W$.
Temporarily fix any $k \geq 0$ and define $\phi_k : \real^n \to \real \cup
\{\infty\}$ to be the minimand in~\eqref{admmxog}, that is,
\begin{align}
\phi_k(x) 
  &\doteq f(x) + \inner{p^k}{Mx}  + \frac{c}{2}\norm{Mx - z^k}^2_W \nonumber \\
  &= f(x) + \inner{p^k}{Mx} + \frac{c}{2}(Mx - z^k)\transpose W (Mx - z^k) \nonumber \\
  &= f(x) + \inner{p^k}{Mx} + \frac{c}{2}(z^k)\transpose W z^k
                            - c (z^k)\transpose W Mx 
                            + \frac{c}{2}(Mx)\transpose W Mx \nonumber \\
  &= f(x) + \inner{p^k - cWz^k}{Mx} 
          + \frac{c}{2}\norm{z^k}^2_W
          + \frac{c}{2}x\transpose M\transpose W M x \nonumber \\
  &= f(x) + \biginner{M\transpose(p^k - cWz^k)}{x}
          + \frac{c}{2}\norm{z^k}^2_W
          + \frac{c}{2} x \transpose M\transpose R\transpose R M x \nonumber \\
  &= f(x) + \biginner{M\transpose(p^k - cWz^k)}{x}
          + \frac{c}{2}\norm{z^k}^2_W
          + \frac{c}{2} \norm{RMx}^2.   \label{fsubfinalfunc}
\end{align}
and let $\dot x^{k+1} \in \argmin_{x\in\real^n} 
\big\{ \phi_k(x) \big\}$  (assumed nonempty by the standing assumption on the
subproblems). Lemma~\ref{lem:partialStrongDistance} with $h$ being the closed
proper convex function given by the first three terms
in~\eqref{fsubfinalfunc}, $A = RM$, and $\sigma=c$ then asserts that
\begin{align*}
\norm{RMx^{k+1} - RM \dot x^{k+1}} 
  &\leq \sqrt{\frac{2}{c}\big(\phi_k(x^{k+1}) - \phi_k(\dot x^{k+1})\big)}
  \leq \sqrt{\frac{2 \tau_k^x}{c}},
\end{align*}
the second inequality following because $x^{k+1}$ is a $\tau_k^x$-approximate
minimizer of $\phi_k$.  Then,
\begin{align}
\norm{Mx^{k+1} - M \dot x^{k+1}}
  &= \norm{R^{-1}(RMx^{k+1} - RM \dot x^{k+1})} \nonumber \\
  &\leq \norm{R^{-1}} \norm{RMx^{k+1} - RM \dot x^{k+1}} \nonumber \\
  &\leq \norm{R^{-1}} \sqrt{\frac{2 \tau_k^x}{c}} 
  = \lambda_{\min}^{-1/2} \sqrt{\frac{2 \tau_k^x}{c}} 
  = \left(\sqrt{\frac{2}{c \lambda_{\min}}}\right) \sqrt{\tau_k^x}, \label{fbound}
\end{align}
where $\lambda_{\min} > 0$ is the smallest eigenvalue of $W$.  Setting
$\delta_x^k \doteq \sqrt{2/c \lambda_{\min}} \cdot \sqrt{\tau_k^x}$ for all $k
\geq 0$, it follows from the assumed summability of
$\big\{\sqrt{\tau_k^x}\big\}$ that $\{\delta_x^k\}$ is summable.  Observing
that~\eqref{fbound} holds for arbitrary $k \geq 0$,
one concludes that~\eqref{admmx} holds for all $k\geq 0$, with with
$\{\delta_x^k\}$ summable.

Next, consider~\eqref{admmzog}.  For all $k \geq 0$ define $t^k \doteq \nu_k
Mx^{k+1} + (1-\nu_k)z^k$ and $\gamma_k : \real^m \to \real \cup \{\infty\}$
to be the minimand in~\eqref{admmzog}, that is,
\begin{align*}
(\forall\, k \geq 0) \quad 
\gamma_k(z)
  &\doteq g(z) - \inner{p^k}{z} + \frac{c}{2}\norm{t^k - z}_W^2 \\
  &= g(z) - \inner{p^k}{z} + \frac{c}{2} (t^k - z)\transpose W (t^k - z) \\
  &= g(z) - \inner{p^k}{z} + \frac{c}{2} (t^k)\transpose W t^k - c \inner{Wt^k}{z}
          + \frac{c}{2} z\transpose W z && \\
  &= g(z) - \inner{p^k + cWt^k}{z} + \frac{c}{2} \norm{t^k}_W^2
          +  \frac{c}{2} z\transpose R\transpose R z \\
  &= g(z) - \inner{p^k + cWt^k}{z} + \frac{c}{2} \norm{t^k}_W^2
          + \frac{c}{2} \norm{Rz}^2.
\end{align*}
Defining $\dot z^{k+1} \doteq \argmin_{z\in\real^m} \big\{ \gamma_k(z) \big\}$
for all $k \geq 0$, that is, the exact minimizer in~\eqref{admmzog}, another
application of Lemma~\ref{lem:partialStrongDistance}, this time with $A=R$,
yields, much as before,
\begin{align*}
(\forall\, k \geq 0) \quad 
\norm{Rz^{k+1} - Rz^{k+1}} 
  &\leq \sqrt{\frac{2}{c}\big(\gamma_k(x^{k+1}) - \gamma_k(\dot x^{k+1})\big)}
  \leq \sqrt{\frac{2 \tau_k^z}{c}},
\end{align*}
the second inequality following from the $\tau_k^z$-optimality of $z^{k+1}$.
Therefore,
\begin{align*}
(\forall\, k \geq 0) \quad 
\norm{z^{k+1} - \dot z^{k+1}} 
&=\norm{R^{-1}(Rz^{k+1} - R \dot z^{k+1})} \\
&\leq \norm{R^{-1}}\norm{Rz^{k+1} - R \dot z^{k+1}}
  \leq \norm{R^{-1}} \sqrt{\frac{2 \tau_k^z}{c}}
  = \left(\sqrt{\frac{2}{c\lambda_{\min}}}\right) \sqrt{\tau_k^z}.
\end{align*}
Defining $d_z^{k+1} = z^{k+1} - \dot z^{k+1}$ for all $k\geq 0$, from
which~\eqref{admmz} immediately holds, it then follows from the assumed
summability of $\big\{\sqrt{\tau_k^z}\big\}$ that $\big\{ \smallnorm{d_z^k}
\big\}$ is summable.

The claimed equivalence of~\eqref{admmxog}-\eqref{admmzog}
to~\eqref{admmx}-\eqref{admmz} has now been fully established, so the
conclusions of the proposition follow immediately from
Proposition~\ref{prop:convergeADMM}, with the exception of the claim regarding
the primal optimality of all limit points of $\{x^k\}$ (which is not present
in Proposition~\ref{prop:convergeADMM}).  This result was claimed in the case
in which it is known that $z^k \to Mx^*$, where $x^*$ is some optimal primal
solution.  Let $x^\infty$ be any limit point of $\{x^k\}$ and $\mathcal{K}$ an
infinitely large subset of $\nat$ such that $\lim_{k\to\infty,k \in
\mathcal{K}} \{ x^{k+1} \} = x^\infty$.  The assumed approximate optimality of
$x^{k+1}$ and $z^{k+1}$ for their respective subproblems implies that
\begin{multline*}
(\forall\,k\geq 0) \quad
f(x^{k+1}) + \inner{p^k}{Mx^{k+1}} + \frac{c}{2}\norm{Mx^{k+1} - z^k}^2_W - \tau_k^x
   \\ \leq f(x^*) + \inner{p^k}{Mx^*} + \frac{c}{2}\norm{Mx^* - z^k}^2_W 
\end{multline*}
\vspace{-3ex}
\begin{multline*}
(\forall\,k\geq 0) \quad
g(z^{k+1}) - \inner{p^k}{z^{k+1}} + \frac{c}{2}\norm{Mx^{k+1} - z^{k+1}}^2_W - \tau_k^z
   \\ \leq g(Mx^*) - \inner{p^k}{Mx^*} + \frac{c}{2}\norm{Mx^{k+1} - Mx^*}^2_W.
\end{multline*}
Adding these inequalities and cancelling the inner product terms on
the right yields
\begin{multline*}
(\forall\,k\geq 0) \quad
f(x^{k+1}) + g(z^{k+1}) + \inner{p^k}{Mx^{k+1} - z^{k+1}} 
   \\ + \frac{c}{2}\norm{Mx^{k+1} - z^k}^2_W 
   + \frac{c}{2}\norm{Mx^{k+1} - z^{k+1}}^2_W
   - \tau_k^x - \tau_k^z
   \\ \leq f(x^*) + g(Mx^*) + \frac{c}{2}\norm{Mx^* - z^k}^2_W 
                            + \frac{c}{2}\norm{Mx^{k+1} - Mx^*}^2_W.
\end{multline*}
Since $Mx^k \to Mx^*$ and $z^k \to Mx^*$, the limit of the right side of this
inequality is $f(x^*) + g(Mx^*)$, the optimal objective value.  For the same
reasons, the two squared-norm terms on left also converge to zero as $k\to\infty$.
Noting that $\{p^k\}$ is convergent, the inner product term on the left
converges to zero as well.  By hypothesis, $\tau_k^x \to 0$ and $\tau_k^z \to
0$, so taking the limit over $k\in\mathcal{K}$ results in
\begin{equation*}
\limsup_{\substack{k\to\infty \\ k\in \mathcal{K}}}
  \big\{ f(x^{k+1}) + g(z^{k+1}) \big\} \leq f(x^*) + g(Mx^*).
\end{equation*}
Since $\{Mx^k\}$ converges to $Mx^*$, one has $Mx^\infty = Mx^*$.  Further,
since $f$ and $g$ are closed (lower semicontinous) and $z^k \to Mx^*$ one has
\begin{align*}
f(x^\infty) + g(Mx^\infty)
  &= f\!\left(\lim_{\substack{k\to\infty \\ k\in \mathcal{K}}}x^k \right )
     + g\!\left(\lim_{\substack{k\to\infty \\ k\in \mathcal{K}}}M x^k \right ) \\
  &\leq \liminf_{\substack{k\to\infty \\ k\in \mathcal{K}}} \big\{ f(x^k) \big\}
     + \liminf_{\substack{k\to\infty \\ k\in \mathcal{K}}} \big\{ g(Mx^k) \big\} \\
  &\leq \liminf_{\substack{k\to\infty \\ k\in \mathcal{K}}} 
            \big\{ f(x^k) + g(Mx^k) \big\} \\
  &\leq \limsup_{\substack{k\to\infty \\ k\in \mathcal{K}}} 
            \big\{ f(x^k) + g(Mx^k) \big\} \\
  &\leq f(x^*) + g(Mx^*),
\end{align*}
the last inequality having been established above.  Therefore $x^\infty$ is an
optimal primal solution, and since the choice of limit point $x^\infty$ of
$\{x^k\}$ as arbitrary, the proof is complete.
\end{proof}

In the above proposition, objective gaps are used in the approximation
criteria for both subproblems.  It is of course also possible to mix criteria
in any manner that verifies the assumptions of
Proposition~\ref{prop:convergeADMM}; for example, one could use a subproblem
objective gap to verify~\eqref{admmx} but a subgradient-based criterion to
establish~\eqref{admmz}.  In the subspace-constrained and stochastic
programming applications to follow, \eqref{admmx} will be verified with and
objective gap, but the $g$ minimization will be essentially exact.

\subsection{Applications with subspace constraints}
\label{sec:subspaceConADMM}
Consider the same class of applications as in Section~\ref{sec:subspaceCon},
of the form $\min_{x\in\real^n} \set{ f(x) }{Mx
\in V}$, where $M$ is $m \times n$ as above and $V$ is a linear subspace of
$\real^m$.  The natural way to formulate this problem in Fenchel-Rockafellar
form $\min_{x\in\real^n}\big\{f(x) + g(Mx)\big\}$ is to set
\begin{equation} \label{indicV}
g(z) \doteq 
\begin{cases}
0, & \text{if~} z\in V \\
+\infty, & \text{otherwise,}
\end{cases}
\end{equation}
the convex indicator function of $V$.  In this case, $\partial g(z) =
V\orthog$ at all $z\in V$ and $\partial g(z) = \emptyset$ for all $z \not\in
V$.  For all $c>0$, the proximal map $\prox_{c\partial g}$ is the orthogonal
projector $\proj_V$ onto $V$, and for any symmetric positive definite matrix
$W$, the preconditioned proximal map $\sprox{W\!}_{c\partial g}$ is the
projector onto $V$ using the $\norm{\spcdot}_W$ norm; assume that the latter
projection operation is tractable to compute exactly. Fixing any $k \geq 0$
and letting $\tau_k^z = 0$ so that the calculation becomes exact, the $g$
minimization~\eqref{admmzog} specializes as follows, letting $t^k \doteq \nu_k
Mx^{k+1} + (1 - \nu_k) z^k$ for brevity:
\allowdisplaybreaks
\begin{align*}
z^{k+1} 
  &= \argmin_{z\in V} \left\{ - \inner{p^k}{z} + \frac{c}{2}\norm{t^k - z}_W^2 \right\} \\
  &= \argmin_{z\in V} \left\{ - \inner{p^k}{z} 
                              + \frac{c}{2}\norm{t^k}_W^2 
                              - c (t^k)\transpose W z
                              + \frac{c}{2}\norm{z}_W^2
                      \right\} \\
  &= \argmin_{z\in V} \left\{ \frac{c}{2}\norm{t^k}_W^2 
                              - c (t^k + \tfrac{1}{c}W^{-1}p^k)\transpose W z
                              + \frac{c}{2}\norm{z}_W^2
                      \right\} \\
  &= \argmin_{z\in V} \left\{ \frac{c}{2}\norm{t^k + \tfrac{1}{c}W^{-1}p^k}_W^2 
                              - c (t^k + \tfrac{1}{c}W^{-1}p^k)\transpose W z
                              + \frac{c}{2}\norm{z}_W^2
                      \right\} \\
  & \qquad [\text{since setting the first term to a different constant
                        cannot change the minimizer}] \\
  &= \argmin_{z\in V} \left\{ \frac{c}{2}\norm{z - (t^k + \tfrac{1}{c}W^{-1}p^k)}_W^2
                      \right\} \\
  &= \sproj{W\!}_V(t^k + \tfrac{1}{c}W^{-1}p^k) \\
  &= \sproj{W\!}_V(t^k) + \sproj{W\!}_V(\tfrac{1}{c}W^{-1}p^k),
\end{align*}
the last step following because $\sproj{W\!}_V$ is a linear map, as stated in
Proposition~\ref{prop:wproj}\ref{item:wproj:linear}.  If one assumes that $p^k
\in V\orthog$, meaning that $(p^k)\transpose v = 0$ for all $v\in V$, then it
follows that
\begin{equation*}
(\forall\,v\in V) \quad 
(W^{-1}p^k)\transpose W v
   = (p^k)\transpose W\invtrans W v
   = (p^k)\transpose W^{-1} W v
   = (p^k)\transpose v
   = 0,
\end{equation*}
meaning that $W^{-1}p^k \in V\conjspace{W}$ and hence that
$\sproj{W\!}_V(\tfrac{1}{c}W^{-1}p^k) = 0$.  Therefore,
\begin{equation} \label{admmconj1}
z^{k+1} = \sproj{W\!}_V(t^k) = \sproj{W\!}_V\big(\nu_k Mx^{k+1} + (1 - \nu_k) z^k\big).
\end{equation}
Using the linearity of $\sproj{W\!}_V$ and inductively assuming that $z^k \in
V$ and, one may further write
\begin{align*}
z^{k+1} &= \nu_k \big(\sproj{W\!}_V( Mx^{k+1})\big) 
                + (1 - \nu_k)\big(\sproj{W\!}_V(z^k)\big) \\
        &= \nu_k \big(\sproj{W\!}_V(Mx^{k+1})\big) + (1-\nu_k) z^k,
\end{align*}
or, introducing an intermediate variable $v^{k+1}$ to represent the $W\!$-projection,
\begin{align*}
v^{k+1} &= \sproj{W\!}_V(Mx^{k+1}) &
z^{k+1} &= \nu_k v^{k+1} + (1-\nu_k) z^k.
\end{align*}
Substituting this formula for $z^{k+1}$ into the multiplier
update~\eqref{admmconj1} then leads to
\begin{align*}
p^{k+1} 
  &= p^k + cW\Big( \nu_k Mx^{k+1} + (1 - \nu_k) z^k 
                     - \big(\nu_k v^{k+1} + (1-\nu_k) z^k\big)\Big) \\
  &= p^k + cW ( \nu_k Mx^{k+1} - \nu_k v^{k+1} ) \\
  &= p^k + \nu_k c W ( Mx^{k+1} - v^{k+1} ) \\
  &= p^k + \nu_k cW \big( \sproj{W\!}_{V\conjspace{W}}(Mx^{k+1})\big).
\end{align*}
Defining $u^{k+1} \doteq \sproj{W\!}_{V\conjspace{W}}(Mx^{k+1}) = Mx^{k+1} -
v^{k+1}$, one has $v\transpose W u^{k+1} = 0$ for all $v\in V$ since $u^{k+1}
\in V\conjspace{W}$.  Regrouping the product, $v\transpose (W u^{k+1})$ for
all $v \in V$ and hence $W u^{k+1} \in V\orthog$.  Hence, if $p^k \in
V\orthog$, then $p^{k+1} = p^k + \nu_k c W u^{k+1} \in V\orthog$. Inductively,
if $p^0 \in V\orthog$, then the sequence $\{p^k\}$ will lie entirely in
$V\orthog$, and clearly $z^k \in V$ for all $k \geq 1$ since these vectors are
all projections onto $V$.\footnote{Using more complicated expressions for the
projection step and multiplier update, it is possible with a modest amount of
additional analysis to dispense with the assumption $p^0\in V\orthog$, while
establishing that  $p^k \in V\orthog$ for all $k \geq 1$. But since the
assumption $p^0\in V\orthog$ is easily met by the customary choice of $p^0 =
0$, the exposition here adopts the simpler approach.}

Assembling the entire resulting method, one obtains the following result.
With the choice of $g$ in~\eqref{indicV}, it is readily seen that $g^*(p) = 0$
for $p\in V\orthog$ and otherwise $g^*(p) = +\infty$, so the dual problem in
the proposition is $\min_{p\in V\orthog} \big\{ f^*(-M\transpose p)
\big\}$.

\begin{proposition} \label{prop:convergeADMMogsp}
Now consider a the class of problems of the form $\min_{x\in\real^n}
\set{f(x)}{Mx \in V}$, where $f:\real^n\to\real\cup\{+\infty\}$ is closed
proper convex, $M$ is an $m \times n$ matrix, and $V$ is a linear subspace of
$\real^m$.  Let $\{\tau_k\}_{k=1}^{\infty}, \subset \real_+$ be a sequence such that
$\sum_{i=1}^\infty \sqrt{\tau_k} < \infty$, and suppose that $\{\nu_k\}_{k=0}^\infty
\subset \real$ is such that $\inf_k \nu_k > 0$ and $\sup_k \nu_k < 2$.
For any constant scalar $c>0$ and arbitrary given initial $z^0 \in V$ and $
p^0 \in V\orthog$, suppose that the sequences $\{z^k\}_{k=0}^\infty,
\{v^k\}_{k=1}^\infty \subset V$, $\{p^k\}_{k=0}^\infty \subset V\orthog$, and
$\{x^k\}_{k=1}^\infty \subset \real^n$ conform to the following recursions for
all $k\geq 0$:
\begin{align}
x^{k+1} &\in 
   \approxmin{\tau_k}{x\in\real^n} 
         \left\{f(x) + \inner{p^k}{Mx}  + \frac{c}{2}\norm{Mx - z^k}^2_W \right\}
         \label{admmxogsp} \\
v^{k+1} &= \sproj{W\!}_V(Mx^{k+1}) 
         \label{admmvogsp} \\
z^{k+1} &= \nu_k v^{k+1} + (1-\nu_k) z^k
         \label{admmzogsp} \\
p^{k+1} &= p^k + \nu_k c W (Mx^{k+1} - v^{k+1}). 
         \label{admmupdateogsp}
\end{align}
If strong duality holds, then $\{p^k\}$ converges to a solution $p^* \in
V\orthog$ of the dual problem, while $\{z^k\}$, $\{v^k\}$, and $\{Mx^k\}$
converge to some $z^*\in V$ such that $-z^* \in \partial\big(f^* \circ
(-M\transpose)\big)(p^*)$.  If the regularity condition $\ri \im \partial f
\cap \im M\transpose \neq \emptyset$ is also true, then $z^* = Mx^*$, where
$x^*$ is some solution of the primal problem $\min_{x\in\real^n} \set{f(x)}{Mx
\in V}$.  In this case, all limit points of $\{x^k\}$ are optimal primal
solutions.

If the dual problem has no solution but there exists some $\bar p \in \real^m$ at
which the dual objective is finite, at least one of the sequences $\{p^k\}$ or
$\{z^k\}$ must be unbounded.
\end{proposition}
\begin{proof}
In view of the development above, all the claimed results except the
convergence of $\{v^k\}$ follow immediately from
Proposition~\ref{prop:convergeADMMog} with $g$ as in~\eqref{indicV}, along
with $\tau_k^x = \tau_k$ and $\tau_k^z = 0$ for all $k \geq 0$, keeping in
mind that with this choice of~$g$,
\begin{align*}
    z^* &\in \partial g^*(p^*)  &&\Leftrightarrow &
    p^* &\in \partial g(z^*)    &&\Leftrightarrow &
    z^* &\in V \; \wedge \; p^* \in V\orthog.
\end{align*}
Regarding the convergence of $\{v^k\}$, since $Mx^k \to
z^* \in V$ and $\sproj{W\!}_V$ is a linear and hence continuous map,
\eqref{admmvogsp} implies that $v^k = \sproj{W\!}_V(M x^k) \to
\sproj{W\!}_V(z^*) = z^*$.
\end{proof}

\section{Stochastic programming applications}
\label{sec:stochProg}
Now consider the ``grid-and-subspace'' stochastic programming formulation
originating with~\cite{RW91}, using the following particular notation: suppose
that a stochastic programming problem is defined on a finite scenario tree
with stages indexed by $t \in 1..T$ and $\ell$ leaf nodes indexed by $i \in
1..\ell$.  For each node $N$ in the tree, let $t(N)$ denote its time stage.
Further, let $\mathcal{T}_t$ denote the set of nodes at time stage $t$ and let
$\mathcal{U} = \bigcup_{t=1}^{T-1} \mathcal{T}_t$ denote the set of non-leaf
nodes. For every node $N$, also let $\mathcal{D}(N)\subseteq 1..\ell$ denote
the set of leaf nodes that are descendents of $N$.  For each $t \in 1..T$, the
sets $\set{\mathcal{D}(N)}{N\in\mathcal{T}_t}$ form a partition of $1..\ell$.
For each $i\in 1..\ell$, denote the probability of leaf node $i$ by $\pi_i \in
(0,1]$, subject to $\sum_{i=1}^\ell \pi_i = 1$.  For each node $N$ in the
tree, let $\pi(N) = \sum_{i\in\mathcal{D}(N)} \pi_i$ denote its probability.
Assume that all zero-probability nodes have already been pruned from the
scenario tree.  The term ``scenario'' will be taken to be synonymous with
``last-stage node'' and thus with the leaves of the tree.

For simplicity (and without loss of generality), assume that every stage-$t$
node in the tree has the same number of decision variables $n_t \geq 1$, and
let $\bar n \doteq \sum_{t=1}^T n_t$ denote the full number of decision
variables associated with each scenario and $\bar m \doteq \bar n - n_T$ be
similar, excluding the last stage.

The working variables $x$ take the form a of a ``grid'', consisting of a
subvector $x_{i} \in \real^{\bar n}$ of full root-to-leaf decision variables
for each scenario $i \in 1..\ell$; thus, their total dimension is $n \doteq
\ell \bar n$.  For each $i\in 1..\ell$ and $t\in 1..T$, let $x_{it} \in
\real^{n_t}$ denote the decision variables for scenario $i$, stage $t$, and let
$x_{itj} \in \real$, $j\in 1..n_t$, denote the individual elements of this
vector. Define $\overline M$ to be the linear map that drops the last $n_T$
variables from a vector of length $\bar m$, dropping the last-stage variables
from a single scenario, that is, for any scenario $i \in 1..\ell$,
\begin{equation} \label{dropdef}
\overline Mx_{i} = \overline M (x_{i1},\ldots,x_{i,T-1},x_{iT}) 
               = (x_{i1},\ldots,x_{i,T-1}),
\end{equation}
so that
\begin{align*}
\overline{M} &= \big[ \identity_{\bar m} \;\; 0_{\bar m \times n_T}\big] &
&\text{and} &
\overline{M}\transpose &= 
\left[
\begin{array}{c}
\identity_{\bar m} \\ 0_{n_T \times\bar m }
\end{array}
\right],
\end{align*}
where $\identity_a$ denotes the $a \times a$ identity matrix and $0_{a \times
b}$ denotes the $a \times b$ zero matrix.

Let $M$ be the linear operator that, given a vector $x\in\real^n$ comprised of a
subvector of length $\bar n$ for each scenario, drops the last-stage variables
from each subvector, so that
\begin{multline} \label{fulldropdef}
Mx = M\big(x_1,\ldots,x_n\big) 
   = M \left(
           \begin{array}{c}
             \big(x_{11},\ldots, x_{1,T-1}, x_{1T}\big) \\
             \vdots \\
             \big(x_{\ell 1}, \ldots, x_{\ell,T-1}, x_{\ell,T}\big)
           \end{array}
       \right) \\
   = \left(
           \begin{array}{c}
             \overline M \big(x_{11}, \ldots, x_{1,T-1}, x_{1T}\big) \\
             \vdots \\
             \overline M \big(x_{\ell 1},\ldots, x_{\ell,T-1}, x_{\ell,T}\big)
           \end{array}
       \right) 
   = \left(
           \begin{array}{c}
             \big(x_{11},\ldots, x_{1,T-1}\big) \\
             \vdots \\
             \big(x_{\ell 1},\ldots, x_{\ell,T-1}\big)
           \end{array}
       \right),
\end{multline}
or equivalently
\begin{align*}
M &=
\underbrace{
\left[
\begin{array}{cccc}
\overline{M} \\
& \overline{M} \\
& & \ddots \\
& & & \overline{M}
\end{array}
\right]
}_{\ell \text{~times}}
&
M\transpose &=
\underbrace{
\left[
\begin{array}{cccc}
\overline{M}\transpose \\
& \overline{M}\transpose \\
& & \ddots \\
& & & \overline{M}\transpose
\end{array}
\right]
}_{\ell \text{~times}}.
\end{align*}

The Rockafellar-Wets formulation approach dedicates multiple subvectors of a
decision vector $x\in\real^n$ to each node of the scenario tree; for example,
the root node corresponds to the $\ell$ subvectors $x_{11}, \ldots, x_{\ell
1}$.  For the solution corresponding to $x$ to be implementable for the
stochastic programming problem, all these subvectors must be equal, that is,
$x_{11}=\cdots=x_{\ell 1}$, meaning that the first-stage decisions cannot
depend on knowledge of which scenario will eventually transpire.  More
generally, for any non-leaf node $N \in \mathcal{U}$, all the corresponding
stage-$t(N)$ decisions $x_{it(N)}$, for $i\in \mathcal{D}(N)$, must be
identical, meaning that the decisions at that node cannot depend on which
descendent of $N$ eventually occurs. For vectors $z$ in the range space
$\real^m$ of $M$, apply the same subvector indexing notation as for $x $ but
excluding stage $T$, and define a linear subspace $\mathcal{N} \subset
\real^m$ by
\begin{equation}
\mathcal{N} \doteq \set{z\in\real^m}{(\forall\,N\in\mathcal{U}) \;
                                      \big(\forall\,i,i' \in \mathcal{D}(N)\big) :
                                      z_{it(N)} = z_{i't(N)}}.
\end{equation}
Vectors in $z\in\mathcal{N}$ are called \emph{nonclairvoyant} or
\emph{nonanticipative} in the sense that $z_{it} = z_{i't}$ whenever
scenarios $i$ and $i'$ are indistinguishable at stage $t$.  They thus
correspond to implementable plans for stages $1..(T-1)$, requiring no
knowledge of future events.

For each $i\in 1..n$, let $f_i : \real^{\bar n} \rightarrow \real \cup \{+\infty\}$
be a closed proper convex function and consider the optimization problem
\begin{equation}
\label{abstractSP}
\begin{array}{ll}
\displaystyle{\min_{x\in\real^n}}\; 
   & \displaystyle{\sum_{i=1}^{\ell} \pi_i f_i\big(x_{i}\big)} \\
\suchthat & Mx \in \mathcal{N}. \vphantom{\displaystyle{\sum_i}}
\end{array}
\end{equation}
This optimization model can subsume any convex stochastic programming problem
defined on the scenario tree $\mathcal{T}$, in the following manner:  within
the context of the ``clairvoyant'' situation in which one knows that leaf node
$i \in 1..\ell$ of the tree will occur,  define $f_i\big(x_{i}\big) = +\infty$
whenever $x_{i}$ is infeasible, and otherwise let $f_i\big(x_{i}\big)$ be the
total cost of the plan described by $x_{i}$.  In particular,
$f_i\big(x_{i}\big)$ is $+\infty$ if $x_{i}$ violates any constraint within a
stage or any coupling constraint between stages.  For all scenarios $i\in
1..\ell$, such constraints are embedded within the objective function
of~\eqref{abstractSP}, while the explicit constraint $Mx \in \mathcal{N}$
requires the selection of an nonclairvoyant plan.  

Defining 
\begin{equation} \label{defspf}
    f(x) \doteq \sum_{i=1}^{\ell} \pi_i f_i(x_{i}),
\end{equation}
the entire stochastic programming problem may be written $\min_{x\in\real^n}
\set{f(x)}{Mx \in \mathcal{N}}$, an instance of the subspace-constrained form
considered in Sections~\ref{sec:subspaceCon} and~\ref{sec:subspaceConADMM}, so
one may apply the algorithms developed there. 

\subsection{Applying the ADMM: progressive hedging}

At present, the customary choice for solving~\eqref{abstractSP} by a proximal
algorithm is effectively to apply the ADMM, resulting in the standard
progressive hedging (PH) algorithm as proposed in~\cite{RW91}.\footnote{The
original analysis in~\cite{RW91} proved convergence of PH from first
principles rather than DR splitting or the ADMM, possibly because neither
technique was broadly known in the numerical optimization community at the
time.  However, PH is a special case of the ADMM and hence of DR splitting.}
The analysis here does the same, but with the generalized objective-gap ADMM
of Proposition~\ref{prop:convergeADMMog}, yielding an objective-gap inexact PH
algorithm, which appears previously unknown.

Expanding the dual problem
$\min_{p\in V\orthog} \big\{ f^*(-M\transpose p) \big\}$ stated just before
Proposition~\ref{prop:convergeADMMogsp}, one obtains the dual problem
\begin{equation} \label{stochProgDual}
    \min_{p \in \mathcal{N}\orthog} \left\{
        \sum_{i=1}^{\ell}
        (\pi_i f_i)^*(-\overline M\transpose p_i)
    \right\},
    \quad \text{or equivalently} \quad
    \min_{p \in \mathcal{N}\orthog} \left\{
        \sum_{i=1}^{\ell}
        \pi_i f_i^*\Big(\big(-(1/\pi_i)p_i,0\big)\Big)
    \right\},
\end{equation}
where the equivalence follows from the form of $\overline M$ and the
conjugate-scaling formula $(\alpha h)^* = \alpha h^*(\spcdot / \alpha)$ for
any function $\real^n \to \real \cup \{+\infty\}$ and scalar $\alpha > 0$; see
for example~\cite[Proposition 13.23(i)]{BauComBook}.

To obtain an algorithm closely generalizing the usual presentation and
implementation of PH, the matrix $W$ is constructed from two components, the
scenario probabilities $\pi_i$, $i\in 1..\ell$, and an $\bar m \times \bar m$
positive definite diagonal matrix $\Rho$ (capital rho) that reflects the
possibly varying scaling of the non-last-stage decision variables within each
scenario (the original formulation of PH effectively set $\Rho = \rho
\identity_{\bar m}$, some positive multiple of the $\bar m \times \bar m$
identity matrix, but acceptable practical performance typically requires
variable-by-variable scaling, often heuristically chosen in the early
iterations of the method). $W$ is then constructed as the $m \times m = \ell
\bar m \times \ell \bar m$ diagonal matrix
\begin{equation} \label{spscaling}
W \doteq 
\left[
\begin{array}{cccc}
\pi_1 \Rho \\
& \pi_2 \Rho \\
& & \ddots \\
& & & \pi_{\ell} \Rho
\end{array}
\right].
\end{equation}
With this choice of $W$, now consider applying the subspace-constrained
ADMM~\eqref{admmxogsp}-\eqref{admmupdateogsp} to $f$ and $M$ as constructed
above, with $c=1$ and $V = \mathcal{N}$.  Applying the same indexing scheme to
the $p$ variables as to the $z$ variables (in turn inherited from the indexing
scheme for $x$ as mentioned above) and expanding
\begin{align*}
\inner{p^k}{Mx} &= \sum_{i=1}^{\ell} \biginner{p^k_{i}}{\overline{M} x_{i}} &
\norm{Mx - z^k}_W^2 &= \sum_{i=1}^{\ell} \norm{\overline{M} x_{i} - z^k_{i}}_{\pi_i \Rho}^2,
\end{align*}
the first step~\eqref{admmxogsp} in the method may be written
\begin{equation} \label{phmin1}
x^{k+1} \in 
   \approxmin{\tau_k}{x\in\real^{\bar n}} 
         \left\{\sum_{i=1}^{\ell} \pi_i f_i(x_{i}) +
                \sum_{i=1}^{\ell} \inner{p^k_{i}}{\overline{M} x_{i}} + 
                \frac{1}{2}\sum_{i=1}^{\ell} 
                              \norm{\overline{M} x_{i} - z^k_{i}}_{\pi_i \Rho}^2 \right\}.
\end{equation}
For each $k\geq 0$, let $\tau_{1k}, \tau_{2k}, \ldots, \tau_{\ell k} \geq 0$
be such that $\sum_{i=1}^\ell \tau_{ik} = \tau_k$.  Then~\eqref{phmin1} may be
satisfied by~$\ell$ independent approximate minimization operations
\begin{align*}
(\forall\,i\in 1..\ell) \quad
x_i^{k+1} &\in 
   \approxmin{\tau_{ik}}{x\in\real^{\bar n}} 
         \left\{\pi_i f_i(x_{i}) +
                \inner{p^k_{i}}{\overline{M} x_{i}} + 
                \frac{\pi_i}{2} \norm{\overline{M} x_{i} - z^k_{i}}_{\Rho}^2 \right\} \\
    &= \approxmin{\tau_{ik}}{x\in\real^{\bar n}} 
         \left\{\pi_i f_i(x_{i}) +
                \sum_{t=1}^{T-1}
                \sum_{j=1}^{n_t}
                 \left( p^k_{itj} x_{itj} + 
                \frac{\pi_i \rho_{tj}}{2} (x_{itj} - z^k_{itj})^2 \right) \right\},
\end{align*}
where $\rho_{tj}$ is the diagonal element of $\Rho$ for stage $t$, decision
variable $j$. The next step in the algorithm is $W\!$-projection of $Mx^{k+1}$
onto $V=\mathcal{N}$ in~\eqref{admmvogsp}.  Applying the same indexing scheme
from the $z$ and $p$ variables to the $v$ variables,
\begin{align*}
v^{k+1} &= \sproj{W\!}_{\mathcal{N}}(Mx^{k+1}) \\
  &= \argmin_{v\in \mathcal{N}}
       \left\{
         \sum_{i=1}^{\ell} (v_i - \overline{M} x_i^{k+1})\transpose
                           (\pi_i \Rho)
                           (v_i - \overline{M} x_i^{k+1})
       \right\} \\
  &= \argmin_{v\in \mathcal{N}}
       \left\{
         \sum_{i=1}^{\ell} 
         \sum_{t=1}^{T-1} (v_{it} - x_{it}^{k+1})\transpose
                          (\pi_i \Rho_t)
                          (v_{it} - x_{it}^{k+1})
       \right\} \\
  &= \argmin_{v\in \mathcal{N}}
       \left\{
         \sum_{i=1}^{\ell} 
         \sum_{t=1}^{T-1}
         \sum_{j=1}^{n_t} \pi_i \rho_{tj}
                          (v_{itj} - x_{itj}^{k+1})^2
       \right\},
\end{align*}
Grouping this summation by (non-leaf) tree node yields tree yields
\begin{align*}
v^{k+1} 
  &= \argmin_{v\in \mathcal{N}}
       \left\{
         \sum_{N\in\mathcal{U}}
         \sum_{i \in \mathcal{D}(N)}       
         \sum_{j=1}^{n_{t(N)}} \pi_i \rho_{t(N)j}
                          (v_{it(N)j} - x_{it(N)j}^{k+1})^2
       \right\},
\end{align*}
If $v \in \mathcal{N}$, then for each non-leaf node $N \in \mathcal{U}$
at and each variable index $j \in 1..n_{t(N)}$, the variables $v_{it(N)j},
i\in\mathcal{D}(N)$ must take some common value $a_{Nj}$, so the goal is
to minimize, substituting each $v_{it(N)j} \leftarrow a_{Nj}$ and
interchanging the order of the last two summations,
\begin{equation*}
         \sum_{N\in\mathcal{U}}
         \sum_{j=1}^{n_{t(N)}} 
         \sum_{i \in \mathcal{D}(N)}       
             \pi_i \rho_{t(N)j} \big(a_{Nj} - x_{it(N)j}^{k+1}\big)^2
\quad \text{over~} \quad
 \{a_{Nj}\}_{\begin{subarray}{l} N\in\mathcal{U} \\ j\in 1..n_{t(N)}\end{subarray}}.
\end{equation*}
This problem decomposes by tree node $N \in \mathcal{U}$ and variable $j \in
1..n_{t(N)}$, so one may solve the problem by the independent calculations
\begin{equation*}
(\forall\,N\in\mathcal{U}) \;
(\forall\,j\in 1..n_{t(N)}) \quad
a_{Nj} = \argmin_{t\in\real} 
              \left\{ 
                 \sum_{i\in \mathcal{D}(N)}
                 \pi_i \rho_{t(N)j} \big(t - x_{it(N)j}^{k+1} \big)^2
              \right\}.
\end{equation*}
Setting the derivatives of the minimands to zero yields
\pagebreak[3]
\begin{align*}
(\forall\,N\in\mathcal{U}) \;
(\forall\,j\in 1..n_{t(N)}) &&&&
\sum_{i\in \mathcal{D}(N)} 
    2 \pi_i \rho_{t(N)j} \big(v_{Nj} - x_{it(N)j}^{k+1}\big) &= 0 \\
&& \Leftrightarrow &&
    2 \rho_{t(N)j} \sum_{i\in \mathcal{D}(N)} 
          \pi_i \big(v_{Nj} - x_{it(N)j}^{k+1}\big) &= 0 \\
&& \Leftrightarrow &&
\sum_{i\in \mathcal{D}(N)} \pi_i \big(v_{Nj} - x_{it(N)j}^{k+1}\big) &= 0 \\
&& \Leftrightarrow &&
\sum_{i\in \mathcal{D}(N)} \pi_i v_{Nj} 
    &= \sum_{i\in \mathcal{D}(N)} \pi_i x_{it(N)j}^{k+1} \\
&& \Leftrightarrow && 
    \pi(N) v_{Nj} &= \sum_{i\in \mathcal{D}(N)} \pi_i x_{it(N)j}^{k+1} \\
&& \Leftrightarrow && 
    v_{Nj} &= \frac{1}{\pi(N)}\sum_{i\in \mathcal{D}(N)} \pi_i x_{it(N)j}^{k+1}.
\end{align*}
Therefore, the calculation of $v^{k+1}$ may be expressed as
\begin{equation} \label{phprojformula}
(\forall\,N\in \mathcal{U}) \;
\big(\forall\,i\in \mathcal{D}(N)\big) \;
(\forall\,j\in 1..n_{t(N)}) \quad
v_{it(N)j}^{k+1} = \frac{1}{\pi(N)}
   \sum_{i'\in \mathcal{D}(N)} \pi_{i'} x_{i't(N)j}^{k+1}.
\end{equation}
That is, at each non-leaf node $N \in \mathcal{U{}}$, the vectors
$v_{it(N)}^{k+1}, i\in\mathcal{D}(N)$ are obtained by taking the
$\pi$-weighted average, across the scenarios in $\mathcal{D}(N)$, of the
corresponding elements in $x^{k+1}$.

The remaining two steps of the algorithm,
\eqref{admmzogsp}-\eqref{admmupdateogsp}, may then be straightforwardly
implemented by
\allowdisplaybreaks[0]
\begin{align}
(\forall\,i\in 1..\ell) \;
\big(\forall\,t\in 1..(T-1)\big) \;
(\forall\,j \in 1..n_t) \quad
z_{itj}^{k+1} &= \nu_k v_{itj}^{k+1} + (1-\nu_k) z_{itj}^k 
  \label{phz1} \\
p_{itj}^{k+1} &= p_{itj}^k + \nu_k \pi_i \rho_{tj} \big(x_{itj}^{k+1} - v_{itj}^{k+1}\big).
  \label{phupdatep}
\end{align}
Finally, to arrive at a method more closely matching the original formulation
of PH in~\cite{RW91}, one may replace the dual variables $p^k$ with variables
$w^k$ such that $w_i^k = (1/\pi_i) p_i^k$ for all $k\geq 0$ and $i\in
1..\ell$. These rescaled Lagrange multiplier estimates reside in the subspace
$\mathcal{N}^*$ of $\real^m$ defined by
\begin{align} \label{defweightedcomplement}
    \mathcal{N}^* &\doteq \set{w\in \real^n}{ (\forall\,N\in\mathcal{U}) \;
                                              \sum_{i \in \mathcal{D}(N)} \!\!\!
                                                \pi_i w_{it(N)} = 0 } \\
    &= \set{w\in \real^n}{ (\forall\,N\in\mathcal{U}) \; 
                           (\forall\,j \in 1..n_{t(N)}) \;
                                              \sum_{i \in \mathcal{D}(N)} \!\!\!
                                                \pi_i w_{it(N)j} = 0 }.
                \nonumber
\end{align}
Then, replacing $p_i^k$ with $\pi_i w_i^k$ (since $w_i = (1/\pi_i) p_i$) and
dividing through by $\pi_i$, the $f_i$ minimizations may be written
\pagebreak[2]
\begin{align*}
(\forall\,i\in 1..\ell) \quad
x_i^{k+1} &\in 
   \approxmin{\tau_{ik}}{x\in\real^{\bar n}} 
         \left\{\pi_i f_i(x_{i}) +
                \inner{\pi_i w^k_{i}}{\overline{M} x_{i}} + 
                \frac{\pi_i}{2} \norm{\overline{M} x_{i} - z^k_{i}}_{\Rho}^2 \right\} \\
    &= \approxmin{\left(\frac{\tau_{ik}}{\pi_i}\right)}{x\in\real^{\bar n}} 
          \left\{f_i(x_{i}) +
                \inner{w^k_{i}}{\overline{M} x_{i}} + 
                \frac{1}{2} \norm{\overline{M} x_{i} - z^k_{i}}_{\Rho}^2 \right\} \\
    &= \approxmin{\sigma_{ik}}{x\in\real^{\bar n}} 
         \left\{f_i(x_{i}) +
                \sum_{t=1}^{T-1}
                \sum_{j=1}^{n_t}
                \left( w^k_{itj} x_{itj} + 
                \frac{\rho_{tj}}{2} (x_{itj} - z^k_{itj})^2 \right) \right\},
\end{align*}
where $\sigma_{ik} \doteq \tau_{ik}/\pi_i$ for all $i\in 1..\ell$ and $k\geq
0$. These scaled tolerances meet the same summability assumptions as the
$\tau_{ik}$.  The same substitution of $\pi_i w_i^k$ for $p_i^k$ and
division-by-$\pi_i$ operations applied to the multiplier update $p_{itj}^{k+1}
= p_{itj}^k + \nu_k \pi_i \rho_{tj}
\big(x_{itj}^{k+1} - v_{itj}^{k+1}\big)$ result in
\begin{align*}
(\forall\,i\in 1..\ell) \;
\big(\forall\,t\in 1..(T-1)\big) \;
(\forall\,j \in 1..n_t) \qquad
w_{itj}^{k+1} &= w_{itj}^k + \nu_k \rho_{tj} \big(x_{itj}^{k+1} - v_{itj}^{k+1}\big).
\end{align*}
Summarizing, the entire method with the rescaled dual variables is
\begin{align}
x_i^{k+1}  
   &\in \approxmin{\sigma_{ik}}{x\in\real^{\bar n}} 
         \left\{f_i(x_{i}) +
                \sum_{t=1}^{T-1}
                \sum_{j=1}^{n_t}
                \left( w^k_{itj} x_{itj} + 
                \frac{\rho_{tj}}{2} \big(x_{itj} - z^k_{itj}\big)^2 \right) \right\}
    && \forall\,i \in 1..\ell \label{phmin} \\
v_{it(N)j}^{k+1} &= \frac{1}{\pi(N)}\sum_{i'\in \mathcal{D}(N)} \pi_i x_{i't(N)j}^{k+1}
    && \forall\, N, i, j \label{phproj} \\
z_{itj}^{k+1} &= \nu_k v_{itj}^{k+1} + (1-\nu_k) z_{itj}^k 
    && \forall\, i, t, j \label{phz} \\
w_{itj}^{k+1} &= w_{itj}^k + \nu_k \rho_{tj} \big(x_{itj}^{k+1} - v_{itj}^{k+1}\big)
    && \forall\, i, t, j. \label{php}
\end{align}
The quantifiers in~\eqref{phproj} are abbreviated for readability and are
identical to those at the beginning of~\eqref{phprojformula}.  The quantifiers
in~\eqref{phz} and~\eqref{php} are similarly condensed and in both cases their
full forms are $\forall\, i\in 1..\ell, \forall\,t\in 1..(T-1), \forall\, j\in
1..n_t$. The algorithm is the progressive hedging method of~\cite{RW91} with
two generalizations: the presence of the overrelaxation factors $\nu_k$ and
inexact solution of the scenario minimizations subject to the objective gap
tolerances $\sigma_{ik}$.  With the rescaled multipliers, the scenario
probabilities $\pi_i$ appear only in the projection step~\eqref{phproj}, as
in~\cite{RW91}.  For each $i\in 1..\ell$, the ojective tolerance sequences
$\{\sigma_{ik}\}_{k=0}^{\infty}$ should be nonnegative and summable. When
$\nu_k \equiv 1$, the sequences $\{z^k\}$ and $\{v^k\}$ become identical, and
$\{v^k\}$ may be eliminated, matching the original algorithm in~\cite{RW91}.
The minimization step in~\eqref{phmin} may be more compactly expressed as
\begin{align}
x_i^{k+1}  
   &\in \approxmin{\sigma_{ik}}{x\in\real^{\bar n}} 
         \left\{f_i(x_{i}) + \inner{w_i^k}{\overline M x_i} 
                 + \frac{1}{2} \norm{\overline M x_i - z_i^k}_{\Rho}^2 \right\}
    && \forall\,i \in 1..\ell, \label{phmin2} 
\end{align}
recalling that the linear operator $\overline M$ discards the last $n_T$
elements of its argument, leaving only the variables associated with the first
$T-1$ stages.

\begin{proposition} \label{prop:convergePH} 
Consider a stochastic programming problem expressed as~\eqref{abstractSP},
with the accompanying notation earlier in this section.  For each $i\in
1..\ell$, let $\{\sigma_{ik}\}_{k=1}^{\infty}, \subset \real_+$ be a sequence
such that $\sum_{k=1}^\infty \sqrt{\sigma_{ik}} < \infty$, and suppose that
$\{\nu_k\}_{k=0}^\infty \subset \real$ is such that $\inf_k \nu_k > 0$ and
$\sup_k \nu_k < 2$. For any constant scalar $c>0$ and arbitrary given initial
$z^0 \in \mathcal{N}, w^0 \in \mathcal{N}^*$, where $\mathcal{N}^*$ is as
defined in~\eqref{defweightedcomplement}, suppose that the sequences
$\{z^k\}_{k=0}^\infty, \{v^k\}_{k=1}^\infty \subset \mathcal{N}$,
$\{w^k\}_{k=0}^\infty \subset \mathcal{N}^*$ and $\{x^k\}_{k=1}^\infty \subset
\real^n$ conform to the recursions~\eqref{phmin}-\eqref{php} for all $k \geq
0$.  

If strong duality holds, then $\{w^k\}$ converges to some $w^* \in
\mathcal{N}^*$ such that $p^* \in \mathcal{N}\orthog$ defined by $p_{ijt}^*
\doteq \pi_i w_{itj}^*$ for all $i \in 1..\ell$, $t \in 1..(T-1)$, and $j \in
1..n_t$ is a solution to the dual problem~\eqref{stochProgDual}, while
$\{z^k\}$, $\{v^k\}$, and $\{Mx^k\}$ converge to some $z_i^*\in\mathcal{N} $
such that
\begin{equation} \label{spsubgrad1}
(\forall\,i\in 1..\ell) \quad
-z_i^* \in \partial \big((\pi_i f_i)^* \circ (-\overline M\transpose)\big)(\pi_i w_i^*).
\end{equation}
If the regularity condition $\ri \im \partial f \cap \im M\transpose \neq
\emptyset$ is also true, then $z^* = Mx^*$, where $x^*$ is some solution of
the primal problem $\min_{x\in\real^n} \set{f(x)}{Mx \in \mathcal{N}}$.  In
this case, all limit points of $\{x^k\}$ are optimal primal solutions and
\begin{equation} \label{spsubgrad2}
    (\forall\,i \in 1..\ell) \quad
    (\exists\,x_{iT}^* \in \real^{n_T}) \quad
    (-w_i^*,0) \in \partial f_i(z_i^*, x_{iT}^*).
\end{equation}
If the dual problem has no solution but there exists some $\bar p \in \real^m$
at which the dual objective is finite, at least one of the sequences $\{w^k\}$
or $\{z^k\}$ must be unbounded.  
\end{proposition}
\begin{proof}
Let $\tau_{ik} \doteq \pi_i \sigma_{ik}$ and $p_i^k \doteq \pi_i w_i^k$ for
all $i\in 1..\ell$ and $k \geq 0$.  Also define the matrix $\Pi \doteq
\diag(\pi_1\identity_{\bar m},\ldots,\pi_{\ell}\identity_{\bar m})$, so that
$p^k \doteq \big\{ (p_1^k,\ldots,p_\ell^k)\big\} = \Pi w^k$ for all $k\geq 0$.

Then $\{p^k\}$, $\{v^k\}$, $\{z^k\}$, and $\{x^k\}$ evolve according to the
recursions~\eqref{phmin1}-\eqref{phupdatep}, with the sequences
$\{\tau_{ik}\}_{k=0}^{\infty}$ being summable since the
$\{\sigma_{ik}\}_{k=0}^{\infty}$ are summable. Thus,
Proposition~\ref{prop:convergeADMMogsp} applies with $V=\mathcal{N}$ and the
particular choices of $f$ in~\eqref{defspf} and $M$ in~\eqref{fulldropdef}.
If strong duality holds, Proposition~\ref{prop:convergeADMMogsp} asserts that:
\begin{enumerate}[nosep]
\item $\{p^k\}$ converges to a dual solution $p^* \in \mathcal{N}\orthog$.  It
is then immediate that $\{w^k\} = \{\Pi^{-1} p^k\}$ converges to $w^* \doteq
\Pi^{-1} p^* \in \mathcal{N}^*$ with the claimed properties.

\item \label{item:ph:primalConverge} $\{z^k\}$, $\{v^k\}$, and $\{Mx^k\}$ converge to some $z^* \in
\mathcal{N}$ such that $-z^* \in \partial\big(f^* \circ
(-M\transpose)\big)(p^*)$.  Substituting $p^* = \Pi w^*$ in this inclusion and
expanding the chosen form of $f$ from~\eqref{defspf} yields the claimed
inclusion~\eqref{spsubgrad1} when coupled with standard results for
subgradients of separable sums of functions, for example~\cite[Proposition
16.9]{BauComBook}.

\item When the regularity condition also holds, $z^*$ is of the form $Mx^*$,
where $x^*$ is an optimal primal solution, and all limit points of $\{x^k\}$ are
primal optimal solutions.
\end{enumerate}
Turning to~\eqref{spsubgrad2}, the regularity condition $\ri \im \partial f
\cap \im M\transpose \neq \emptyset$ implies, as above
in~\eqref{nicesubgrad}, that $\partial\big(f^* \circ (-M
\transpose) \big) = -M \circ \partial f^* \circ (-M \transpose)$.  Therefore, again using the specific choices of $f$ and $M$, \eqref{spsubgrad1} implies for all $i\in 1..\ell$ that
\begin{align*}
&& &&
-z_i^* &\in -M \partial \big(\pi_i f_i)^* (-\overline M\transpose p_i^*) \\
\Leftrightarrow && && 
 z_i^* &\in M \partial \big(\pi_i f_i)^* (-\overline M\transpose p_i^*) \\
\Leftrightarrow && 
  (\exists\,x_{iT}^* \in \real^{n_T}) &&
  (z_i^*, x_{iT}^*) &\in \partial (\pi_i f_i)^*\big((-p_i^*,0)\big) \\
\Leftrightarrow && 
  (\exists\,x_{iT}^* \in \real^{n_T}) &&
  (-p_i^*,0) &\in \partial (\pi_i f_i)(z_i^*, x_{iT}^*) 
  \quad\; \text{\cite[Corollary 23.5.1]{Roc70book}} \\
\Leftrightarrow && 
  (\exists\,x_{iT}^* \in \real^{n_T}) &&
  (-p_i^*,0) &\in \pi_i \partial f_i(z_i^*, x_{iT}^*) \\
\Leftrightarrow && 
  (\exists\,x_{iT}^* \in \real^{n_T}) &&
  \big(-(1/\pi_i) p_i^*,0\big) &\in \partial f_i(z_i^*, x_{iT}^*),
\end{align*}
which is equivalent to~\eqref{spsubgrad2} since $p_i^* = \pi_i w_i^*$.

Finally, consider the second alternative in the hypothesis, that the dual
function is finite somewhere but has no solution.  Then
Proposition~\ref{prop:convergeADMMogsp} asserts that at least one of $\{p^k\}$
or $\{z^k\}$ is unbounded.  Since $w^k = \Pi^{-1} p^k$ for all $k$ and
$\Pi^{-1}$ is nonsingular, unboundedness of $\{p^k\}$ is equivalent to
unboundedness of $\{w^k\}$, so at least one of $\{w^k\}$ or $\{z^k\}$ must be
unbounded.
\end{proof}

The regularity condition $\ri \im \partial f \cap \im M\transpose \neq
\emptyset$ may appear somewhat technical and difficult to verify in practice.
However, it automatically satisfied in a case that covers many practical
applications, namely when the the set of feasible solutions to each scenario
subproblem is bounded.  This condition is equivalent to $\dom f_i$ being
bounded for each $i\in 1..\ell$.  In turn, the effective domain $\dom f =
\bigtimes_{i=1}^{\ell} \dom f_i$ of $f$ given in~\eqref{defspf} is  bounded.
Then $\dom \partial f \subseteq \dom f$ is bounded, and standard results for
surjectiveness of monotone operators such as in~\cite{Rock69} imply that $\im
\partial f = \real^n$, leading to $\ri \im \partial f \cap \im M\transpose =
\real^n \cap \im M\transpose = \im M\transpose \neq \emptyset$, so the
regularity condition holds.

\subsection{Applying the ALM}

Using the same choice of $W$, a conceivable alternative to the generalized PH
method~\eqref{phmin}-\eqref{php} is to apply the ALM
method~\eqref{lessabsapproxmin}-\eqref{lessabsupdate} from
Section~\ref{sec:subspaceCon}.  The recursions
in~\eqref{lessabsapproxmin}-\eqref{lessabsupdate}, setting $V = \mathcal{N}$,
require formulas for $\big(\sdist{W\!}_{\mathcal{N}}(Mx)\big)^2$ and
$\sproj{W\!}_{\mathcal{N}\conjspace{W}}(Mx)$.  Exploiting the formula already
developed for $\sproj{W\!}_{\mathcal{N}}(Mx)$ in~\eqref{phprojformula}, one has,
for any $x\in \real^n$, that
\begin{align*}
\Big( \sdist{W\!}_V(Mx) \Big)^2 &= \norm{\sproj{W\!}_{\mathcal{N}\conjspace{W}}(Mx)}_W^2 \\
&= \norm{Mx - \sproj{W\!}_{\mathcal{N}}(Mx)}_W^2 \\
&= \sum_{N\in\mathcal{U}}
   \sum_{j=1}^{n_{t(N)}} 
   \sum_{i \in \mathcal{D}(N)}       
             \pi_i \rho_{t(N)j} 
    \left(x_{it(N)j} 
       - \frac{1}{\pi(N)} \sum_{i'\in \mathcal{D}(N)} \pi_{i'} x_{i't(N)j}\right)^2 \\
&= \sum_{N\in\mathcal{U}}
   \sum_{j=1}^{n_{t(N)}} 
   \sum_{i \in \mathcal{D}(N)}       
             \pi_i \rho_{t(N)j} \big(x_{it(N)j} - a_{Nj}(x)\big)^2,
\end{align*}
with the definition
\begin{equation} \label{averageop}
(\forall\,N \in \mathcal{U}) \;\;
\big(\forall\,j \in 1..t(N)\big) \;\;
(\forall\,x\in \real^n) \quad
a_{Nj}(x) \doteq \frac{1}{\pi(N)} \sum_{i\in \mathcal{D}(N)} \pi_{i} x_{it(N)j},
\end{equation}
that is, $a_{Nj}(x)$ is the probability-weighted average of decision variable
$j$ at tree node $N$. Using the formula for
$\big(\sdist{W\!}_{\mathcal{N}}(Mx)\big)^2$, along with the specific forms of
$f$, $M$, and $W$, makes the minimand in~\eqref{lessabsapproxmin} equal to
\begin{multline*}
\sum_{i=1}^{\ell} \big(
   \pi_i f_i(x_{i}) + \inner{p_i^k}{x_i}
   \big)
   + \frac{c_k}{2}
     \sum_{N\in\mathcal{U}}
     \sum_{j=1}^{n_{t(N)}} 
     \sum_{i \in \mathcal{D}(N)}  
           \pi_i \rho_{t(N)j}
           \left(
            x_{it(N)j} 
              - a_{Nj}(x)
           \right)^2 \\
= 
\sum_{i=1}^{\ell} \left(
   \pi_i f_i(x_{i}) + 
   \sum_{t=1}^T
   \sum_{j=1}^{n_T}
     p_{itj}^k x_{itj} 
       + \frac{c_k \pi_i \rho_{tj}}{2} 
           \left(
            x_{itj} 
              - a_{N(i,t)j}(x)
           \right)^2 \right),
\end{multline*}
where $N(i,t)$ is the unique time-$t$ tree node from which leaf scenario $i$
is a descendant.  Making similar substitutions into~\eqref{lessabsupdate}, one
arrives at the algorithm recursions, for all $k \geq 0$, 
\begin{align}
x^{k+1} &\in \approxmin{\delta_k}{x\in \real^n}
      \left\{
         \sum_{i=1}^{\ell} \left(
   \pi_i f_i(x_{i}) + 
   \sum_{t=1}^T
   \sum_{j=1}^{n_T}
     p_{itj}^k x_{itj} 
       + \frac{c_k \pi_i \rho_{tj}}{2} 
           \left(
            x_{itj} 
              - a_{N(i,t)j}(x)
           \right)^2 \right) \right\}  \label{spalmminp}  \\
p_{itj}^{k+1} &= p_{itj}^k + \nu_k \pi_i \rho_{t(N)j} \left(
                        x_{itj}^{k+1} 
                          - a_{N(i,t)j}(x^{k+1})
                       \right)
                       \quad (\forall\, i\in 1..\ell) \;
                             (\forall\, t\in 1..T) \;
                             (\forall\, j \in 1..n_t).  \label{spalmupdatep}
\end{align}
The convergence properties of this method are given by
Proposition~\ref{prop:absALMapproxObj}.  Rescaling the Lagrange multiplier estimates as $w_i^k = (1/\pi_i) p_i^k$ for all $k\geq 0$ and $i\in 1..\ell$ produces the equivalent method
\begin{align}
x^{k+1} &\in \approxmin{\delta_k}{x\in \real^n}
      \left\{
         \sum_{i=1}^{\ell} \pi_i \left(
           f_i(x_{i}) + 
           \sum_{t=1}^T
           \sum_{j=1}^{n_T}
             w_{itj}^k x_{itj} 
               + \frac{c_k \rho_{tj}}{2} 
                   \left(
                    x_{itj} 
                      - a_{N(i,t)j}(x)
                   \right)^2 \right) 
      \right\}  \label{spalmmin} \\
w_{itj}^{k+1} &= w_{itj}^k + \nu_k \rho_{t(N)j} \left(
                        x_{itj}^{k+1} 
                          - a_{N(i,t)j}(x^{k+1})
                       \right)
                       \quad (\forall\, i\in 1..\ell) \;
                             (\forall\, t\in 1..T) \;
                             (\forall\, j \in 1..n_t), \label{spalmupdate}
\end{align}
which more closely resembles the usual presentation of progressive hedging.

Ordinarily, this method is unappealing for large-scale or parallel computation
because the subproblem objective in the minimization step is not separable.
Specifically, expanding the formula for $a_{N(i,t)j}(x)$ reveals that for every
non-leaf node $N \in \mathcal{U}$, leaf scenarios $i,i' \in \mathcal{D}(N)$
with $i\neq i'$, and $j\in 1..n_{t(N)}$, the subproblem objective contains a
``cross'' term proportional to $x_{it(N)j} x_{i't(N)j}$. Thus, such methods
are typically not considered, even though, being based on the ALM rather than
the ADMM, they potentially require many fewer iterations than PH.  However,
the next section will formulate a plausible application.

\section{Stochastic mixed-integer programs and \\ Frank-Wolfe subproblem solvers}
Now consider stochastic programming problems expressed in the
form~\eqref{abstractSP}, in the special case that
\begin{equation} \label{convexifiedfi}
(\forall\,i\in 1..\ell) \quad \quad
f_i(x_i) \doteq
\begin{cases}
h_i(x_i), \;\; & x_i \in \conv K_i \\
+\infty, & \text{otherwise,}
\end{cases}
\end{equation}
where, for each $i\in 1..\ell$, the function $h_i : \real^{\bar n} \to \real$
is a continuously differentiable and convex, while $K_i \subset \real^{\bar
n}$ is a (potentially very large) finite set.  The resulting problem may be
expressed as
\begin{equation} \label{lagrel}
\begin{array}{crclll}
\displaystyle{\min_{x_1,\ldots,x_n}} & 
\multicolumn{1}{l}{\sum_{i=1}^n \pi_i h_i(x_i)} \\
\suchthat & M(x_1,\ldots,x_n) &\in &\mathcal{N} \\
& x_i &\in & \conv K_i && i=1,\ldots,\ell,
\end{array}
\end{equation}
and is a convex relaxation of the discrete stochastic programming problem
(without the ``$\conv$'' operations applied to the $K_i$)
\begin{equation} \label{discreteSP}
\begin{array}{crclll}
\displaystyle{\min_{x_1,\ldots,x_n}} & 
\multicolumn{1}{l}{\sum_{i=1}^n \pi_i h_i(x_i)} \\
\suchthat & M(x_1,\ldots,x_n) &\in &\mathcal{N} \\
& x_i &\in & K_i && i=1,\ldots,\ell.
\end{array}
\end{equation}
The optimal value of the convex problem~\eqref{lagrel} is identical to the
standard Lagrangian bound on the optimal value of the discrete
problem~\eqref{discreteSP}, namely
\begin{equation} \label{lagdual}
\max_{(p_1,\ldots,p_n)\in\mathcal{N}\orthog} \left\{
\sum_{i=1}^n 
       \min_{x_i\in K_i} \big\{ \pi_i h_i(x_i) + p_i\transpose M_i x_i \big\} \right\}.
\end{equation}
Assume that each $K_i$ takes a form such that, for any $d_i \in \real^{\bar
n}$, the discrete linear optimization problem
\begin{equation} \label{lmo}
\begin{array}{ll}
\displaystyle{\max_{x_i\in\real^{\bar n}}} & \inner{d_i}{x_i} \\
\suchthat & x_i \in K_i
\end{array}
\end{equation}
is possible to perform (although perhaps time consuming).  The principal
envisioned application is when $K_i$ consists of all integer-feasible vertex
solution to a mixed-integer linear program (MILP) with a bounded feasible set.
In this case, \eqref{lmo} may be solved by invoking a standard MILP solver,
which serves as a ``linear minimization oracle'' (LMO). In this case,
\eqref{lagdual} is a Lagrangian bound on the optimal value of a
stochastic integer programming problem, a bound that is typically far stronger
than the continuous relaxation of the stochastic program's extensive form.
Further, the primal variable values obtained in computing such bounds are
often useful in computing high-quality feasible solutions to the integer
stochastic program.  The topic of this section is formulating algorithms to
closely approximate the Lagrangian bound.

While~\eqref{lagrel} is a convex programming problem theoretically suited to
the algorithms described in the previous section, direct solution of the
resulting PH subproblems~\eqref{phmin} or ALM subproblems~\eqref{spalmmin} is
generally not possible (even inexactly) due to the lack of any tractably sized
description of the convex sets $\conv K_i$. However, when solving
linear-objective problems of the form~\eqref{lmo} is possible, one may
entertain approximately solving such subproblems by some variant of the
Frank-Wolfe (FW), as proposed in~\cite{BCDELL17}.  FW methods, which date back
to~\cite{FW56}, typically measure their progress toward optimality by an
objective gap, hence the connection to the analysis earlier in this paper; a
recent comprehensive survey of FW methods is~\cite{FWBook25}.  The following
proposition states the classic Frank-Wolfe gap calculation, also allowing for
inexact results from the LMO:

\begin{lemma} \label{lem:fwGap}
For any positive integer $q$, suppose that $\emptyset \neq K \subset \real^q$
and $h:\real^q \to \real$ is a continuously differentiable convex function.
For some $\bar x \in \conv K$ and $\epsilon \geq 0$, also suppose that $\hat x
\in \epsilon\text{-}\!\argmin_{x\in K} \inner{\nabla h(\bar x)}{x}$, that is,
\begin{equation} \label{approxMin}
\inner{\nabla h(\bar x)}{\hat x} 
   \leq \inf_{x\in K} \big\{ \inner{\nabla h(\bar x)}{x} \big\} + \epsilon.
\end{equation}
Then
\begin{equation} \label{fwgap}
h(\bar x) \leq \inf_{y\in\conv K}\big\{ h(y) \big\} 
             + \inner{\nabla h(\bar x)}{\bar x - \hat x} + \epsilon.
\end{equation}
The quantity $\gamma \doteq \inner{\nabla h(\bar x)}{\bar x - \hat x} +
\epsilon$ on the right side of~\eqref{fwgap} must be nonnegative.  If $\gamma=0$,
then $\bar x$ minimizes $h$ over $\conv K$.  If $\epsilon=0$, then,
conversely, $\bar x$ minimizing $h$ over $\conv K$ implies that $\gamma =
\inner{\nabla h(\bar x)}{\bar x - \hat x} = 0$.
\end{lemma}
\begin{proof}
First, \eqref{approxMin} may be rearranged into
\begin{equation} \label{approxMinRearranged}
\inf_{x\in K} \big\{ \inner{\nabla h(\bar x)}{x} \big\} 
   \geq \inner{\nabla h(\bar x)}{\hat x} - \epsilon.
\end{equation}
Fix any $y\in\conv K$.  Since $h$ is convex, $\nabla h(\bar x)$ is a
subgradient of $h$ at $\bar x$, so
\begin{align*}
h(y) &\geq h(\bar x) + \inner{\nabla h(\bar x)}{y - \bar x} \\
&= h(\bar x) + \inner{\nabla h(\bar x)}{y} - \inner{\nabla h(\bar x)}{\bar x} \\
&\geq h(\bar x) + \inf_{x\in \conv K}\big\{\inner{\nabla h(\bar x)}{x}\big\} 
           - \inner{\nabla h(\bar x)}{\bar x} 
   &&\big[\text{since~} y \in \conv K \big]\\
&= h(\bar x) + \inf_{x\in  K}\big\{\inner{\nabla h(\bar x)}{x}\big\} 
           - \inner{\nabla h(\bar x)}{\bar x}  
  &&\big[\text{by linearity of~} \inner{\nabla h(\bar x)}{\spcdot}\big]\\
&\geq h(\bar x) + \inner{\nabla h(\bar x)}{\hat x} - \epsilon
    - \inner{\nabla h(\bar x)}{\bar x} 
  &&\big[\text{substituting \eqref{approxMinRearranged}}\big] \\
&= h(\bar x) + \inner{\nabla h(\bar x)}{\hat x - \bar x} - \epsilon.
\end{align*}
Rearranging the resulting inequality yields $h(\bar x) \leq h(y) +
\inner{\nabla h(\bar x)}{\bar x - \hat x} + \epsilon$.
Since the choice of $y \in \conv K$ was arbitrary, one may take the infimum of
the right side of this relation over all $y \in \conv K$ to obtain
\begin{align*}
h(\bar x) &\leq \inf_{y\in \conv K} \big\{ h(y) 
                        + \inner{\nabla h(\bar x)}{\bar x - \hat x} + \epsilon \big\} \\
&= \inf_{y\in\conv K}\big\{ h(y) \big\} 
             + \inner{\nabla h(\bar x)}{\bar x - \hat x} + \epsilon,
\end{align*}
establishing~\eqref{fwgap}.  Since $\bar x \in \conv K$, the
inequality~\eqref{fwgap} would yield an immediate contradiction if
$\inner{\nabla h(\bar x)}{\bar x - \hat x} + \epsilon < 0$, establishing the
claim that $\inner{\nabla h(\bar x)}{\bar x - \hat x} + \epsilon \geq
0$.

If $\inner{\nabla h(\bar x)}{\bar x - \hat x} + \epsilon = 0$,
then~\eqref{fwgap} implies that $h(\bar x) \leq
\inf_{y\in\conv K} \big\{ h(y) \big\}$.  Since $\bar x \in \conv K$, the
immediate conclusion is that $\bar x$ minimizes $h$ over $\conv K$.

Finally, suppose that $\epsilon=0$ and that $\bar x$ minimizes $h$ over $\conv
K$.  Then the standard necessary optimality condition
\begin{align*}
&&&&
(\forall\,x\in \conv K) && \inner{\nabla h(\bar x)}{x - \bar x} &\geq 0 
&&&&
\end{align*}
for $\bar x$ to minimize the convex function $h$ over the convex set $\conv K$
yields by taking $x = \hat x$ that
\begin{align} \label{FWoptCase}
&&
\inner{\nabla h(\bar x)}{\hat x - \bar x} & \geq 0 
&&\Leftrightarrow &
\inner{\nabla h(\bar x)}{\bar x - \hat x} & \leq 0. &&
\end{align}
When $\epsilon=0$, one has $\gamma=\inner{\nabla h(\bar x)}{\bar x - \hat x}$.
It has already been established that $\gamma \geq 0$, but
now~\eqref{FWoptCase} implies that $\gamma \leq 0$. So $\gamma =
\inner{\nabla h(\bar x)}{\bar x - \hat x} = 0$.
\end{proof}
The nonnegative quantity $\gamma \doteq \inner{\nabla h(\bar x)}{\bar x - \hat x} +
\epsilon$ will here be called the \emph{inexact Frank-Wolfe gap}; when
$\epsilon=0$, it is the classical \emph{Frank-Wolfe gap} $\inner{\nabla h(\bar
x)}{\bar x - \hat x}$.  This estimate of objective suboptimality may be
obtained whenever one (approximately) minimizes over $K$ the linear function
given by the gradient $\nabla h(\bar x)$ at some point $\bar x \in \conv K$.
Since this operation is fundamental to all FW algorithm variants, all such
variants provide objective-gap suboptimality estimates.

For the stochastic mixed-integer stochastic programming applications
envisioned here, it is important in practice to allow for inexact LMOs, that
is, $\epsilon > 0$. The LMO in these applications is a MILP solver, and MILP
solvers are generally configured to terminate when they reach a specified
nonzero optimality tolerance, or ``MIPGap'', with their running time often
strongly dependent on this tolerance.  An exactly zero tolerance may be very
time-consuming to achieve and in many cases impossible to obtain with a
standard MILP solver due to numerical round-off issues.

\subsection{A progressive-hedging-based algorithm}

Algorithm~\ref{alg:FW-PH-ObjGap} presents an algorithmic template for
solving the Lagrangian relaxation problem~\eqref{lagrel} by embedding an FW
subproblem-solving procedure within the objective-gap inexact progressive
hedging method~\eqref{phmin}-\eqref{php}.  The specifics of the particular
variant of FW employed are left open and marked with asterisks.  The subproblem to be solved for each scenario $i$ and PH iteration $k$ is
\begin{equation} \label{phsubproblem}
\min_{x\in\conv K_i} 
    \left\{h_i(x_{i}) + \inner{w_i^k}{\overline M x_i} 
      + \frac{1}{2} \norm{\overline M x_i - z_i^k}_{\Rho}^2 \right\},
\end{equation}
Define $h_{ik}$ to be the objective function of this subproblem,
whose gradient is given by
\begin{align*}
\nabla h_{ik}(\bar x_i) 
    &= \nabla h_i(\bar x_i) 
          + \overline{M}\transpose w_i^k
          + \overline{M}\transpose \Rho (\overline{M}\bar x_i - z_i^k) \\
&= \nabla h_i(\bar x_i) +
    \left[
    \begin{array}{c}
    w_i^k + \Rho(\overline M \bar x_i - z_i^k) \\ 0
    \end{array}
    \right],
\end{align*}
where the ``$0$'' has dimension $n_T$.  

~

\begin{algorithm}[ht]{}
    \caption{FW-PH-ObjGap Algorithm Template} \label{alg:FW-PH-ObjGap}
    \begin{algorithmic}[1]
        \State \textbf{Initialization} Choose $w^0 \in \mathcal{N}^*$, $z^{0} \in \mathcal{N}$, and a diagonal $\bar m \times \bar m$ matrix $\Rho$
        \For {$k=0,1,2, \ldots$}
          \For {$i \in 1..\ell$}  \label{line:PHscenLoopStart}
             \State * Determine an FW starting point 
                    $\bar x_i^{k0} \in \conv K_i$ \label{line:PHFWstart}
             \Repeat ~\textbf{for} $l = 0, 1, 2, \ldots$
                \State Compute the subproblem gradient 
                       $d_i^{kl} = \nabla h_i(\bar x_i^{kl})
                           + \big( w_i^k + \Rho(\overline M \bar x_i^{kl} - z_i^k),0\big)$
                     \label{line:PHgradient}
                \State * Determine a MIPGap accuracy 
                           $\bar\epsilon_{ikl} < \sigma_{ik}$
                     \label{line:PHsetTol}
                \LongState{Call a MILP solver with a requested
                           absolute objective accuracy of 
                $\bar \epsilon_{ik}$ to find \\ $\text{~~~~~~}
                  \hat x_i^{kl} \in 
                     \epsilon_{ikl}\text{-}\!\argmin_{x_i\in K_i} 
                                               \inner{d_i^{kl}}{x_i}$,
                            where $\epsilon_{ik} \leq \bar \epsilon_{ik}$}
                     \label{line:PHMILPPsolve}
                \State Compute the scenario-$i$ FW gap $\gamma_{ikl} =
                         \inner{d_i^{kl}}{\bar x_i^{kl} - \hat x_i^{kl}} + \epsilon_{ikl}$
                \State * Determine the next FW iterate $\bar x_i^{k,l+1}$ 
                    \label{line:PHnextFW}
        \Until{$\gamma_{ikl} \leq \sigma_{ik}$} \label{line:PHbreakout}
        \State $x_i^k = \bar x_i^{kl}$ 
               $\big[$* or $x_i^k = \bar x_i^{k,l+1}$ 
                  if $h_{ik}(x_i^{k,l+1}) \leq h_{ik}(x_i^{kl})\big]$
                  \label{line:PHnextIterate}
        \EndFor ~(processing scenarios) \label{line:PHscenLoopEnd}
        \State $v^{k+1} = \sproj{W\!}_{\mathcal{N}}\big(Mx^{k+1}\big)$ 
                 \label{line:PHFWStartUpdate}
        \State $z^{k+1} = \nu_k v^{k+1} + (1-\nu_k) z^k$
        \State $w^{k+1} = w^k + \nu_k \Rho (Mx^{k+1} - v^{k+1})$ 
                \label{line:PHFWEndUpdate}
    \EndFor
    \end{algorithmic}
\end{algorithm}

~

As above, $x_i^{k+1} = (x_1^{k+1},\ldots,x_{\ell}^{k+1})$, $\overline M$
denotes removing the last-stage decision variables from a single scenario, and
$M$ denotes collectively dropping the last-stage decision variables from all
scenarios. The details of performing the update operations in
steps~\ref{line:PHFWStartUpdate}-\ref{line:PHFWEndUpdate} are
in~\eqref{phproj}-\eqref{php}.

\paragraph{Requested versus returned MIPGaps.}  The algorithm distinguishes
between the MILP solver objective gap requested, $\bar \epsilon_{ik}$, and the
actual gap reported by the MILP solver, $\epsilon_{ik}$, because it is common
for MILP solvers to return a solution together with an objective gap that is
smaller than was requested.  Since any valid objective gap $\epsilon$ may be
used in Lemma~\ref{lem:fwGap}, the FW gap calculation uses the reported,
possibly smaller gap.

\paragraph{Choosing tolerances}
For $i\in 1..\ell$, the tolerance sequences $\{\sigma_{ik}\}_{k=0}^{\infty}$
referred to in steps~\ref{line:PHsetTol} and~\ref{line:PHbreakout} should be such
that $\sum_{k=0}^{\infty}
\sqrt{\sigma_{ik}} < \infty$, in which case Proposition~\ref{prop:convergePH}
guarantees convergence of the algorithm, assuming that the innermost, FW loop
is able to progressively reduce the FW gaps so that the termination test in
step~\ref{line:PHbreakout} can always be eventually satisfied.  In practice, the
sequences $\sum_{k=0}^{\infty}
\sqrt{\sigma_{ik}} < \infty$ need not be treated as externally specified, but
can be dynamically determined in any way that satisfies the summability
conditions.  One possible way of determining these tolerances, along with the
requested MIPgap accuracies $\bar \epsilon_{ikl}$, is as follows:
\begin{itemize}
\item At iteration $0$, simply terminate the FW procedure for each scenario
$i\in 1..\ell$ after some small fixed number of iterations $l_0 \geq 1$,
performing each MILP solve to a relatively large relative MIPGap, say $1\%$.
This special treatment of the first iteration effectively sets
$\epsilon_{i0l}$ to the absolute gap implied by the relative MIPGap, and
$\sigma_{i0} = \gamma_{i0l_0}$.
\item At each subsequent iteration $k \geq 1$ for scenario $i \in 1..\ell$,
set $\sigma_{ikl} = \beta \sigma_{i0} \alpha^k$ and $\bar\epsilon_{ik} =\mu_{ik}
\sigma_{ikl}$, where $\beta > 0$, $\alpha \in (0,1)$, and $\mu_{ik} \in
(0,1)$.  These choices result in each $\{\sigma_i^k\}_{k=1}^{\infty}$ being
decreasing geometric sequence, making the $\{\sigma_i^k\}$ square-root
summable. One could choose $\beta$ to be considerably larger than $1$ and
$\alpha$ only slightly smaller than $1$, so a very small number of FW
iterations, possibly only one, should often suffice to satisfy the specified
gap.  In addition to the rquested close-to-exact relative MIPGap such as
$10^{-7}$.
\end{itemize}

The goal of such an approach is to perform a very small number of inner FW
steps for each scenario, while still guaranteeing summability of the error
sequences $\big\{ \sqrt{\sigma_{ik}} \big\}_{k=0}^\infty$, $i\in 1..\ell$.
When $\beta \gg 1$, the actual final gaps $\gamma_{ikl}$ attained could be
considerably smaller than than their upper bounds $\sigma_{ik}$; since they
would be upper bounded by square-root-summable sequences, they would
themselves of course be square root summable.  Experimentation will be needed
to determine good values of the $\mu_{ik}$, which might vary dynamically with
$k$ and $i$. Values near $1$ put relatively little pressure on the MILP solver
to produce accurate solutions, but demand more accuracy from the Frank-Wolfe
procedure, whereas values near $0$ demand high accuracy from the MILP but
relatively low accuracy from the FW procedure.

The above procedure for determining the $\sigma_{ikl}$ and
$\bar\epsilon_{ikl}$ is of course just one of many imaginable possibilities.
For example, it might be possible in some situations to determine the
$\sigma_{ikl}$ dynamically: if some MILP subproblems have achieved high
accuracy, one might be able to truncate some other, concurrently running
subproblems to lower accuracy than originally planned.

\paragraph{Which FW iterate to use as the next PH iterate.}
Once the objective gap test on line~\ref{line:PHbreakout} passes at some FW
iteration $l$, line~\ref{line:PHnextIterate} selects the next-to-last FW
iterate $\bar x_i^{kl}$ as the approximate solution to return to the PH
``outer loop'' for scenario $i$, even though a subsequent iterate $\bar
x_i^{k,l+1}$ may have already been computed.  The reason behind this choice is
that the FW gap $\gamma_{ikl}$ bounds the suboptimality of $\bar x_i^{kl}$,
not $\bar x_i^{k,l+1}$.  For some FW variants, it is conceivable that the
subproblem objective of $\bar x_i^{k,l+1}$ could be worse that that of $\bar
x_i^{kl}$, perhaps by enough to fail the FW termination test on
line~\ref{line:PHbreakout}.  Therefore, the safest general prescription is to
use $\bar x_i^{kl}$.  However, if $h_{ik}(x_i^{k,l+1}) \leq h_{ik}(x_i^{kl})$,
that is, FW iterate $l + 1$ has better subproblem objective FW iterate $l$,
then clearly the same gap is also valid for $\bar x_i^{k,l+1}$, so that
solution may be used and may in fact be preferrable. Thus,
line~\ref{line:PHnextIterate} indicates the option of using $\bar x_i^{k,l+1}$
instead of $\bar x_i^{kl}$ when doing so is admissible.

In situations in which computing the next FW iterate or making sure that its
subproblem objective is an improvement are time-consuming tasks, it could be
better to check the condition termination condition $\gamma_{ikl} \leq
\sigma_{ik}$ prior to line~\ref{line:PHnextFW}.  If the termination
condition holds, one would skip line~\ref{line:PHnextFW} and set $x_i^k = \bar
x_i^{kl}$ in line~\ref{line:PHnextIterate}.

\paragraph*{Parallel implementation.}
Parallel implementation of Algorithm~\ref{alg:FW-PH-ObjGap} may follow the
general outlines of existing approaches such as in~\cite{mpisppy23}: the
scenario calculations in the loop in
steps~\ref{line:PHscenLoopStart}-\ref{line:PHscenLoopEnd} are independent of
one another and can be performed concurrently, while the averaging operations
and multiplier updates in
steps~\ref{line:PHFWStartUpdate}-\ref{line:PHFWEndUpdate} can be implemented
using parallel reduction\footnote{See for example~\cite[Sections
4.1-4.3]{grama2003introduction}} and vector operations.  If information on
multiple scenarios is stored in the same memory address space (for example, in
the extreme case of a fully serial implementation), some extent of ``tree''
storage, rather then ``grid'' storage could be used for the $v^k$ and $z^k$,
to avoid storing necessarily identical values in multiple memory locations.

\paragraph{Computing lower bounds.}
Although the method causes $\{w^k\}$ to converge to some $w^* \in
\mathcal{N}^*$ such that $p^* \doteq \Pi w^* \in \mathcal{N}\orthog$ is dual
optimal and would thus yield the Lagrangian dual bound when letting $p = p^*$
in~\eqref{lagdual}, the algorithm does not inherently compute lower bounds on
the optimal objective value as it proceeds.  Clearly, an auxiliary
computation of the form~\eqref{lagdual} could provide a bound, setting $p =
\Pi w^{k_{\last}}$, where $k_{\last}$ is the last iteration before deciding to
terminate the method.  A sequence of Lagrangian bounds could similarly be
obtained as the algorithm runs by periodically solving problems of the
form~\eqref{lagdual} with $p$ set to $p^k \doteq \Pi w^k$, a technique
originally suggested in~\cite{GHRWWW2016}. Another option, used
in~\cite{BCDELL17}, is to choose the starting points $\bar x_i^{k0}$ in
step~\ref{line:PHFWstart} so that
\begin{equation*}
\left[
\begin{array}{c}
\Rho(\bar x_1^{k0} - z_1^k) \\
\vdots \\
\Rho(\bar x_{\ell}^{k0} - z_{\ell}^k)
\end{array}
\right]
\in \mathcal{N}^*
\quad \Rightarrow \quad
\hat w^k \doteq
\left[
\begin{array}{c}
w_1^k + \Rho(\bar x_1^{k0} - z_1^k) \\
\vdots \\
\hat w_{\ell}^k + \Rho(\bar x_{\ell}^{k0} - z_{\ell}^k)
\end{array}
\right]
\in \mathcal{N}^*,
\end{equation*}
since $w^k = (w_1^k,\ldots,w_{\ell}^k) \in \mathcal{N}^*$.  This condition is
equivalent to
\begin{align}
&& 
   (\forall\,N\in\mathcal{U}) \; \big(\forall\,j\in 1..t(N)\big) &&
      \sum_{i\in\mathcal{D}(N)} \!\!\!
         \pi_i \rho_{tj} (\bar x_{it(N)j}^{k0} - z_{it(N)j}^k) &= 0
      \nonumber \\
\Leftrightarrow &&
   (\forall\,N\in\mathcal{U}) \; \big(\forall\,j\in 1..t(N)\big) &&
      \sum_{i\in\mathcal{D}(N)} \!\!\!
         \pi_i \bar x_{it(N)j}^{k0} &= \pi(N) z_{Nj}^k
      \label{btweakintermediate} \\  
\Leftrightarrow && &&
    \sproj{W\!}_{\mathcal{N}}(\bar x^{k0}) &= z^k,
      \label{btweakcompact}
\end{align}
where $z_{Nj}^k$ in~\eqref{btweakintermediate} denotes the common value of the
$z_{it(N)j}^k$ for $i \in \mathcal{D}(N)$ (recalling that $z^k \in \mathcal{N}$)
and $x^{k0} \doteq (x_1^{k0},\ldots,x_{\ell}^{k0})$ in~\eqref{btweakcompact}.
When $\nu_k = 1$, this condition may be met by setting $\bar x^{k0} =
x^k$ at any iteration $k \geq 1$.

Note that $d_i^{k0} = \nabla h_i(\bar x_i^{k0}) + \hat w_i^k$ for all $i\in
1..\ell$. If one enforces~\eqref{btweakcompact}, then in the case that all the
original scenario objective functions are linear, of the form $h_i(x_i) =
\inner{\kappa_i}{\hat x_i}$ for $\kappa_i \in \real^{\bar n}$, one has
$d_i^{k0} = \kappa_i + \hat w_i^k$ for all $i$, and the MILP solves in
step~\ref{line:PHMILPPsolve} collectively compute, across all $i\in 1..\ell$,
a Lagrangian lower bound $\sum_{i=1}^{\ell} \pi_i \big( \inner{d_i^{kl}}{\hat
x_i^{k0}} - \epsilon_{ik0} \big)$.  Without assuming linearity of the original
objective, Lemma~\ref{lem:fwGap} applied to the functions $h_i + \inner{\hat
w_i^{k0}}{\spcdot}$ can be shown to provide a Lagrangian bound
$\sum_{i=1}^{\ell} \pi_i\big( h_i(\bar x_i^{k0}) - \gamma_{ik0} \big)$, which
reduces to the previously mentioned bound in the linear case.

This procedure provides an ongoing sequence of bounds with very little
apparent incremental effort.  However, it does constrain the selection of the
FW starting points in step~\ref{line:PHFWstart} of the algorithm.  Depending on
the FW variant selected, this constraint could impact the performance of the
FW sub-method, so its desirability should not be treated as a forgone
conclusion.  

\paragraph{Deciding when to terminate.}
Algorithm~\ref{alg:FW-PH-ObjGap} does not specify how to terminate the outer
loop (over $k$); any technique applicable to PH may be used for this purpose.
One standard possibility is based on measuring the primal infeasibility
$\smallnorm{Mx^k - z^k}$ and a dual feasibility estimate of how far the
current solution is from minimizing the augmented Lagrangian, which can be
computed as noted in~\cite{Boyd11} from successive ``$z$'' iterates by
$\smallnorm{M\transpose(z^{k+1} - z^k)}$.  The method may be terminated when
both these quantities are assessed to be ``small.''  

In computational settings like~\cite{mpisppy23}, where additional Lagrangian
lower bounds and heuristic upper bounds are computed as the algorithm
progresses, one may also use the difference between such upper and lower
bounds as a termination criterion.  However, this approach will never trigger
termination if the desired tolerance is smaller than the actual duality gap,
which is typically unknown.

\paragraph{Differences from~\cite{BCDELL17}.}
Algorithm~\ref{alg:FW-PH-ObjGap} resembles the FW-PH algorithm proposed
in~\cite[Algorithm 3]{BCDELL17}, but with some important differences.  The
most critical difference is in the convergence analysis: here, convergence is
established by showing that terminating the scenario subproblems with a small
enough FW gap causes the method to behave, from the first iteration, as an
application of the inexact ADMM and DR splitting.  The analysis
in~\cite{BCDELL17}, on the other hand, argues that after some finite number of
iterations, the method will start following the same path as an exact PH
method for the stochastic program~\eqref{lagrel}.  This argument is based on
the algorithm proposed in\cite{BCDELL17} using the fully corrective
(simplicial decomposition) FW method~\cite{Hol74,vHoh77,LJJ15}, keeping full
lists $V_{ik}$ of all the vertices of $K_i$ encountered through iteration $k$
for each scenario $i\in 1..\ell$.  To find the next FW iterate in the
equivalent of step~\ref{line:PHnextFW} in Algorithm~\ref{alg:FW-PH-ObjGap},
the fully corrective FW method optimizes the subproblem objective $h_{ik}$
over $\conv V_{ik}$.  Since the $K_i$ are finite, the $V_{ik} \subseteq K_i$
must eventually stabilize, so that for some $k^* \geq 1$, $V_{ik} = V_{ik'}$
for all $i\in 1..\ell$ and $k,k' \geq k^*$.  If the LMO solves and the
solutions of these auxiliary problems are exact, the resulting subproblem
solutions then become exact over convex hull of the full set of vertices
$\conv K_i$: if they were not, then new vertices would eventually be generated
and stabilization would not have occurred.  As a result, after iteration
$k^*$, the method follows the same path as exact progressive hedging
initialized from $w^{k^*}$ and $z^{k^*}$, and convergence follows from the
analysis of exact PH as in~\cite{RW91}.

Unfortunately, there is no practical way to definitively determine when
stabilization of the vertex sets $V_{ik}$ has occurred (for example,
stabilization over two successive iterations does not necessarily imply
long-term stabilization).  In the analysis of~\cite{BCDELL17}, the only formal
purpose of iterations $k < k^*$ before stabilization occurs is to discover the
necessary vertices.  This apparent ``wandering'' phase in the analysis leaves
open the theoretical possibility that the method might not produce much useful
information until stabilization, although the computational results
in~\cite{BCDELL17} suggest otherwise.  By contrast, the analysis here instead
shows that, so long as the FW gaps decrease sufficiently quickly, the method
functions as a special case of the inexact ADMM in
Proposition~\ref{prop:convergeADMM}, which is a modest generalization of the
inexact ADMM known since~\cite{EckBer92}.  This phenomon may explain the
encouraging computational results in~\cite{BCDELL17} and other applications of
the same ideas.

Since the analysis here requires only controlling the objective gap for each
subproblem, as opposed to reaching full vertex stabilization, the template in
Algorithm~\ref{alg:FW-PH-ObjGap} can be adapted to use essentially any variant
of the FW method, with no need to store sets of encountered vertices $V_{ik}$
or solve auxiliary nonlinear problems over their convex hulls. This freedom
opens many possibilities for simplifying and improving the method, although
computational experimentation will be necessary to identify the most efficient
approaches.

An additional benefit of the approach here is that the MILP solves used by the
FW method need not be exact, which is theoretically assumed in~\cite{BCDELL17}
(although likely not the case in its experimental work).  In
Algorithm~\ref{alg:FW-PH-ObjGap}, one may explicitly use positive MIPGaps,
adjusting them as the algorithm progresses.  Early on, one could use
relatively large MIPGaps since the scaled Lagrange multiplier estimates $w^k$
are likely inaccurate, and it is not worth expending the computing time needed
to find extremely accurate MILP solutions.

The main disadvantage of Algorithm~\ref{alg:FW-PH-ObjGap} is that, depending
on the method for determining the tolerances $\sigma_{ik}$, one may not simply
be able to truncate the solution of a subproblem after some fixed number of FW
iterations $l_{\max}$ (called $t_{\max}$ in~\cite{BCDELL17}, which uses the
symbol $t$ to index FW iteration).  This kind of truncation is permitted
in~\cite{BCDELL17}, although it introduces significant complication in the
analysis and in the case of $l_{\max} = 1$ requires technical recourse
assumptions on the problem instance to guarantee convergence.  However, the
hope here is that since summability the square roots of the objective
tolerances $\{\sigma_{ik}\}_{k=0}^{\infty}$ can be made weak requirement in
practice, the occasional added FW inner-loop iterations that might be imposed
will be a small price to pay for the ``non-wandering'' convergence theory and
freedom in choosing FW variants provided by methods following the template in
Algorithm~\ref{alg:FW-PH-ObjGap}.  The methods proposed here also use FW gaps
adaptively, setting the number of FW iterations in response to the current
subproblem accuracy, rather than truncating after a fixed number of FW steps
in most iterations.

\subsection{Alternatives based on inexact augmented Lagrangian methods}
\label{sec:stochProgFWALM}
This subsection presents a method for the same application as
Algorithm~\ref{alg:FW-PH-ObjGap}, but based on the augmented Lagrangian
method~\eqref{spalmmin}-\eqref{spalmupdate} instead of the PH
method~\eqref{phmin}-\eqref{php}. Ordinarily, such approaches are impractical
because the subproblem objective in~\eqref{spalmmin} is not separable, so
minimizing it could be essentially as hard as optimizing the extensive form of
the original stochastic programming problem.  However, the situation is
different when employing a Frank-Wolfe subproblem solver.  Consider applying a
Frank-Wolfe algorithm to solve~\eqref{spalmmin}, with the $f_i$ defined as
in~\eqref{convexifiedfi}.  For this form of the $f_i$, one may
rewrite the ALM subproblem~\eqref{spalmmin} as
\begin{equation} \label{fwalm1}
x^{k+1} \in \approxmin{\delta_k}
                      {\substack{x_1\in \conv K_1 \\ 
                                 \svdots \\ 
                                 x_{\ell} \in \conv K_{\ell}}} \!\!
      \left\{
         \sum_{i=1}^{\ell} \pi_i \left(
           h_i(x_{i}) + 
           \sum_{t=1}^T
           \sum_{j=1}^{n_T}
             w_{itj}^k x_{itj} 
               + \frac{c_k \rho_{tj}}{2} 
                   \left(
                    x_{itj} 
                      - a_{N(i,t)j}(x)
                   \right)^2 \right)
      \right\}.
\end{equation}
This problem's feasible set is a Cartesian product, but its objective is not
separable over the components of the Cartesian product.  FW methods can induce
decomposition over such structures.

Within iteration $k$ of the augmented Lagrangian
method~\eqref{spalmmin}-\eqref{spalmupdate}, suppose the algorithm is at
Frank-Wolfe iteration $l$ and some trial solution $x^{kl} \in \real^n$.
The Frank-Wolfe method would then determine the gradient $y^{kl} =
(y^{kl}_1,\ldots,y^{kl}_{\ell})$ of the differentiable function
in~\eqref{fwalm1} and then solve the linear subproblem
\begin{equation*}
x^{k,l+1} \in \approxmin{\sigma_{kl}}
                      {\substack{x_1\in K_1 \\ \svdots \\ x_{\ell} \in K_{\ell}}}
      \Big\{
        \biginner{(y^{kl}_1,\ldots,y^{kl}_{\ell})}{(x_1,\ldots,x_{\ell})}
      \Big\}
      =
      \approxmin{\sigma_{kl}}
                      {\substack{x_1\in K_1 \\ \svdots \\ x_{\ell} \in K_{\ell}}}
        \left\{
            \sum_{i=1}^{\ell} \inner{y^{kl}_i}{x_i}
         \right\},
\end{equation*}
where $\sigma_{kl} \in (0,\delta_k)$ is some objective tolerance for the LMO
solution.  Due to the linear and hence fully separable nature of the objective
in the FW subproblem, this problem may be solved by performing $\ell$ independent
approximate minimizations of the form
\begin{equation*}
x_i^{k,l+1} \in \approxmin{\sigma_{ikl}}{x_i \in K_i}
      \Big\{
        \inner{y^{kl}_i}{x_i}
      \Big\},
\end{equation*}
where $\sigma_{1kl},\ldots,\sigma_{\ell k l} > 0$ are sub-tolerances such that
$\sum_{i=1}^{\ell} \sigma_{ikl} = \sigma_{kl}$.  Thus, the inseparability
of~\eqref{fwalm1} does not impede decomposition if the subproblem solution
method only uses linearizations of its objective function. In the Frank-Wolfe
context, it is thus possible to take advantage of the block structure of the
problem from an augmented Lagrangian method, and not only from an ADMM method.

Implementing such an approach requires being able to compute the
gradient of the function in~\eqref{sec:stochProgFWALM}:

\begin{lemma} \label{lem:stochGradient}
For any $k \geq 0$, let $h_k : \real^n \to \real$ denote the function
in~\eqref{fwalm1}, that is,
\begin{equation*}
h_k(x) = \sum_{i=1}^{\ell} \pi_i \left(
           h_i(x_{i}) + 
           \sum_{t=1}^T
           \sum_{j=1}^{n_T}
             w_{itj}^k x_{itj} 
               + \frac{c_k \rho_{tj}}{2} 
                   \left(
                    x_{itj} 
                      - a_{N(i,t)j}(x)
                   \right)^2 \right).
\end{equation*}
Then $\nabla h_k(x) = (\pi_1 d_1,\ldots,\pi_{\ell} d_{\ell})$, where
\begin{align} \label{almstochgrad}
z &\doteq \sproj{W\!}_{\mathcal{N}}(Mx) &
(\forall\, i\in 1..\ell) \;\; d_i &\doteq 
      \nabla h_i(x)
      + \left[
           \begin{array}{c}
           w_i^k + c_k \Rho(\overline M x_i - z_i) \\
           0
           \end{array}
        \right],
\end{align}
the ``$0$'' vector having dimension $n_T$ and the linear operators $\overline
M$ and $M$ respectively denoting dropping the last-stage elements from
single-scenario decision vector or the entire decision vector, as above.
\end{lemma}
\begin{proof}
To avoid unnecessary complexity, one may apply Lemma~\ref{lem:quadGradient} for
$V=\mathcal{N}$ and the specific choices of $M$ and $V$ starting in
Section~\ref{sec:stochProg}.  Observe that $h_k$ may be more compactly
expressed as
\begin{align}
h_k(x) &= \sum_{i=1}^{\ell} \Big( 
                             \pi_i h_i(x_i) + \pi_i \inner{w_i^k}{\overline Mx_i} 
                           \Big)
              + \frac{c_k}{2}\big( \sdist{W\!}_{\mathcal{N}}(Mx)\big)^2 \nonumber \\
       &= \sum_{i=1}^{\ell} \Big( 
                             \pi_i h_i(x_i) 
                                + \pi_i \biginner{{\overline M}\transpose w_i^k}{x_i} 
                           \Big)
              + \frac{c_k}{2}\big( \sdist{W\!}_{\mathcal{N}}(Mx)\big)^2. 
              \label{stochquad1}
\end{align}
Lemma~\ref{lem:quadGradient} asserts that the gradient of the last term in
this expression is
\begin{align*}
r \doteq \nabla \left[ \frac{c_k}{2}\Big( \sdist{W\!}_{\mathcal{N}}(Mx) \Big)^2 \right]
&= c_k M\transpose W \big((\identity - \sproj{W\!}_{\mathcal{N}})(Mx) \big) \\
&= c_k M\transpose W \big(Mx - \sproj{W\!}_{\mathcal{N}}(Mx) \big) \\
&= c_k M\transpose W (Mx - z),
\end{align*}
where $z \doteq \sproj{W}_{\mathcal{N}}(Mx)$. Applying the indexing
conventions for the other vectors to $r$ and using the structure of $M$ and
$W$ then produces
\begin{equation*}
(\forall\,i\in 1..\ell) \qquad
r_i = c_k {\overline M}\transpose\pi_i\Rho(\overline M x_i - z_i)
    = \pi_i \left[
            \begin{array}{c}
            c_k \Rho(\overline M x_i - z_i) \\
            0
            \end{array}
        \right]
\end{equation*}
Taking the derivatives of the other terms~\eqref{stochquad1}, 
it then follows that
\begin{equation*}
(\forall\,i\in 1..\ell) \quad
\nabla_{\!\!x_i} h_k(x) = \pi_i \nabla h_i(x_i) + \pi_i {\overline M}\transpose w_i^k + r_i 
                    = \pi_i d_i, \tag*{\qedhere}
\end{equation*}
\end{proof}

Using Lemma~\ref{lem:stochGradient} to help compute the gradient,
Algorithm~\ref{alg:FW-ALPH-ObjGap} presents an ALM-based alternative to
Algorithm~\ref{alg:FW-PH-ObjGap}. As with Algorithm~\ref{alg:FW-PH-ObjGap},
algorithm components that depend on the specific choice of FW variant are
marked with asterisks. The details of the projection operations in
lines~\ref{line:ALFWproj} and~\ref{line:ALFWUpdate} may be found
in~\eqref{averageop} and~\eqref{spalmupdate}.  Conditions for terminating the
outer ($k$) loop are omitted; these considerations are similar to
Algorithm~\ref{alg:FW-PH-ObjGap}.

\begin{algorithm}[t]{}
    \caption{FW-ALPH-ObjGap Algorithm Template} \label{alg:FW-ALPH-ObjGap}
    \begin{algorithmic}[1]
        \State \textbf{Initialization} Choose $w^0 \in \mathcal{N}^*$
               and a diagonal $\bar m \times \bar m$ matrix $\Rho$
        \For {$k=0,1,2, \ldots$}
           \State * Determine an FW starting point 
                     $\bar x^{k0} \in \bigtimes_{i=1}^{\ell} \conv K_i$ 
                     \label{line:ALFWstart}
           \Repeat ~\textbf{for} $l = 0, 1, 2, \ldots$
              \State $z^{kl} = \sproj{W\!}_{\mathcal{N}}\big(M \bar x^{kl}\big)$
                     \label{line:ALFWproj}
              \For {$i \in 1..\ell$}  
                     \label{line:ALscenLoopStart}
                 \State Compute the subproblem gradient 
                       $d_i^{kl} = \nabla h_i(\bar x_i^{kl})
                           + \big( w_i^k 
                                   + \Rho(c_k \overline M \bar x_i^{kl} - z_i^{kl}),0\big)$
                     \label{line:ALgradient}
                \State * Determine a MIPGap accuracy 
                    $\bar \epsilon_{ikl} < \delta_{k}$
                     \label{line:ALsetTol} 
                \LongState{Call a MILP to solver a requested absolute 
                           objective accuracy of 
                $\bar \epsilon_{ik}$ to find \\ $\text{~~~~~~}
                  \hat x_i^{kl} \in 
                     \epsilon_{ikl}\text{-}\!\argmin_{x_i\in K_i} 
                                               \inner{d_i^{kl}}{x_i}$,
                            where $\epsilon_{ik} \leq \bar \epsilon_{ik}$}
                     \label{line:ALMILPPsolve}
                \State Compute the scenario FW gap $\gamma_{ikl} =
                         \inner{d_i^{kl}}{\bar x_i^{kl} - \hat x_i^{kl}} + \epsilon_{ikl}$
              \EndFor ~(processing scenarios) \label{line:ALscenLoopEnd}
              \State Compute the overall FW gap
                     $\gamma_{kl} = \sum_{i=1}^{\ell} \pi_i \gamma_{ikl}$
              \State * Determine the next FW iterate $\bar x^{k,l+1}$ 
                     \label{line:ALnextFW}
           \Until{$\gamma_{kl} \leq \delta_{k}$} \label{line:ALbreakout} 
        \State $x^{k+1} = \bar x^{kl}$ 
               $\big[$* or $x^k = \bar x^{k,l+1}$ 
                  if $h_{k}(x^{k,l+1}) \leq h_{k}(x^{kl})\big]$ 
                     \label{line:ALnextIter}
         \State $w^{k+1} = w^k + \nu_k c_k \Rho \Big(Mx^{k+1} 
                               - \big(\sproj{W\!}_{\mathcal{N}}(M x^{k+1})\big)\Big)$ 
                \label{line:ALFWUpdate}
    \EndFor
    \end{algorithmic}
\end{algorithm}

Proposition~\ref{prop:absALMapproxObj} guarantees convergence of
Algorithm~\ref{alg:FW-ALPH-ObjGap} when $\sum_{k=0}^{\infty} \sqrt{c_k
\delta_k} < \infty$, assuming that the Frank-Wolfe method and subproblem
tolerances $\bar\epsilon_{ikl}$ are configured such that the Frank-Wolfe ($l$)
loop is eventually able to satisfy the stopping condition on
line~\ref{line:ALbreakout}.

Algorithm~\ref{alg:FW-ALPH-ObjGap} resembles Algorithm~\ref{alg:FW-PH-ObjGap},
consisting many of the same elements, but organized somewhat differently. Most
notably, the nesting of the scenario ($i$) and Frank-Wolfe ($l$) loops is
reversed: in the augmented Lagrangian approach of
Algorithm~\ref{alg:FW-ALPH-ObjGap}, one loops over scenarios within each FW
iteration, as opposed to looping over FW iterations within each scenario as in
Algorithm~\ref{alg:FW-PH-ObjGap}.  The algorithms are similar enough that a
framework implementing one of them should be relatively easy to extend to 
implementing the other.

Determining which class of algorithms will ultimately be more efficient in
practice is likely to require extensive experimentation, along with evaluation
of numerous different options for implementing the undetermined parts of each
template.  Augmented Lagrangian methods have a reputation for converging
faster than ADMM methods, but if their subproblems are computed sufficiently
inexactly, they are sometimes slower in terms of overall computational effort.
Some points worth considering are as follows:

\paragraph{Setting tolerances.}  As written, the tolerance sequence
$\{\delta_k\}$ is treated as given, but a valid implementation could set
$\delta_k$ dynamically so long as the implementation assures that
$\sum_{k=1}^{\infty} \sqrt{c_k \delta_k} < \infty$ would hold if the method were
to run indefinitely.  
One could employ a strategy similar to that suggested for
Algorithm~\ref{alg:FW-PH-ObjGap}, but accounting for possibly varying $c_k$:
in outer iteration $0$, one could simply run FW for some fixed number of
iterations $\ell_0$, set $\delta_0 = \beta \sum_{i=1}^{\ell} \pi_i
\gamma_{i0\ell_0}$ with $\beta \geq 1$ so that the $k=0$ results are
effectively accepted after $l_0$ FW iterations, and subsequently take
$\delta_k = \alpha^k \delta_0 c_0 / c_k$ for some $\alpha \in (0,1)$ (but
likely close to $1$).  Then $c_k \delta_k = c_0 \delta_0 \alpha^k$ for all
$k\geq 0$, meaning that $\{c_k \delta_k\}$ is a decreasing geometric sequence
and hence square-root summable. Of course, there are many other possibilities.
To increase the chance of satisfying the FW termination criterion in
line~\ref{line:ALbreakout}, it is likely preferable to set the tolerances
$\bar \epsilon_{ikl}$ chosen in line~\ref{line:ALsetTol} to satisfy
$\sum_{i=1}^{\ell} \pi_i \bar\epsilon_{ikl} < \delta_k$.

\paragraph{Choice of FW variant.}  Like Algorithm~\ref{alg:FW-PH-ObjGap},
Algorithm~\ref{alg:FW-ALPH-ObjGap} may be configued to use many different
variants of the Frank-Wolfe algorithm.  The preferred FW variants for the two
algorithm classes seem likely to differ:  in particular, fully corrective
methods appear far less attractive in the augmented Lagrangian setting because
each auxiliary continuous nonlinear problem needed to find the next FW iterate
would have   dimension $n = \ell \bar n$, presenting a greater challenge than
solving $\ell$ independent problems of dimension $\bar n$ as in the ADMM
approach.\footnote{In principle, one could compress these auxiliary problems
to the dimension of the number of collected vertices, but doing so tends to
lead to numerical difficulties.}  The augmented Lagrangian setting may benefit
from FW methods that are specifically designed to operate over
Cartesian-product domains; some example resources that may prove helpful in
investigating this possibility include~\cite{BRZ25,BHPW25}.

\paragraph{Coordination of FW solves.}  In the augmented Lagrangian approach,
each LMO invocation by the Frank-Wolfe method involves $\ell$ MILP solves, one
for each scenario.  In the ADMM version, a separate FW algorithm runs for each
subproblem, so some scenarios may take more FW steps than others within a
given outer iteration $k$.  Thus, the ADMM approach affords more flexibility,
but on the hand other each scenario $i$'s information about the other
scenarios remains ``frozen'' until the next outer iteration $k$.  In the
augmented Lagrangian approach, scenarios share information (through the
calculation of $z^{kl}$) as the FW method proceeds. A possible middle ground
between these alternatives could be to adapt the ALM approach to use a
``block-iterative'' FW variant such as described in~\cite{BHPW25}.  This
combination might conceivably yield an ALM-based method more flexible than
presented in Algorithm~\ref{alg:FW-ALPH-ObjGap}, with varying numbers of LMO
calls per scenario within each augmented Lagrangian iteration $k$, but with
information still dynamically shared between scenarios within the inner,
Frank-Wolfe layer of the algorithm.

\paragraph{Which FW iterate to use at the next ALM iterate.}
Line~\ref{line:ALnextIter} presents a similar situation to
line~\ref{line:PHnextIterate} of Algorithm~\ref{alg:FW-PH-ObjGap}, but in
aggregate over all scenarios instead of individually by scenario.  The
aggregate FW gap $\gamma_{kl}$ evaluates the subproblem suboptimality of $\bar
x^{kl}$; absent further information, that point is therefore the only safe
choice to use for the next augmented Lagrangian iterate.  However, if the next
FW iterate $\bar x^{k,l+1}$ is known not to have a worse objective value, then
it may also be used.  As with Algorithm~\ref{alg:FW-PH-ObjGap}, there could be
classes of situations in which it would be better to check for FW-loop
termination before computing the next FW iterate, and skip that calculation if
the test passes.  When $\bar x^{kl}$ is used, then the
$\sproj{W}_{\mathcal{N}}(Mx^{k+1})$ term on line~\ref{line:ALFWUpdate} is
identical to the $z^{kl}$ last computed on line~\ref{line:ALFWproj}, but if
$\bar x^{k,l+1}$ is used then an new projection must be performed before
executing the multiplier update.

\paragraph{Ongoing generation of lower bounds.}  The optimal value of the
augmented Lagrangian subproblem provides a lower bound on the problem optimal
value.  An inexact solution of the augmented Lagrangian, combined with a
Frank-Wolfe gap, should provide a similar bound.  Thus, the augmented
Lagrangian approach readily provides a lower approximation of the Lagrangian
relaxation bound with every Frank-Wolve iteration $\ell$.  By comparison, the
ADMM approach only provides one such bound per outer iteration $k$, and only
if the starting point of the FW method is suitably constrained, as
in~\cite{BCDELL17}.

\paragraph{Varying scalar penalty parameters.}  Augmented Lagrangian
methods theoretically allow continually varying the scalar penalty parameter
$c_k$, so long as it remains bounded away from zero.  In theory, without
burdensome additional assumptions, the ADMM requires a constant $c_k$.  By
applying a constant scaling to the change-of-variables diagonal matrix $\Rho$,
one may simply take $c_k \equiv 1$ in the ADMM.  In practice, however, some
degree of scaling and penalty adjustment is of course often used in ADMM-based
methods.

\section*{Future computational experiments}
Computational work exploring the empirical properties of the algorithms
proposed here are underway and will be included in future revisions of this
work, with added authors.

\bibliographystyle{abbrv}
\bibliography{prox,general}

@article {Hol74,
    AUTHOR = {Holloway, Charles A.},
     TITLE = {An extension of the {F}rank and {W}olfe method of feasible
              directions},
   JOURNAL = {Math. Program.},
  FJOURNAL = {Mathematical Programming},
    VOLUME = {6},
      YEAR = {1974},
     PAGES = {14--27}
}

@article {vHoh77,
    AUTHOR = {von Hohenbalken, Balder},
     TITLE = {Simplicial decomposition in nonlinear programming algorithms},
   JOURNAL = {Math. Program.},
  FJOURNAL = {Mathematical Programming},
    VOLUME = {13},
      YEAR = {1977},
    NUMBER = {1},
     PAGES = {49--68}
}

@article {FW56,
    AUTHOR = {Frank, Marguerite and Wolfe, Philip},
     TITLE = {An algorithm for quadratic programming},
   JOURNAL = {Naval Res. Logist. Quart.},
  FJOURNAL = {Naval Research Logistics Quarterly},
    VOLUME = {3},
      YEAR = {1956},
     PAGES = {95--110}
}

@inproceedings{LJJ15,
 author = {Lacoste-Julien, Simon and Jaggi, Martin},
 booktitle = {Advances in Neural Information Processing Systems},
 editor = {C. Cortes and N. Lawrence and D. Lee and M. Sugiyama and R. Garnett},
 publisher = {Curran Associates, Inc.},
 title = {On the Global Linear Convergence of {Frank-Wolfe} Optimization Variants},
 volume = {28},
 year = {2015}
}

@book{FWBook25,
author = {Braun, Gábor and Carderera, Alejandro and Combettes, Cyrille W. and Hassani, Hamed and Karbasi, Amin and Mokhtari, Aryan and Pokutta, Sebastian},
title = {Conditional Gradient Methods},
publisher = {Society for Industrial and Applied Mathematics},
year = {2025},
doi = {10.1137/1.9781611978568},
address = {Philadelphia, PA}
}

@article {mpisppy23,
    AUTHOR = {Knueven, Bernard and Mildebrath, David and Muir, Christopher
              and Siirola, John D. and Watson, Jean-Paul and Woodruff, David
              L.},
     TITLE = {A parallel hub-and-spoke system for large-scale scenario-based
              optimization under uncertainty},
   JOURNAL = {Math. Program. Comput.},
  FJOURNAL = {Mathematical Programming Computation},
    VOLUME = {15},
      YEAR = {2023},
    NUMBER = {4},
     PAGES = {591--619}
}

@article {BRZ25,
    AUTHOR = {Bomze, Immanuel and Rinaldi, Francesco and Zeffiro, Damiano},
     TITLE = {Projection free methods on product domains},
   JOURNAL = {Comput. Optim. Appl.},
  FJOURNAL = {Computational Optimization and Applications. An International
              Journal},
    VOLUME = {91},
      YEAR = {2025},
    NUMBER = {2},
     PAGES = {511--540}
}

@techreport{BHPW25,
      title={Flexible block-iterative analysis for the {Frank-Wolfe} algorithm}, 
      author={G\'{a}bor Braun and Jannis Halbey and Sebastian Pokutta and Zev Woodstock},
      year={2025},
      number={2409.06931},
      institution={arXiv} 
}

@book{grama2003introduction,
  author    = {Grama, Ananth and Gupta, Anshul and Karypis, George and Kumar, Vipin},
  title     = {Introduction to Parallel Computing},
  edition   = {2nd},
  year      = {2003},
  publisher = {Addison-Wesley},
  address   = {Harlow, UK}
}

@Article{Eck94c,
  author = 	 {Eckstein, J.},
  title = 	 {Some saddle-function splitting methods 
                  for convex programming},
  journal = 	 {Optim. Meth. Software},
  fjournal =     {Optimization Methods and Software},
  year = 	 {1994},
  volume = 	 {4},
  number = 	 {1},
  pages = 	 {75--83},
}

@article {EckBer92,
    AUTHOR = {Eckstein, Jonathan and Bertsekas, Dimitri P.},
     TITLE = {On the {D}ouglas-{R}achford splitting method and the proximal
              point algorithm for maximal monotone operators},
   JOURNAL = {Math. Program.},
  FJOURNAL = {Mathematical Programming},
    VOLUME = {55},
      YEAR = {1992},
    NUMBER = {3},
     PAGES = {293--318},
}

@InCollection{Gab83,
  author = 	 {Gabay, Daniel},
  title = 	 {Applications of the method of multipliers to variational 
                  inequalities},
  SERIES    = {Studies in Mathematics and its Applications},
  VOLUME    = {15},
  BOOKTITLE = {Augmented {L}agrangian methods: {A}pplications to the numerical solution of boundary-value problems},
  editor    = {Fortin, Michel and Glowinski, Roland},
  YEAR      = {1983},
  PUBLISHER = {North-Holland},
  ADDRESS   = {Amsterdam},
  pages =    {299--340},
}

@article {LioMer79,
    AUTHOR = {Lions, Pierre-Louis and Mercier, Bertrand},
     TITLE = {Splitting algorithms for the sum of two nonlinear operators},
   JOURNAL = {SIAM J. Numer. Anal.},
  FJOURNAL = {SIAM Journal on Numerical Analysis},
    VOLUME = {16},
      YEAR = {1979},
    NUMBER = {6},
     PAGES = {964--979},
}

@article {Min62,
    AUTHOR = {Minty, George J.},
     TITLE = {Monotone (nonlinear) operators in {H}ilbert space},
   JOURNAL = {Duke Math. J.},
  FJOURNAL = {Duke Mathematical Journal},
    VOLUME = {29},
      YEAR = {1962},
     PAGES = {341--346}
}

@article {Roc76a,
    AUTHOR = {Rockafellar, R. Tyrrell},
     TITLE = {Monotone operators and the proximal point algorithm},
   JOURNAL = {SIAM J. Control Optim.},
  FJOURNAL = {SIAM Journal on Control and Optimizaton},
    VOLUME = {14},
      YEAR = {1976},
    NUMBER = {5},
     PAGES = {877--898},
}

@article {Roc76b,
    AUTHOR = {Rockafellar, R. Tyrrell},
     TITLE = {Augmented {L}agrangians and applications of the proximal point
              algorithm in convex programming},
   JOURNAL = {Math. Oper. Res.},
  FJOURNAL = {Mathematics of Operations Research},
    VOLUME = {1},
      YEAR = {1976},
    NUMBER = {2},
     PAGES = {97--116}
}

@article {RocFenchDual,
    AUTHOR = {Rockafellar, R. Tyrrell},
     TITLE = {Duality and Stability in Extremum Problems Involving Convex Functions},
   JOURNAL = {Pacific J. Math.},
  FJOURNAL = {Pacific Journal of Mathematics},
    VOLUME = {21},
      YEAR = {1967},
     PAGES = {167--187}
}

@article {DouRac56,
    AUTHOR = {Douglas, Jr., Jim and Rachford, Jr., Henry H.},
     TITLE = {On the numerical solution of heat conduction problems in two
              and three space variables},
   JOURNAL = {Trans. Amer. Math. Soc.},
  FJOURNAL = {Transactions of the American Mathematical Society},
    VOLUME = {82},
      YEAR = {1956},
     PAGES = {421--439},
      ISSN = {0002-9947},
   MRCLASS = {65.3X},
  MRNUMBER = {MR0084194 (18,827f)},
MRREVIEWER = {C. Saltzer},
}

@preamble{
   "\def\cprime{$'$} "
}

@book{BauComBook,
    AUTHOR = {Bauschke, Heinz H. and Combettes, Patrick L.},
     TITLE = {Convex Analysis and Monotone Operator Theory in {H}ilbert
              Spaces},
 PUBLISHER = {Springer},
   ADDRESS = {New York},
   EDITION = {Second},
      YEAR = {2017}
   }

@book {Roc70book,
    AUTHOR = {Rockafellar, R. Tyrrell},
     TITLE = {Convex Analysis},
    SERIES = {Princeton Mathematical Series, No. 28},
 PUBLISHER = {Princeton University Press},
   ADDRESS = {Princeton, N.J.},
      YEAR = {1970}
}

@article {RW91,
    AUTHOR = {Rockafellar, R. Tyrrell and Wets, Roger J.-B.},
     TITLE = {Scenarios and policy aggregation in optimization under
              uncertainty},
   JOURNAL = {Math. Oper. Res.},
  FJOURNAL = {Mathematics of Operations Research},
    VOLUME = {16},
      YEAR = {1991},
    NUMBER = {1},
     PAGES = {119--147}
}

@Article{GHRWWW2016,
author="Gade, Dinakar and Hackebeil, Gabriel and Ryan, Sarah M. and Watson, Jean-Paul
      and Wets, Roger J.-B. and Woodruff, David L.",
title="Obtaining lower bounds from the progressive hedging algorithm 
       for stochastic mixed-integer programs",
fjournal="Mathematical Programming",
journal="Math. Program.",
year="2016",
volume="157",
number="1",
pages="47--67"
}

@article {BCDELL17,
    AUTHOR = {Boland, Natashia and Christiansen, Jeffrey and Dandurand,
              Brian and Eberhard, Andrew and Linderoth, Jeff and Luedtke,
              James and Oliveira, Fabricio},
     TITLE = {Combining progressive hedging with a {F}rank-{W}olfe method to
              compute {L}agrangian dual bounds in stochastic mixed-integer
              programming},
   JOURNAL = {SIAM J. Optim.},
  FJOURNAL = {SIAM Journal on Optimization},
    VOLUME = {28},
      YEAR = {2018},
    NUMBER = {2},
     PAGES = {1312--1336}
}

@article{Boyd11,
 author = {Boyd, Stephen and Parikh, Neal and Chu, Eric and Peleato, Borja and Eckstein, Jonathan},
 title = {Distributed Optimization and Statistical Learning via the Alternating Direction Method of Multipliers},
 fjourna = {Foundations and Trends in Machine Learning},
 journal = {Found. Trends Mach. Learn.},
 volume = {3},
 number = {1},
 year = {2011}
}

@book{BerConvex,
      title     = "Convex Analysis and Optimization",
      author    = "Dimitri P. Bertsekas and Angelia Nedi\'{c} and Asuman E. Ozdaglar",
      year      = 2003,
      publisher = "Athena Scientific",
      address   = "Belmont, MA, USA"
    }

@techreport{FenchNotes,
  author      = "Fenchel, Werner",
  title       = "Convex Cones, Sets, and Functions",
  institution = "Princeton University, Department of Mathematics",
  year        = "1955",
  type        = "Lecure notes",
  address     = "Princeton, NJ"
}

@phdthesis{RockThesis,
  author       = {R. Tyrrell Rockafellar}, 
  title        = {Convex Functions and Dual Extremum Problems},
  school       = {Harvard University, Deparment of Mathematics},
  year         = {1963},
  address      = {Cambridge, MA}
}

@book {Roc74,
    AUTHOR = {Rockafellar, R. Tyrrell},
     TITLE = {Conjugate Duality and Optimization},
    SERIES = {Conference Board of the Mathematical Sciences Regional
              Conference Series in Applied Mathematics, No. 16},
  XXXXNOTE = {Lectures given at the Johns Hopkins University, Baltimore,
              Md., June, 1973},
 PUBLISHER = {Society for Industrial and Applied Mathematics (SIAM)},
   ADDRESS = {Philadelphia, PA},
      YEAR = {1974}
}

@article{Rock69,
author = {Rockafellar, R. Tyrrell},
title = {Local boundedness of nonlinear, monotone operators},
volume = {16},
journal = {Michigan Math. J.},
fjournal = {Michigan Mathematical Journal},
number = {4},
pages = {397--407},
year = {1969}
}

\end{document}